\documentclass[11pt]{amsart}
\usepackage[dvipsnames]{xcolor}
\usepackage{tikz-cd}

\usepackage{amsmath,mathdots}
\usepackage{amssymb}
\usepackage{mathrsfs}

\newtheorem{theorem}{Theorem}[section]
\newtheorem{claim}[theorem]{Claim}

\newtheorem{lemma}[theorem]{Lemma}

\newtheorem{proposition}[theorem]{Proposition}
\newtheorem{corollary}[theorem]{Corollary}

\theoremstyle{definition}
\newtheorem{definition}[theorem]{Definition}

\newtheorem{question}[theorem]{Question}

\theoremstyle{remark}
\newtheorem{remark}[theorem]{Remark}

\newtheorem{notation}[theorem]{Notation}
\newtheorem{fact}[theorem]{Fact}

\def\l{{\langle}}
\def\r{{\rangle}}

\newcount\skewfactor
\def\mathunderaccent#1#2 {\let\theaccent#1\skewfactor#2
\mathpalette\putaccentunder}
\def\putaccentunder#1#2{\oalign{$#1#2$\crcr\hidewidth
\vbox to.2ex{\hbox{$#1\skew\skewfactor\theaccent{}$}\vss}\hidewidth}}

\def\smallbox#1{\leavevmode\thinspace\hbox{\vrule\vtop{\vbox
   {\hrule\kern1pt\hbox{\vphantom{\tt/}\thinspace{\tt#1}\thinspace}}
   \kern1pt\hrule}\vrule}\thinspace}

\DeclareMathOperator{\dom}{dom}
\DeclareMathOperator{\supp}{Supp}

\DeclareMathOperator{\rng}{rng}
\newcommand{\cf}{{\rm cf}}

\title{Stationary reflection and the failure of SCH at $\aleph_{\omega_1}$}
\date{\today}
\author{Tom Benhamou}
\address{Rutgers University, New Brunswick, NJ, USA}
\email{tom.benhamou@rutgers.edu}
\thanks{The research of the first author was supported by the National Science Foundation under Grant
No. DMS-2346680}

\author{Dima Sinapova}
\thanks{The research of the second author was supported by the National Science Foundation under Grant
No. DMS-2246781}
\address{Rutgers University, New Brunswick, NJ, USA}
\email{dima.sinapova@rutgers.edu}


\begin{document}
\begin{abstract}
    Combining stationary reflection (a compactness property) with the failure of SCH (an instance of non-compactness) has been a long-standing theme. We obtain this at $\aleph_{\omega_1}$, answering a question of Ben-Neria, Hayut, and Unger: We prove from the existence of uncountably many supercompact cardinals the consistency of $\aleph_{\omega_1}$ is strong limit together with $2^{\aleph_{\omega_1}}>\aleph_{\omega_1+1}$ and every stationary set of $\aleph_{\omega_1+1}$ reflects.
\end{abstract}
\maketitle
\let\labeloriginal\label
\let\reforiginal\ref
\def\ref#1{\reforiginal{#1}}
\def\label#1{\labeloriginal{#1}}

\section{Introduction}
The modern study of cardinal arithmetic dates back to Cohen's invention of forcing which was pivotal in solving Hilbert's first problem - the independence of the Continuum Hypothesis (CH). Building on Cohen's method, Easton showed that the value of the powerset function for \textit{regular cardinals} is only constrained by  K\"{o}nig's Lemma. In other words, one can force any ``reasonable" behavior of $\kappa\mapsto 2^\kappa$. However, for singular cardinals, the situation is more intricate. A cardinal $\kappa$ is regular if the union of less than $\kappa$-many sets all smaller than $\kappa$ has size less than $\kappa$. For example, $\aleph_0$ is regular as a finite union of finite sets is finite; the first uncountable cardinal $\aleph_1$ is regular as a countable union of countable sets is countable. {\it Singular cardinals}, e.g. $\aleph_\omega$, $\aleph_{\omega_1}$, are where this pattern breaks down. 

The Singular Cardinal Hypothesis (SCH) is a parallel of CH for singular cardinals. Early work by of Silver \cite{Siver1974} and Prikry \cite{Prikry} established that while it is possible to force the failure of SCH, this requires the existence of large cardinals. The difficulty in producing models where SCH fails is rooted in deeper ZFC-constraints on the powerset value of singular cardinals. In the early 1980s,  Shelah \cite{ShelahCard} found bounds for the power set of  $\aleph_{\omega}$, when $\aleph_{\omega}$ is a strong limit (i.e. for all $n$, $2^{\aleph_n}<\aleph_\omega$).  
That is in stark contrast to the fact by Cohen that the powerset of $\omega$ can consistently be arbitrarily high. 

Another key result by Silver \cite{Siver1974} says that SCH cannot fail for the first time at a singular cardinal of uncountable cofinality. This is in contrast to the case of singular cardinals of countable cofinality since by Magidor \cite{MagidorAnaals}, GCH {\it can} first fail at such cardinals, and even at  $\aleph_\omega$.  

In this paper, we are motivated by exploring the constraints on the failure of SCH at $\aleph_{\omega_1}$, the first singular cardinal of uncountable cofinality.  We demonstrate that SCH at $\aleph_{\omega_1}$ is not a consequence of \textit{stationary reflection}-- a central reflection type principle-- by constructing a forcing extension where SCH fails at $\aleph_{\omega_1}$ and stationary reflection holds at $\aleph_{\omega_1+1}$.

Informally, stationary reflection says that ``large enough" or positive sets have ``large enough" i.e. positive initial segments. Precisely, a stationary $S\subset \kappa$ reflects if there is $\alpha<\kappa$, such that $S\cap\alpha$ is stationary in $\alpha$. 
It is well-known that stationary reflection fails at the successor of a regular cardinal $\kappa$ as the stationary set $E^{\kappa^+}_\kappa$ of ordinals below $\kappa^+$ of  \textit{critical cofinality} (i.e. cofinality $\kappa$) does not reflect. So the possible cardinals where stationary reflection holds are limit cardinals (and more interestingly inaccessible cardinals), and successors of singular cardinals.

Stationary reflection at inaccessible cardinals follows from weak compactness and is equivalent to it within the constructible universe $L$. However, as proven by Mekler and Shelah \cite{MeklerShelah}, it is strictly weaker than weak compactness in terms of consistency strength i.e. the consistency of theory ZCF+ ``there is an inaccessible cardinal where stationary reflection holds" does not imply the consistency of ZFC+``there is a weakly compact cardinal". Further results on this line can be found in \cite{BagariaMagidorMancilla,BagariaMagidorSakai,TomJing}.

As we previously noted, our focus lies in stationary reflection at successors of singular cardinals. Unlike other combinatorics involving successors of singular, the method of singularizing a measurable cardinal via Prikry-type forcings does not quite work as the problematic set of ordinals of critical (ground model) cofinality usually remains stationary in the extension. However, this non-reflecting stationary set is the only obstacle in the sense that any stationary set of points of non-critical cofinality reflects. In the first part of this paper, we will review this aspect for Magidor forcing \cite{ChangeCofinality} and some of its relevant variants.  

Hence, to achieve full stationary reflection at a successor of a singular it is more feasible to work directly with successors of singulars i.e. without singularizing cardinals. Indeed, Magidor \cite{Magidor1982} first showed that stationary reflection holds at the successor of a singular limit of supercompact cardinals,  a result we will frequently use throughout this paper. Recently, Hayut and Unger \cite{HAYUT_UNGER_2020} improved the initial large cardinal assumption of infinitely many supercompact cardinals to just one $\kappa^+$-supercompact cardinal. All the models mentioned thus far satisfy SCH. Obtaining models of stationary reflection at the successor of a singular cardinal alongside the failure of SCH is more challenging, as the failure of SCH entails instances of non-reflection. (see for example- Foreman and Todorcevic \cite{ForemanTodorcevic}, Poveda, Rinot and Sinapova \cite{SigmaPrikry1}). 

The first construction of stationary reflection at the successor of a singular together with the failure of SCH is due to Assaf Sharon in his thesis \cite{SharonThesis}. The approach that Sharon took is to blow up the powerset of a singular cardinal using the long-extender diagonal Prikry forcing, and then, as Mekler and Shelah did, iterate (in a Prikry-type style) the forcing which destroys non-reflecting stationary sets that may rise along the iteration. Laying the foundations to iterate such forcings, Sharon's method was recently generalized to an abstract iteration scheme by Poveda, Rinot and Sinapova \cite{SigmaPrikry1,SigmaPrikry2,SigmaPrikry3}.

Interestingly, more recently Gitik found a way to force stationary reflection with the failure of SCH outright \cite{GitikReflection} i.e. without iterating  to kill non-reflective stationary sets. This was done in his overlapping extender forcing from \cite{GITIK_2019}. Independently,  Ben-Neria, Hayut and Unger \cite{BEN-NERIA_HAYUT_UNGER_2024} proved the same result by employing the novel technique of the generalized Bukovski-Dehornoy phenomena of iterated ultrapowers. 
However, bringing the result down to $\aleph_\omega$ was first obtained by Poveda, Rinot and Sinapova \cite{SigmaPrikry3} and independently by Ben-Neria, Hayut and Unger \cite{BEN-NERIA_HAYUT_UNGER_2024}.  

As for uncountable cofinalities, it was pointed out by Gitik that his proof generalizes to the uncountable version of the overlapping extender, and so does the iterated ultrapower argument from \cite{BEN-NERIA_HAYUT_UNGER_2024}. This established the consistency of the failure of SCH at a singular cardinal of uncountable cofinality together with stationary reflection at his successor. However, it remained open how to bring this down to small cardinals, prompting the following question from Ben-Neria, Hayut, and Unger:
\begin{question}[{\cite[Question 1.5]{BEN-NERIA_HAYUT_UNGER_2024}}]
    Is it consistent that SCH fails $\aleph_{\omega_1}$
and every stationary
subset of $\aleph_{\omega_1+1}$ reflects?
\end{question}
In this paper we provide a positive answer to this question. Our approach reverts back to Sharon's original approach, putting the collapses on as in the $\Sigma$-Prikry framework. The forcing notion which replaces the \textit{diagonal extender based Prikry forcing with interleaved collapses} is the uncountable version of \textit{the overlapping extender forcing with collapses} which is due to Sittinon Jirattikansakul \cite{Sittinon}. Although several ideas from \cite{SigmaPrikry3} are used in this paper, we present here a concrete construction of the iteration we use. We do not intend to develop an abstract iteration scheme for Magidor-Radin types of forcings, leaving that as a possible avenue for future research. 

One might hope that as in Gitik's argument, already Sittinon's forcing suffices to obtain the failure of SCH at $\aleph_{\omega_1}$ with stationary reflection at $\aleph_{\omega_1+1}$. However, this is not the case: The argument for reflection in \cite{GitikReflection} and in many of the papers cited above relies on the fact that given a stationary set in the full Prikry extension, one can define its ``traces" by only looking at the Prikry poset restricted to the direct extension order. Then, one arranges that reflection holds when the Prikry order is restricted  to direct extensions. This leads to an argument that stationary sets with stationary traces reflect.

Gitik shows that in his construction the traces of stationary sets are always stationary, and so the above strategy gives reflection.  However, when we add collapses, even in the countable case, traces of stationary sets may no longer be stationary. We call such sets {\it fragile}, this terminology was coined in \cite{SigmaPrikry3}. And this is the reason we will define an iteration to kill all fragile  stationary sets. 

This paper is organized as follows:
\begin{itemize}
    \item In Section~\S\ref{Section: magidor Case} we review stationary reflection in the Magidor and the Magidor forcing with collapses extensions.
    \item In Section~\S\ref{Section: Sittinons forcing}, we present Sittinon's forcing, develop the relevant notations for this paper, and prove several results regarding nonreflecting stationary sets in the extension.
    \item In Section~\S\ref{Section: iteratio} we define the iteration that kills non-reflecting stationary sets.
    \item In Section~\S\ref{section: projections} we prove several properties of some factors and quotients of the iteration.
    \item In Section~\S\ref{Section: Killing stationary sets} we prove that successor steps of the iteration kill the stationarity of the intended non-reflecting stationary sets.
    \item In Section~\S\ref{section: Prikry} we prove the Prikry property for the iteration.
    \item In Section~\S\ref{Section: reflection} we prove that every non-reflecting stationary set that is generated by the iteration must be fragile.
    \item In Section~\S\ref{Section: end game} we present the bookkeeping argument and the proof of our main result.
    
\end{itemize}
\subsection{Notations}
$[\lambda]^{<\mu}$ denotes subsets of $\lambda$ of size less than $\mu$. For $\vec{\alpha}=\l \alpha_1,...,\alpha_n\r$, denote by $|\vec{\alpha}|=n$ and $\vec{\alpha}(i)=\alpha_i$. If $I\subseteq\{1,..,n\}$ then $\vec{\alpha}\restriction I=\l \vec{\alpha}(i_1),...,\vec{\alpha}(i_k)\r$ where $\{i_1,i_2,...,i_k\}$ is the increasing enumeration of $I$. For $Y\subseteq\omega$, $\vec{\alpha}\restriction Y=\vec{\alpha}\restriction (Y\cap\{1,..,n\})$. We will usually identify $\vec{\alpha}$ with the set $\{\alpha_1,..,\alpha_n\}$.
We denote by $E^\lambda_\mu$ the subset of $\lambda$ of ordinal of cofinality $\mu$. $\text{Refl}(A,B)$ denoted the statement that for every stationary set $S\subseteq A$ there is $\alpha\in B$ such that $S\cap\alpha$ is stationary. Our forcing direction is mostly standard, where $p\leq q$ means that $p$ is stronger, namely, $p\Vdash q\in\dot{G}$.
\section{Reflection in Magidor forcing generic extension}\label{Section: magidor Case}
Before diving into the proof of our main result, we go over reflection properties after forcing with the well-known  Magidor forcing. These results will illuminate the necessity of the key notion of {\it fragile sets} which will come up later.

Recall that after forcing with Prikry forcing there is a non-reflecting stationary set at $\kappa^+$ which consists of all the ordinals $\alpha<\kappa^+$ whose $V$-cofinality is $\kappa$. Moreover, if $\text{Refl}(E^{\kappa^+}_{<\kappa})$ in the ground model, then $\text{Refl}(E^{\kappa^+}_{<\kappa})$ holds in the Prikry extension.   

We follow the presentation of Magidor forcing due to Mitchell \cite{MITCHELLHowWeak} which uses \textit{coherent sequences} which are also due to Mitchell \cite{MitCoherent}.
 A \textit{coherent sequence} is a sequence  \( \vec{U}=\langle U(\alpha,\beta) \mid \beta<o^{\vec{U}}(\alpha)\ ,\alpha\leq\kappa \rangle \) such that \((U(\alpha,\beta)\) is a normal ultrafilter over \(\alpha\) and
letting \(j:V\rightarrow Ult(U(\alpha,\beta),V)\) be the corresponding elementary embedding, $j(\vec{U})\restriction\alpha=\vec{U}\restriction\langle\alpha,\beta\rangle$,
where 
$$\vec{U}\restriction\alpha=\langle U(\gamma,\delta) \mid \delta<o^{\vec{U}}(\gamma) \ ,\gamma\leq\alpha \rangle$$ 
$$\vec{U}\restriction\langle\alpha,\beta\rangle=\langle U(\gamma,\delta) \mid (\delta<o^{\vec{U}}(\gamma) ,\ \gamma<\alpha) \vee (\delta<\beta ,\ \gamma=\alpha)\rangle$$
  For every $\alpha\leq\kappa$, denote $\cap\vec{U}(\alpha)=\underset{i<o^{\vec{U}}(\alpha)}{\bigcap}U(\alpha,i)$.
\begin{definition}\label{Magidor-conditions}
$\mathbb{M}[\vec{U}]$ consist of elements $p$ of the form
$p=\langle t_1,...,t_n,\langle\kappa,B\rangle\rangle$.
 For every  $1\leq i\leq n $, $t_i$ is a pair $\langle\kappa_i,B_i\rangle$ and $B_i=\emptyset$ if $o^{\vec{U}}(\kappa)=0$. 
\begin{enumerate}
\item $\langle\kappa_1,...,\kappa_n\rangle\in [\kappa]^{<\omega}$.
\item $B\in\cap\vec{U}(\kappa)$, \  $\min(B)>\kappa_n$.
    \item  For every  $1<i\leq n$, if \ $o^{\vec{U}}(\kappa_i)>0$, $B_i\in \cap\vec{U}(\kappa_i)$, $\min(B_i)>\kappa_{i-1}$.
\end{enumerate}

\end{definition}
For a condition $p=\langle t_1,...,t_n,\langle \kappa,B\rangle\rangle\in\mathbb{M}[\vec{U}]$ we denote
$n=l(p)$, $p_i=t_i$, $B_i(p)=B(t_i)$ and $\kappa_i(p)=\kappa(t_i)$ for any $1\leq i\leq l(p)$, $t_{l(p)+1}=\langle\kappa,B\rangle$, $t_0=0$. 
The \textit{lower part} of a condition $\l t_1,...,t_n,\l A,\vec{U}\r\r\in \mathbb{M}[\vec{U}]$ is  $\l t_1,...,t_n\r$.
\begin{definition}\label{Magidor-order}
 For $p=\langle t_1,t_2,...,t_n,\langle\kappa,B\rangle\rangle,q=\langle s_1,...,s_m,\langle\kappa,C\rangle\rangle\in \mathbb{M}[\vec{U}]$ , define  $q \leq p$ ($q$ extends $p$) iff $\exists 1 \leq i_1 <...<i_n \leq m\leq i_{n+1}=m+1$ such that for every $1 \leq j \leq m+1$:
\begin{enumerate} 
        \item If  $i_r=j$ then $\kappa(t_r)=\kappa( s_{i_r})$ and $C(s_{i_r})\subseteq B(t_r)$.
        \item If $ i_{r-1}<j<i_{r}$ then $\kappa(s_j) \in B(t_r)$ and $B(s_j)\subseteq B(t_r)\cap \kappa(s_j)$.
\end{enumerate}
We also use ``q directly extends p", $q \leq^{*} p$ if $q\leq p$ and $l(p)=l(q)$.
\end{definition}
The Magidor forcing is intended to change turn a measurable cardinal of high Mithcell order to be a singular cardinal of uncountable cofinality.

  Let $p\in\mathbb{M}[\vec{U}]$. For every $ i\leq l(p)+1$, $\alpha\in B_{i}(p)$, define
$$p^{\frown}\l\alpha\r=\langle p_1,...,p_{i-1},\langle\alpha,B_{i}(p)\cap \alpha\rangle,\langle\kappa_{i}(p),B_{i}(p)\setminus(\alpha+1)\rangle,p_{i+1},...,p_{l(p)+1}\rangle.$$
 For $\langle\alpha_1,...,\alpha_n\rangle\in[\kappa]^{<\omega}$ define recursively,
 
$$p^{\frown}\langle\alpha_1,...,\alpha_n\rangle=(p^{\frown}\langle\alpha_1,...,\alpha_{n-1}\rangle)^{\frown}\langle\alpha_n\rangle$$

The condition $p^\smallfrown{\vec{\alpha}}$ form minimal extensions of $p$ in the sense that if $q\leq p$ and the ordinals appearing in $q$ include the ordinal of $\vec{\alpha}$ then $q\leq p^\smallfrown \vec{\alpha}$. We will see several generalization of this 
and also of the following straightforward decomposition.
\begin{proposition}\label{dec} Let $p\in\mathbb{M}[\vec{U}]$ and $\langle\lambda,B\rangle$ a pair in $p$. Then
$$\mathbb{M}[\vec{U}]/p\simeq \Big(\mathbb{M}[\vec{U}]\restriction \lambda\Big)/\Big(p\restriction\lambda\Big)\times\Big(\mathbb{M}[\vec{U}]\restriction(\lambda,\kappa)\Big)/\Big(p\restriction(\lambda,\kappa)\Big)$$
where 
 $$\mathbb{M}[\vec{U}]\restriction\lambda=\Big\{p\restriction\lambda\mid p\in\mathbb{M}[\vec{U}] \ and \ \lambda \ appears \ in \ p\Big\}$$
 $$\mathbb{M}[\vec{U}]\restriction(\lambda,\kappa)=\{p\restriction (\lambda,\kappa)\mid p\in\mathbb{M}[\vec{U}] \ and \ \lambda \ appears \ in \ p\}$$

\end{proposition}
 \begin{lemma}
$\mathbb{M}[\vec{U}]$ is $\kappa^+$-Knaster and therefore $\kappa^+$-c.c.
\end{lemma}
\begin{lemma}
    $\mathbb{M}[\vec{U}]$ satisfies the Prikry condition i.e. for any statement in the forcing language $\sigma$ and any $p\in\mathbb{M}[\vec{U}]$ there is $p^*\leq^*p$ such that $p^*||\sigma$ i.e. either $p^*\Vdash\sigma$ or $p\Vdash\neg\sigma$.
\end{lemma}
\begin{lemma}[The strong Prikry Lemma]
Assume that $o^{\vec{U}}(\kappa)$ is below the first measurable. Then for every $p\in\mathbb{M}[\vec{U}]$ and dense open set $D\subseteq \mathbb{M}[\vec{U}]$ there is $p^*\leq^* p$ and $\vec{\xi}\in [o^{\vec{U}}(\kappa)]^{<\omega}$ such that for every $\vec{\mu}\in \prod_{i=0}^{|\vec{\xi}|-1} B_{\vec{\mu}(i)}(p)$ $p^{*\smallfrown}\vec{\mu}\in D$. 
\end{lemma}
Lemmas of similar flavor will be proven in this paper. For more details regarding the Magidor forcing see \cite{Gitik2010,TomMoti,partOne,Parttwo}.

We are now ready to analyze the reflection properties of $\mathbb{M}[\vec{U}]$.
\begin{theorem}\label{Thm: reflection in magidor}
 Let $\mathbb{M}[\vec{U}]$ be the Magidor forcing, such that $\Vdash_{\mathbb{M}[\vec{U}]} cf(\check{\kappa})=\check{\gamma}$ for some $\gamma<\kappa$. Then in any generic extension of $\mathbb{M}[\vec{U}]$, $(E^{\kappa^+}_\kappa)^V$ does not reflect. In particular $\neg \text{Refl}((E^{\kappa^+}_\gamma)^{V[G]})$. 
 
 On the other hand, if in $V$ we have $$\text{Refl}((E^{\kappa^+}_{<\kappa})^V,\{\alpha<\kappa^+\mid cf^V(\alpha)\text{ is a successor cardinal}\}),$$
 we have $V[G]\models\text{Refl}((E^{\kappa^+}_{<\kappa})^V)$.
\end{theorem}
\begin{proof}
 For the first part, it is clear that in $V$ the set of $(E^{\kappa^+}_\kappa)^V$ does not reflect. Hence it will stay non-reflecting in the generic extension. Since $(E^{\kappa^+}_\kappa)^{V}\subseteq (E^{\kappa^+}_\gamma)^{V[G]}$, we also get $\neg \text{Refl}( (E^{\kappa^+}_\gamma)^{V[G]})$.
 
 Now to see that $$V[G]\models \text{Refl}((E^{\kappa^+}_{<\kappa})^V),$$  suppose that $T\subseteq \kappa^+$ is stationary such that for every $\alpha\in T$, $cf^V(\alpha)<\kappa$. By F\"{o}dor's Lemma,  we can assume that $cf^V(\alpha)=\theta$ for some fixed $\theta<\kappa$. For every possible lower part $p_0$ of a condition $p\in\mathbb{M}[\vec{U}]$, let $$T_{p_0}=\{\alpha<\kappa^+\mid \exists B\in\bigcap\vec{U},\ \ p_0^{\smallfrown}\l \kappa,B\r\Vdash_{\mathbb{M}[\vec{U}]}\alpha\in \dot{T}\}.$$
 Denote $G_{<\kappa}=\{p_0\mid \exists p\in G, \ p_0\text{ is the lower part of }p\}$. Since there are only $\kappa$-many lower parts, and $T\subseteq \bigcup_{p_0\in G_{<\kappa}}T_{p_0}$,  there is a lower part $p_0$ of a condition $p\in G$ such that $T_{p_0}$ is stationary in $V[G]$ and therefore stationary in $V$. By our assumption, there is $\delta<\kappa^+$, $cf(\delta)$ is a successor cardinal $\tau^+$ in $V$ such that $T_{p_0}\cap \delta$ is stationary.
 By extending $p_0$ if necessary, we may assume that $\max(p_0)>cf(\delta)$, (note that if $q\leq p_0$ then $T_{p_0}\subseteq T_q$ and therefore $T_q\cap \delta$ would also be stationary). Find a club $C\in V$ in $\delta$ of order type $cf(\delta)$ and let $T^*=T_{p_0}\cap C$.
 Upper parts for $p_0$ have sufficient closure, so by the definition of $T_{p_0}$,  there is a set $B^*\in\bigcap\vec{U}$ for which $p_0^{\smallfrown}\l \kappa,B^*\r$ forces that $T^*\subseteq T$. 
 Since $\mathbb{M}[\vec{U}]$ preserves successor cardinals $cf^{V[G]}(\delta)=cf^V(\delta)$. Finally, we note that stationary sets of $\tau^+$ are preserved. This is true by factoring $\mathbb{M}[\vec{U}]$ to the part with $\tau^+$-c.c, and a part which is more than $\tau^{++}$-closed with respect to $\leq^*$.
\end{proof}
\begin{corollary}
   Suppose that $\kappa$ is $\kappa^+$-supercompact and $\vec{U}$ is a coherent sequence. Then in any generic extension of $\mathbb{M}[\vec{U}]$,   $E^{\kappa^+}_{<\kappa}$ reflects. 
\end{corollary}
\begin{proof}
 Note that if $j$ is the $\kappa^+$-supercompact embedding, then for every stationary set $S\subseteq E^{\kappa^+}_{<\kappa}$, the point $\sup j''\kappa^+$ is a reflection point of $j(S)$. Hence there would be many reflection points of $S$ of successor cofinality. This implies that the hypotheses of the previous theorem are satisfied and therefore $\text{Refl}((E_{<\kappa}^{\kappa^+})^V)$ holds in the generic extension.
 
\end{proof}



Given a name of a stationary set $\dot{T}$ we can define its ``traces" as follows, for each lower part $p_0$ of a condition in $\mathbb{M}[\vec{U}]$, we define $T_{p_0}$ as in the proof of \ref{Thm: reflection in magidor}. Then we use that some such trace is stationary and we can apply reflection in $V$. 

In the subsequent section, we will also employ similar traces, except that now, they will be defined not in the ground model but in some directed closed extension, where a suitable instance of reflection holds. (Those will be constructed by looking at the direct extension of our Prikry forcing). Then, following the basic idea above, we will show that if a set has stationary traces, it must reflect. Sets without stationary traces will be called ``fragile" and we will kill them.

\section{Reflection in the overlapping extender-based forcing with collapses}\label{Section: Sittinons forcing}
We follow S. Jirattikansakul \cite{Sittinon} who developed the uncountable version of the overlapping extenders forcing which is based on Gitik's overlapping extender forcing \cite{GITIK_2019}. We start with a model of $GCH$, and a sequence $\l \kappa_i\mid i<\omega_1\r$ of supercompact cardinals, and let $\kappa=\sup_{i<\omega_1}\kappa_i$ 
.
\begin{notation}
Let $\bar{\kappa}_0=\omega$, and for every $0<\beta<\omega_1$, $\beta\leq\omega_1$ denote by $\bar{\kappa}_\beta=\sup_{\alpha<\beta}\kappa_\alpha$. In particular, $\kappa=\bar{\kappa}_{\omega_1}$, and if $\beta$ is successor then $\bar{\kappa}_{\beta}=\kappa_{\beta-1}$. Note that for all $\beta<\omega_1$, $\bar{\kappa}_{\beta}<\kappa_\beta$.  
 \end{notation}

In the ground model $V$ we assume that the following holds:
\begin{enumerate}
    \item [(I)]GCH.
    \item [(II)]  For each $i<\omega_1$, $\kappa_i$ is indestructible under $\kappa_i$-directed closed forcings.
    \item [(III)] For each $i<\omega_1$, we fix a $(\kappa_i,\kappa^{++})$-extender $E_i$ on $\kappa_i$  such that the extender ultrapower $M_{E_i}$ computes cardinals correctly up to and including $\kappa^{++}$, and $M_{E_i}^{\kappa_i}\subseteq M_{E_i}$.
    \item [(IV)] For each $i<\omega_1$, we have $s_i:\kappa_i\rightarrow \kappa_i$ the function representing $\kappa$ in $M_{E_i}$, namely $j_{E_i}(s_i)(\kappa_i)=\kappa$. We can assume that $s_i(\nu)>\max\{\nu,\bar{\kappa}_i\}$ for every $\nu$.
    \item [(V)] For each $i_1<i_2<\omega$, $j_{E_{i_2}}(E_{i_1})\restriction\kappa^{++}=E_{i_1}$, and in particular $j_{E_{i_2}}(\alpha\mapsto E_{i_1}\restriction s_{i_1}(\alpha)^{++})(\kappa_{i_2})=E_{i_1}\in M_{E_{i_2}}$.
\end{enumerate}

Before defining the forcing, we have the following notation due to Merimovich \cite{Mer,Mer3}:
\begin{definition}
    
\begin{enumerate}
    \item [(a)] For $i<\omega_1$, an \textit{$i$-domain} is a set $d\in [\kappa^{++}]^{\kappa_i}$ such that $\kappa_i+1\subseteq d$. (a set which can be the domain of the Cohen part in the extender-based forcing)
    \item [(b)] Define the $d$\textit{-maximal coordinate} to be $$mc_i(d):=(j_{E_i}\restriction d)^{-1}=\{\l j_{E_i}(x),x\r\mid x\in d\}.$$ 
    \item [(c)] Denote by $E_i(d)$, the measure on $V_{\kappa^{++}}$ generated by the seed $mc_i(d)$, namely $$X\in E_i(d)\Longleftrightarrow mc_i(d)\in j_{E_i}(X).$$
\end{enumerate}

\end{definition}
We define a typical element in a $E_i(d)$-measure-one is called an \textit{object}. It is a sequence which reflect the properties of the maximal coordinate and will provide a ``layer" of points for the Prikry-point to be added below $\kappa_i$. 
\begin{definition}
 An \textit{$(i,d)$-object} is a function $\mu$, such that:
 \begin{enumerate}
     \item $\kappa_i \in\dom(\mu)\subseteq d$ 
     \item $\dom(\mu)\cap \kappa_i=\mu(\kappa_i)$ and $\mu\restriction \mu(\kappa_i)=id$.
     \item
     $\rng(\mu)\subseteq \kappa_i$ 
     \item $\bar{\kappa}_i<|\dom(\mu)|=\mu(\kappa_i)<\kappa_i$ and $\mu(\kappa_i)$ is inaccessible.
     
     \item $\mu$ is order preserving. 
 \end{enumerate}
 
The set \textit{$OB_i(d)$}  is the set of $(i,d)$-objects, and clearly $OB_i(d)\in E_i(d)$.
\end{definition}
We can omit the `$i$' from the ``$(i,d)$-object" and from $OB_i(d)$ since $i$ is determined by $d$ (recall that $|d|=\kappa_i$). 

Merimovich's notation pays off when analyzing the Rudin-Keisler projections between the different ultrafilters: At the price of complicating the notations, the projections between the measures of the extender are plain restrictions:
\begin{definition}
If $d\subseteq d'$ are $i$-domains let $\pi_{d',d}:OB(d')\rightarrow OB(d)$ be the restriction function $\pi_{d',d}(\mu)=\mu\restriction d(=\mu\restriction \dom(\mu)\cap d)$.
\end{definition}
Clearly the generators and the measures are projected using the restriction map, and therefore $\pi_{d',d}$ is a Rudin-Keisler projection of $E_i(d')$ to $E_i(d)$.

Here are two relevant combinatorial lemmas regarding such measures:
\begin{proposition}\label{Prop: Rowbottom for object-measures}
 Let $0\leq i_1,i_2<...i_n<\omega_1$ and $F:\prod_{k=0}^nA_{i_k}\rightarrow X$ is any function such that $d_{i_k}$ is $i_k$-domain, $A_{i_k}\in E_{i_k}(d_{i_k})$ and $|X|<\kappa_0$. Then there is $B_{i_k}\subseteq A_{i_k}$ such that $B_{i_k}\in E_{i_k}(d_{i_k})$ such that $F\restriction \prod_{k=0}^n B_{i_k}$ is constant.
\end{proposition}
\begin{proposition}\label{Prop: Bound the number of same projection to normal}
 For each $i<\omega_1$ and an 
 $i$-domain $d$, there is a set $A_i(d)$ such that
$A_i(d) \in E_i(d)$, and for each $\nu<\kappa_i$, the size of $\{\mu \in A_i(d) \mid \mu(\kappa_i) = \nu\}$ is at
most $s_i(\nu)^{++}$.
\end{proposition}
We keep the notation of $A_i(d)$. We also need notations for the normal measure:
\begin{itemize}
    \item The normal measure derived from $E_i$ denoted by $E_i(\kappa_i)$ is the set of all $X\subseteq\kappa_i$ such that $\kappa_i\in j_{E_i}(X)$.
    \item If $A\in E_i(d)$ (then recall that $\kappa_i\in d$) define $$A(\kappa_i)=\{\mu(\kappa_i)\mid \mu\in A\}\in E_i(\kappa_i).$$
\end{itemize}
\subsection{The Extender-Based forcing with collapses}
A condition in $\mathbb{P}_{\bar{E}}$ is a sequence $p=\l p_i\mid i<\omega_1\r$ such that there is a finite set $\supp(p)\in[\omega_1]^{<\omega}$, and we have that:
$$p_i=\begin{cases} \l f_i, h^0_i,h^1_i,h^2_i\r & i\in \supp(p)\\
\l f_i,A_i,H^0_i,H^1_i,H^2_i\r & i\notin\supp(p)\end{cases}.$$
We require that for every $i_1<i_2<\omega_1$, $\dom(f_{i_1})\subseteq \dom(f_{i_2})$. 
Denote  $\supp(p)=\{i_1<i_2<...<i_r\}$, then for every $i<\omega_1$:
$$\bar{\kappa}_{i}<\bar{\kappa}_{i}^+<\underset{\text{inac. placeholder for }\mu(\kappa_i)}{f_i(\kappa_i)}<s_i(f_i(\kappa_i))<s_i(f_i(\kappa_i))^+<s_i(f_i(\kappa_i))^{++}<\kappa_{i}$$
and we require that:
\begin{enumerate}
    \item [(i)] If there is $k<r$ such that $i\in [i_k,i_{k+1})$ $f_i$ is a partial function from $s_{i_{k+1}}(f_{i_{k+1}}(\kappa_{i_{k+1}}))^{++}$ to $\kappa_i$ such that $\kappa_i+1\subseteq\dom(f_i)$ and $|f_i|=\kappa_i$ (In particular $\dom(f_i)$ is an $i$-domain).
    \item [(ii)] If $i\in [i_r,\omega_1)$, then $f_i$ is a partial function from $\kappa^{++}$ to $\kappa_i$ such that $\dom(f_i)$ is an $i$-domain.
    \item [(iii)] for $i\in \supp(p)$,   $h^0_i\in Col(\bar{\kappa}_{i}^+,<f_i(\kappa_i))$, $h^1_i\in Col(f_i(\kappa_i), s_i(f_i(\kappa_i))^+)$, $h^2_i\in Col(s_i(f_i(\kappa_i))^{+3},<\kappa_i)$. (So in the generic extension we will have:
        $$\bar{\kappa}_{i}<\bar{\kappa}_{i}^+<\bar{\kappa}_{i}^{++}=f_i(\kappa_i)<\bar{\kappa}_i^{+3}=s_i(f_i(\kappa_i))^{++}<\bar{\kappa}_{i}^{+4}=s_i(f_i(\kappa_i))^{+3}<\bar{\kappa}_i^{+5}=\kappa_i$$
 \item [(iv)] For $i\notin\supp(p)$:
    \begin{enumerate}
        \item $A_i\in E_i(\dom(f_i))$.
        \item $\dom(H^0_i)=\dom(H^1_i)=A_i\in E_i(\dom(f_i))$ and $\dom(H^2_i)=A_i(\kappa_i)\in E_i(\kappa_i)$.
        \item $H^0_i(\mu)\in Col(\bar{\kappa}_{i}^+,<\mu(\kappa_i))$, $H^1_i(\mu)\in Col(\mu(\kappa_i),s_i(\mu(\kappa_i))^+)$ and $H^2_i(\mu(\kappa_i))\in Col(s_i(\mu(\kappa_i))^{+3},<\kappa_i)$.
    \end{enumerate}
    
\end{enumerate}
The direct extension is clear:
\begin{definition}
The direct order is defined by $p\geq^* q$ if $\supp(p)=\supp(q)$, for every $i$, $f^p_i\subset f^q_i$ and:
\begin{enumerate}
\item If $i\in \supp(p)$ then for $h^{r,p}_i\subseteq h^{r,q}_i$ for $r=0,1,2$.
    \item If $i\notin \supp(p)$, $\pi_{\dom(f^q_i),\dom(f^p_i)}[A^q_i]\subseteq A^p_i$. $H^{r,p}_i(\mu\restriction\dom(f^p_i))\subseteq H^{r,q}_i(\mu)$ for $r=0,1,2$.
\end{enumerate}
\end{definition}
Note that in $2$ above, since $\dom(f^p_i)\subseteq \dom(f^q_i)$, it is possible that elements of $A^p_i$ are also elements of $A^q_i$ even if the domain strictly increases.

In the next definition, we adopt the following notations:
    for any two functions $f:A\rightarrow B$, $g:C\rightarrow D$ we denote by $g\circ f=g\circ (f\restriction f^{-1}[\dom(g)])$.
\begin{definition} Let $i\notin \supp(p)$.
$\mu\in A^p_{i}$ is addable to $p$ if:
\begin{enumerate}
    \item [(I)] $\bigcup_{\alpha<i}\dom(f_\alpha)\subseteq \dom(\mu)$.
    \item [(II)] For every $\beta\in [\max(\supp(p)\cap i),i)$:
    \begin{enumerate}
        \item $\mu[\dom(f_\beta)]\subseteq s_i(\mu(\kappa_i))^{++}$ (so the decomposition $f_\beta\circ \mu^{-1}$ will have domain which is a subset of $s_i(\mu(\kappa_i))^{++}$ this is necessary by condition (i) above). 
        \item $A^p_\beta\circ \mu^{-1}:=\{\nu\circ\mu^{-1}\mid \nu\in A^p_\beta\}\in E_\beta(\mu[\dom(f_\beta)])$. 
    \end{enumerate}
    
\end{enumerate}
\end{definition}
Denote by $f\oplus\mu$ the function where $(f\oplus\mu)(x)=\mu(x)$ for $x\in\dom(\mu)$ and $(f\oplus\mu)(x)=f(x)$ otherwise.
\begin{definition}
Let $i\notin \supp(p)$,  $i_*=\max(\supp(p)\cap i)$, and $\mu\in A^p_{i}$, define $p^\smallfrown\mu$ as the condition $q$ such that 
$\supp(q)=\supp(p)\cup\{i\}$, and 
\begin{enumerate}
    \item For $r\in [0,i_*)\cup (i,\omega_1)$, $p_r=q_r$.
    \item For $r=i$, $f^q_i=f^p_i\oplus\mu$, $h^{0,q}_i=H^{0,p}_i(\mu)$, $h^{1,q}_i=H^{1,p}_i(\mu)$ and $h^{2,q}_i=H^{2,p}_i(\mu(\kappa_i))$
    \item For $r\in [i_*,i)$, $f^q_r=f^p_r\circ\mu^{-1}$, so $\dom(f^{q}_r)=\mu[\dom(f^p_r)]$ and $A^q_r=A^p_r\circ \mu^{-1}\in E_r(\dom(f^q_r))$.  For $\nu\in A^q_r$, define $H^{l,q}_r(\nu)=H^{l,p}_r(\nu\circ\mu)$ for $l=0,1$ and $H^{2,q}_r=H^{2,p}_r$ (note that $\kappa_r=\mu(\kappa_r)$ by requirement (2),(4) of an $(i,d)$-object above and since $\kappa_r\leq \bar{\kappa}_i$). 
\end{enumerate}
 \end{definition}
 We refer to the operation of composing $f^p_i$ with $\mu^{-1}$ by saying that $f^p_i$ is ``squished by $\mu$". 
 
 Inductively we define when $\l \mu_1,...,\mu_n\r$ is addable to $p$ and $p^{\smallfrown}\l\mu_1,...\mu_n\r$, if $\l \mu_1,...,\mu_{n-1}\r$ is addable, $\mu_n$ is addable to $p^{\smallfrown}\l\mu_1,...,\mu_{n-1}\r$ and $$p^{\smallfrown}\l\mu_1,...,\mu_{n-1},\mu_n\r=(p^{\smallfrown}\l\mu_1,...,\mu_{n-1}\r)^\smallfrown\l\mu_n\r.$$
 The order is defined by $q\leq p$ iff $q\leq^* p^\smallfrown \vec{\mu}$ for some sequence $\vec{\mu}$ addable to $p$. 
 \begin{remark}\label{Ramark: Commutativity}
     We have some sort of commutativity, namely if $\nu\in A_i^p$ and $\mu\in A_j^{p^\smallfrown\nu}$ where $j<i$, then by definition, there is $\mu'\in A_j^{p}$ such that $\mu'\circ \nu^{-1}=\mu$, and 
$p^{\smallfrown}\l\nu,\mu\r=p^\smallfrown\l\mu',\nu\r$. Moreover, by $(II)$, $\dom(\mu')\subseteq \dom(f^p_j)\subseteq \dom(\nu)$ and therefore   $\mu'$ is unique. Hence, when considering step-extensions, we may always assume that $\l \mu_1,...,\mu_n\r\in A^p_{i_1}\times....\times A^p_{i_n}$ is such that $i_1<i_2<...<i_n$.
 \end{remark}
 \begin{definition}
For a given $p\in\mathbb{P}_{\vec{E}}$ and $\vec{\xi}=\l \xi_1,...,\xi_k\r\in [\omega_1\setminus\supp(p)]^{<\omega}$ (i.e. $\vec{\xi}$ is increasing), we let $Ex_{\vec{\xi}}(p)$ to be the set of all finitely supported functions $\vec{\mu}$, such that $\dom(\vec{\mu})=\vec{\xi}$ and for every $1\leq r\leq k$, $\vec{\mu}_{\xi_r}\in A^p_{\xi_r}$ is addable. 
For a given $\vec{\mu}\in Ex_{\vec{\xi}}(p)$, $p^{\smallfrown}\vec{\mu}$ is called a \textit{$\vec{\xi}$-extension of $p$}. 
 
\end{definition}
By the definition of the order, the set $\{p^{\smallfrown}\vec{\mu}\mid \vec{\mu}\in Ex_{\vec{\xi}}(p)\}$ is pre-dense. We would like to make sure it is an antichain. At this point  there might be distinct $\l\mu_1,...,\mu_k\r,\l\nu_1,...,\nu_k\r$ such that $p^{\smallfrown}\l\nu_1,...,\nu_k\r,p^{\smallfrown}\l\mu_1,...,\mu_k\r$ are compatible. Let us prove that by restricting to a dense subset of $\mathbb{P}_{\vec{E}}$ where we avoid this problem.
     \begin{claim}
         Let $p\in \mathbb{P}_{\vec{E}}$, $i<\omega_1$. Denote by $f^p_i:\lambda_i\rightarrow\kappa_i$ the $i^{\text{th}}$ Cohen function in $p$ (where $\lambda_i$ is either of the form $s_{i^*}(f_{i^*}(\kappa_{i^*}))^{++}$ or $\mu_i=\kappa^{++}$). Then there is $A\subseteq A^p_i$, $A\in E_i(\dom(f^p_i))$ such that for every $\mu\in A$, and every $\alpha\in\dom(\mu)\cap \mu(\kappa_i)$,  $\mu(\alpha)=\alpha$ and for $\alpha\in \dom(\mu)\setminus\mu(\kappa_i)$ then $\mu(\alpha)>f^p_i(\alpha)$.
     \end{claim}
     \begin{proof}
         Consider $mc_i(\dom(f^p_i))$, the seed of the measure $E_i(\dom(f))$. Then for every $\beta\in \dom(mc_i(\dom(f^p_i)))$, $\beta=j_{E_i}(\alpha)$ for some $\alpha\in \dom(f^p_i)$ and $j_{E_i}(f^p_i)(j_{E_i}(\alpha))=j_{E_i}(f^p_i(\alpha))$. Since the range of $f^p_i$ is $\kappa_i$,  $j_{E_i}(f^p_i)(j_{E_i}(\alpha))=f^p_i(\alpha)$. If $\alpha<\kappa_i$, then $$mc_i(\dom(f^p_i))(\alpha)=mc_i(\dom(f^p_i))(j_{E_i}(\alpha))=\alpha.$$ If $\alpha\geq\kappa_i$, then necessarily $mc_i(\dom(f^p_i))(j_{E_i}(\alpha))=\alpha>f^p_i(\alpha)$.
     \end{proof}
 \begin{corollary}\label{cor: 1-point extension unique}
     Suppose that $\mu,\nu\in A$ ($A$ is the set from the previous claim) and $f^p_i\oplus\mu=f^p_i\oplus\nu$ then $\mu=\nu$.
 \end{corollary}
 \begin{proof}
     First note that $\mu(\kappa_i)\neq f^p_i(\kappa_i)$ and $\nu(\kappa_i)\neq f^p_i(\kappa_i)$ and therefore $\mu(\kappa_i)=f^p_i\oplus\mu(\kappa_i)=f^p_i\oplus\nu(\kappa_i)=\nu(\kappa_i)$. If $\alpha<\mu(\kappa_i)$ then by definition $\mu(\alpha)=\alpha=\nu(\alpha)$. If $\alpha\geq\kappa$ and $\alpha\in \dom(\mu)$ then since $\mu(\alpha)\neq f^p_i(\alpha)$ then $f^p_i\oplus\mu(\alpha)=\mu(\alpha)$ and therefore $f^p_i\oplus\nu(\alpha)=\mu(\alpha)$ but then it must be that $\alpha\in\dom(\nu)$ and $\nu(\alpha)=\mu(\alpha)$. The argument in case $\alpha\in\dom(\nu)$ is similar. 
 \end{proof}
 Hence, we can shrink every measure one set and directly extend each condition $p$ to a condition $p^*\leq^*p$ with the property of \ref{cor: 1-point extension unique}. Let us only consider conditions of the form $p^*$ and force with this dense subset of $\mathbb{P}_{\vec{E}}$.
 
\begin{corollary}\label{Cor: properties of step extension}
    Let $p\in \mathbb{P}_{\bar{E}}$.\begin{enumerate}
        \item For every $\{i_1,...,i_n\}\subseteq \omega_1\setminus \supp(p)$ (not necessarily increasing) there is a unique increasing sequence $\l \xi_1,...,\xi_n\r\in [\omega_1]^{<\omega}$ such that for every $\l \nu_1,...,\nu_n\r\in A^p_{i_1}\times ...\times A^p_{i_n}$, there is a unique $\vec{\mu}\in Ex_{\vec{\xi}}(p)$ for which $$p^{\smallfrown}\l \nu_1,...,\nu_n\r=p^{\smallfrown}\l\mu_1,...,\mu_n\r.$$ 
        \item Suppose that $p\leq^* q$ and $\vec{\xi}\subseteq \vec{\zeta}\in[ \omega_1\setminus \supp(p)]^{<\omega}$, then for every $\vec{\mu}\in Ex_{\vec{\zeta}}(p)$ there is a unique $\vec{\nu}\in Ex_{\vec{\xi}}(q)$ and $\vec{\rho}\in Ex_{\vec{\zeta}\setminus \vec{\xi}}(q^\smallfrown\vec{\nu})$ such that $$p^\smallfrown \vec{\mu}\leq^*( q^\smallfrown\vec{\nu})^\smallfrown\vec{\rho}.$$
    \item If $p\leq q$ then for every $\vec{\zeta}\in [\omega_1]^{<\omega}$, and every $\vec{\mu}\in Ex_{\vec{\zeta}}(p)$ there is a unique $\vec{\mu}'\in Ex_{\vec{\zeta}\cup (\supp(p)\setminus \supp(q)}(q)$ such that $p^\smallfrown\vec{\mu}\leq^* q^\smallfrown\vec{\mu}'$.\end{enumerate}
\end{corollary}
\begin{proof}
    For $(1)$, first we define $\vec{\xi}$ is simply the increasing enumeration of $\{i_1,..,i_n\}$. Clearly $\vec{\xi}$ is unique as it is merely the increasing enumeration of $\supp(p^{\smallfrown}\l\nu_1,...,\nu_n\r)\setminus\supp(p)$.  For the second part of $(1)$ we proceed by induction on $n$, for $n=1$ the existence is trivial (take $\nu=\mu$) and the uniqueness follows by Corollary \ref{cor: 1-point extension unique}. For the induction step, suppose that $i_j=\min\{i_1,...,i_n\}$. By definition, we have that $$\nu_{j}\in A^{p^\smallfrown\nu_1,...,\nu_{j-1}}_{i_j}=\{\nu\circ \nu_1^{-1}\circ\nu_2^{-1}\circ...\circ \nu_{j-1}^{-1}\mid \nu\in A^p_j\}$$ 
    Hence there is  unique $\mu_1\in A^p_j$ such that $\mu_1\circ \nu_1^{-1}\circ\nu_2^{-1}\circ...\circ \nu_{j-1}^{-1}=\nu_j$ and $$p^{\smallfrown}\l\nu_1,...,\nu_j\r=(p^{\smallfrown}\mu_1)^\smallfrown\l\nu_1,...,\nu_{j-1}\r.$$
    It follows that $$p^{\smallfrown}\l\nu_1,...,\nu_n\r=(p^{\smallfrown}\mu_1)^\smallfrown\l\nu_1,...,\nu_{j-1},\nu_{j+1},...\nu_n\r.$$
    By the induction hypothesis applied to the condition $p^{\smallfrown}\mu_1$ and
    $\l \nu_1,...\nu_{j-1},\nu_{j+1},...,\nu_n\r$, we have there is a unique sequence $\l\mu_2,...,\mu_n\r\in Ex_{\vec{\xi}\setminus \{\xi_1\}}$ such that  $(p^{\smallfrown} \mu_1)^\smallfrown\l \mu_2,...,\mu_n\r=p^\smallfrown\l\nu_1,...,\nu_n\r$.  
    
    For $(2)$, by induction on $|\vec{\zeta}|$. In the induction step, we consider $\nu=\mu_i$ where $\zeta_i=\min(\vec{\xi})$. 
    Since $\vec{\mu}$ is increasing, by definition, $\nu\in  A^p_{\zeta_i}$, and 
$$(p^\smallfrown\nu)^\smallfrown \l \mu_1\circ \nu^{-1},...,\mu_{i-1}\circ \nu^{-1}\r=p^\smallfrown \l\mu_1,...,\mu_i\r$$
    Note that since $p\leq^* q$, $p^\smallfrown \nu\leq^* q^\smallfrown (\nu\restriction \dom(f^q_{\zeta_i}))$. 
    Hence we may apply the induction hypothesis to $\vec{\mu'}=\l  \mu_1\circ\nu^{-1},...,\mu_{i-1}\circ\nu^{-1},\mu_{i+1},...\mu_n\r$, to find $\vec{\nu}'$ and $\vec{\rho}$ such that $$(p^\smallfrown \nu)^\smallfrown \vec{\mu}'\leq^* (q^\smallfrown (\nu\restriction \dom(f^q_{\zeta_1})))^\smallfrown \vec{\nu}')^\smallfrown \vec{\rho}$$ 
Noting that $(\nu\restriction \dom(f^q_{\zeta_1})))^\smallfrown \vec{\nu}'$ is increasing, we have that $\vec{\nu}=(\nu\restriction \dom(f^q_{\zeta_1})))^\smallfrown \vec{\nu}'\in Ex_{\vec{\xi}}(q)$ is as wanted.

    For $(3)$, $p\leq^* q^\smallfrown\vec{\nu}$  so we apply $(2)$ with $\vec{\xi}=\emptyset$, to find the unique $\vec{\rho}$ such that $p^\smallfrown\vec{\mu}\leq^*(q^\smallfrown \vec{\nu})^\smallfrown \vec{\rho}$. By $(1)$ there is a unique $\vec{\mu}'\in Ex_{\vec{\eta}\cup \supp(p)\setminus \supp(q)}(q)$ such that $(q^\smallfrown \vec{\nu})^\smallfrown\vec{\rho}=q^\smallfrown\vec{\mu'}$.  
\end{proof}
Similarily to the above, we have that for all increasing $\l\mu_1,...,\mu_n\r,\l\mu'_1,...,\mu'_k\r$, if $$ p^{*\smallfrown}\l\mu_1,...,\mu_n\r,p^{*\smallfrown}\l\mu_1',...,\mu'_{k}\r\geq^* q$$ then $\l\mu_1,...,\mu_n\r=\l\mu_1',...,\mu'_{k}\r$. This enables us to identify each extension $q\leq p$ with a unique (increasing) $\vec{\mu}$ such that $q\leq^*p^{\smallfrown}\vec{\mu}$.

Let summarize the properties of $\mathbb{P}_{\bar{E}}$: \begin{proposition}\label{Properties}
 \begin{enumerate}
     \item $\mathbb{P}_{\bar{E}}$ is $\kappa^{++}$-cc.
     \item Cardinals $\lambda\geq \kappa$ are preserved.
     \item For every $p\in \mathbb{P}_{\bar{E}}$, and for every $\xi\in \supp(p)$ the forcing $\mathbb{P}_{\bar{E}}/p$ can be factored to a product   $$\overset{\mathbb{S}_\xi}{\overbrace{\underset{\mathbb{S}_{<\xi}}{\underbrace{\mathbb{P}\restriction\xi\times Col(\bar{\kappa}_\xi^+,{<}f_\xi(\kappa_\xi))\times Col(f_\xi(\kappa_\xi),\tau^+)}}\times Col(\tau^{+3},{<}\kappa_\xi)}}\times \underset{\mathbb{S}_{>\xi}}{\underbrace{Add(\kappa_\xi^+,\lambda)\times \mathbb{P}_{>\xi}}}$$
     Where,
     \begin{enumerate}
     \item $\tau=s_\xi(f_\xi(\kappa_\xi))$, and
     $\lambda$ is either $\kappa^{++}$ if $\xi=\max(\supp(p))$ or $s_{\xi^*}(f_{\xi^*}(\kappa_{\xi^*}))^{++}$ where $\xi^*=\min(\supp(p)\setminus\xi+1)$.
     \item $\mathbb{P}\restriction \xi=\{p\restriction \xi\mid p\in\mathbb{P}_{\bar{E}}\}$, then $|\mathbb{P}\restriction\xi|\leq s_\xi(f_\xi(\kappa_\xi))^{++}$ and has the $\bar{\kappa}_\xi$-cc. 
     
     \item
     $\mathbb{P}_{>\xi}$ is an $\kappa_\xi^+$-closed forcing  with respect to the $\leq^*$-order.
      \end{enumerate}
      Note that $|\mathbb{S}_{<\xi}|\leq s_\alpha(f_\alpha(\kappa_\alpha))^{++}$ and $\mathbb{S}_{>\xi}$ is $\kappa_\xi^+$-closed with respect to $\leq^*$.
     \item $\mathbb{P}_{\bar{E}}$ satisfies the Prikry property, that is, for every formulate in the forcing language $\sigma$ and every $p\in \mathbb{P}_{\bar{E}}$, there is $p^*\leq^* p$ such that $p^*||\sigma$.
     \item $\mathbb{P}_{\bar{E}}$ satisfies the strong Prikry property, that is, for every dense open set $S\subseteq \mathbb{P}_{\bar{E}}$, and every $p\in \mathbb{P}_{\bar{E}}$, there is $p^*\leq^* p$ and $\vec{\xi}\in [\omega_1\setminus \supp(p)]^{<\omega}$, such that for every $\vec{\mu}\in Ex_{\vec{\xi}}(p^*)$, $p^{*\smallfrown}\vec{\mu}\in D$.
     \item $Card^{V^{\mathbb{P}_{\bar{E}}}}\cap [\bar{\kappa}_\xi,\kappa_\xi)=\{\bar{\kappa}_\xi, \bar{\kappa}_\xi^+,f_\xi(\kappa_\xi),\tau^{++},\tau^{+3}\}$. If $\xi\leq\omega_1$ is limit then $\bar{\kappa}_\xi=(\aleph_\xi)^{V^{\mathbb{P}_{\bar{E}}}}$ and in particular $\kappa=\aleph_{\omega_1}^{V^{\mathbb{P}_{\bar{E}}}}$.
     \item In $V^{\mathbb{P}_{\bar{E}}}$, $2^\kappa=\kappa^{++}$, $2^{\bar{\kappa}_\xi}\leq \kappa_{\xi}$. If $\xi\leq\omega_1$ is limit then $\bar{\kappa}_\xi$ is a strong limit, and moreover if $\xi<\omega_1$, then $2^{\bar{\kappa}_\xi}=\tau^{++}$.
      
 \end{enumerate}
  
 \end{proposition}
 \subsection{Reflection and fragile stationary sets in $\mathbb{P}_{\bar{E}}$}
 \begin{definition}

For $p\in\mathbb{P}_\xi$, define $\xi(p)=\max(\supp(p))$. Let $\xi<\omega_1$, define 
$$\mathbb{P}_\xi=\{p\in\mathbb{P}\mid \xi(p)=\xi\}$$ with the induced order from $\mathbb{P}$.
 \end{definition}
 \begin{remark}
Any condition $p$ in  $\l \mathbb{P}_{\xi},\leq\r$ has a component $H^p_{\xi+1}\colon A_{\xi+1}(\dom(f_{\xi+1}))\to V$ such that 
$H^p_{\xi+1}(\mu)$ is in $$ Col(\kappa_\xi^+,<\mu(\kappa_{\xi+1}))\times Col(\mu(\kappa_{\xi+1}),s_{\xi+1}(\mu(\kappa_{\xi+1}))^+),$$ 
note that these are the same forcings as computed in $M_{E_{\xi+1}}$ since these forcings consist of sequences of size $\leq\kappa_{\xi+1}$ which are available in $M_{E_{\xi+1}}$ and $\kappa^+$ is computed correctly in $M_{E_{\xi+1}}$.
Recall that $s_{\xi+1}$ represents $\kappa$ in the extender ultrapower by $E_{\xi+1}$, it follows that $j_{E_{\xi+1}}(H^p_{\xi+1})(mc_{\xi+1}(\dom(f^p_{\xi+1})))$ is in 
$$ Col(\kappa_\xi^+,<\kappa_{\xi+1})\times Col(\kappa_{\xi+1},\kappa^+).$$ Let us denote these conditions by $H^{p,0}_{\xi+1}$ and $H^{p,1}_{\xi+1}$. 
 Suppose that $G$ is $V$-generic for $\l\mathbb{P}_\xi,\leq\r$ and let $h^0_G=\bigcup_{p\in G} H^{p,0}_{\xi+1}$ and $h^{1}_G=\bigcup_{p\in G}H^{p,1}_{\xi+1}$. 
 \begin{claim}
     $h^{0}_G\times h^{1}_G$ is $V$ generic for $Col(\kappa_\xi^+,<\kappa_{\xi+1})\times Col(\kappa_{\xi+1},\kappa^+)$
 \end{claim}
 \begin{proof}
     We will prove that $p\mapsto H^{p,0}_{\xi+1},H^{p,1}_{\xi+1}$ defines a projection from $\mathbb{P}_E$ to the collapses. 
     Let $j_{\xi+1}(H^{p})(mc)\geq q\in Col(\kappa_\xi^+,<\kappa_{\xi+1})\times Col(\kappa_{\xi+1},\kappa^+)$ be any condition, then $q$ is 
     represented by some function $j_{\xi+1}(g)(d)=q$ where $d\in[\kappa^{++}]^{<\omega}$. 
     Without loss of generality, suppose that $\dom(f_{\xi+1})\subseteq d$, and consider $mc_{\xi+1}(d)$. 
     Define a condition $p'\leq^* p$ that defers only at coordinate $\xi+1$, where we arbitrarily extend $f^p_{\xi+1}$ to $f^{p'}_{\xi+1}$ with domain $d$, 
     we take $A^{p'}_{\xi+1}=\pi^{-1}_{d,\dom(f^p_{\xi+1})}$ and for each $\mu$ in the measure one set we let $(H^{p',0}_{\xi+1}(\mu),H^{p',1}_{\xi+1}
     (\mu))=g(Im(\mu))$ (and $H^2$ is lifted naturally as well). The condition $p'\in \mathbb{P}_\xi$ satisfies that $p'\leq^* p$ and  $j_{\xi+1}(H^{p'}_{\xi+1})(mc_{\xi+1}(f^{p'}_{\xi+1})=j_{\xi+1}(g)(d)=q$.
 \end{proof} 
 So $\l\mathbb{P}_{\xi},\leq\r$ collapses $\kappa^+$ to $\kappa_{\xi+1}$, makes $\kappa_{\xi+1}=\kappa_{\xi}^{++}$. To see that $\kappa_{\xi}^+$ and $\kappa_{\xi+1}$ are preserved, we note that in fact $\l\mathbb{P}_{\xi},\leq\r$ is forcing equivalent to: \begin{equation}\label{equation1}
     \mathbb{P}\restriction\xi+1\times   Col(\kappa_\xi^+,<\kappa_{\xi+1})\times  Col(\kappa_{\xi+1},\kappa^+)\times\mathbb{T}
 \end{equation} where $\mathbb{P}\restriction \xi+1$ has the $\kappa_\xi^+$-c.c. and $\mathbb{T}$ is a $\kappa_{\xi+1}^+$-directed closed forcing. 
 From this factorization, it is clear that $\kappa^+_\xi,\kappa_{\xi+1}$ are preserved.
 
 \end{remark}
  \begin{lemma}\label{Lemma: generic lift}
In $V^{\mathbb{P}_\xi}$, $\text{Refl}(E^{(\kappa^+)^V}_{\leq\bar{\kappa}
_{\xi}},E^{(\kappa^+)^V}_{<\kappa_\xi})$ holds.
 \end{lemma}
 \begin{proof}
First note that in $V^{\mathbb{P}_{\xi}}$, $|(\kappa^{+})^V|=cf((\kappa^+)^V)=\kappa_{\xi+1}=\kappa_{\xi}^{++}$. Thus, it is enough to prove $\text{Refl}(E^{\kappa_\xi^{++}}_{\leq\bar{\kappa}
_{\xi}},E^{\kappa_\xi^{++}}_{<\kappa_\xi})$.
Let $$\mathbb{R}=Col(\kappa_\xi^+,<\kappa_{\xi+1})\times Add(\kappa_{\xi+1}^+,\kappa^{++})\times \mathbb{T}$$ be the $\kappa_\xi^+$-closed part from the factoring of $\mathbb{P}_\xi$ in Equation \ref{equation1}. By the indestructibility of $\kappa_\xi$, there is $j:V^{\mathbb{R}}\rightarrow M$ be a $\kappa_{\xi+1}$ - supercompact embedding with critical point $\kappa_{\xi}$. We lift this embedding to $V^{\mathbb{P}_\xi}$, by first lifting it over the small forcing $\mathbb{S}_{\xi}$, and then over $Col(\tau, <\kappa_{\xi})$ in an generic extension by the $\tau$-closed forcing $Col(\tau, <j(\kappa_{\xi}))/Col(\tau, <\kappa_{\xi})$as in \cite[Theorem 10.5]{CummingsHand}. So we have $j^*:V^{\mathbb{P}_\xi}\rightarrow M^*$, with  $j^*$ being definable in  $V^{\mathbb{P}_\xi}[H]$, where $H$ is generic for a $\tau$-closed forcing for where $\tau>\bar{\kappa}_\xi^+$. 
 
 Let $S$ be a stationary subset of $E^{\kappa_\xi^{++}}_{\leq\bar{\kappa}
_{\xi}}$ in $V^{\mathbb{P}_\xi}$. By Shelah \cite[Lemma 4.4]{Shelah1991}, $S\in I[\kappa_\xi^{++}]$ (here it is important that we look at the double successor), and since $\tau>\bar{\kappa}_\xi^+$, by Shelah again \cite{SHELAH1979357}, its stationarity is preserved under $\tau$-closed forcing and in particular in the generic extension by $H$. Then by standard arguments, we have that $j^*(S)\cap \sup j''\kappa_{\xi+1}$ is stationary. Since $M^*\models\cf(\sup j''\kappa_{\xi+1})=\kappa_{\xi+1}<j(\kappa_\xi)$, by elementarity, $S$ reflects at a point $\gamma$ of cofinality less than $\kappa_{\xi}$.
  \end{proof}
  \begin{definition}
      For a $\mathbb{P}_{\vec{E}}$-name $\dot{S}$, and $\xi<\omega_1$, let
      $$\dot{S}_\xi=\{\l\beta,p\r\mid p\in\mathbb{P}_{\xi}, \ p\Vdash_{\mathbb{P}_{\bar{E}}}\beta\in \dot{S}\}$$

  \end{definition}
  \begin{definition}
       For a condition $r^*\in\mathbb{P}_{\vec{E}}$ and a $\mathbb{P}_{\vec{E}}$-name 
 such that $r^*\Vdash_{\mathbb{P}_{\vec{E}}}``\dot{S}$ is a stationary subset of $\kappa^+$", we say that $\dot{S}$ is $r^*$-fragile, if for all sufficiently large $\xi<\omega_1$, for all $q\leq r^*$, with $\xi(q)\geq\xi$, $q\Vdash_{\mathbb{P}_{\xi(q)}}`` \dot{S}_{\xi(q)}$ is nonstationary". We say that $\dot{S}$ is fragile if it is $1_{\mathbb{P}_{\vec{E}}}$-fragile
  \end{definition}
\begin{theorem}\label{Theore: nonreflecting implies fragile for Q0}

Suppose that $p\Vdash_{\mathbb{P}_{\vec{E}}}``\dot{S}$ is a non-reflecting stationary set", then $\dot{S}$ is $p$-fragile. In particular, if $p=1_{\mathbb{P}_{\vec{E}}}$,  then $\dot{S}$ is fragile.
\end{theorem}
\begin{proof}
Suppose that $p\Vdash_{\mathbb{P}_{\vec{E}}}``\dot{S}$ is a stationary subset of $\kappa^+$" and  non-$p$-fragile. By shrinking if necessary, we may assume that $p\Vdash_{\mathbb{P}_{\vec{E}}}``\dot{S}$ concentrates on some fixed cofinality $\theta$ below $\kappa$". By the non-fragility, we can extend $p$ if necessary so that there is some $\xi<\omega_1$ with $\xi(p)=\xi$ (namely $p\in\mathbb{P}_{\xi}$) such that $p\Vdash_{\mathbb{P}_{\xi}}``\dot{S}_{\xi}$ is stationary". Increasing $\xi$ if necessary, we may assume that $\theta<\bar{\kappa}_\xi$. Note that $p\Vdash_{\mathbb{P}_\xi}``\dot{S}_\xi$ consist of points of cofinality $\theta$". To see this, first note that since $\theta<\bar{\kappa}_\xi$, then for any $\nu$, if $cf^{V^{\mathbb{P}_{\vec{E}}}}(\nu)=\theta$, then $cf^{V^{\mathbb{P}_{\xi}}}(\nu)=\theta$. Suppose toward a contradiction that $q\leq p$, $q\in \mathbb{P}_\xi$ is such that $q\Vdash_{\mathbb{P}_\xi}``\nu\in \dot{S}_\xi\wedge cf(\nu)\neq\theta$". Then for some $q'\leq q$, $\l\nu,q'\r\in \dot{S}_\xi$ and therefore $q'\Vdash_{\mathbb{P}}``\nu\in \dot{S}\wedge cf(\nu)\neq\theta$", contradiction the choice of $\theta$. 

We conclude that $p\Vdash_{\mathbb{P}_\xi}``\dot{S}_\xi$ is a stationary subset of $E^{(\kappa^+)^V}_{\leq\bar{\kappa}_\xi}$". By the previous lemma, it follows that $p\Vdash_{\mathbb{P}_\xi}``\dot{S}_{\xi}$ reflects at a point $\gamma$ of cofinality less than $\kappa_{\xi}$". Hence there is a $\mathbb{P}_\xi$-name  for a stationary set $\dot{A}$ such that $p\Vdash_{\mathbb{P}_\xi}``\dot{A}\subseteq \dot{S}_\xi\cap\gamma$ and $otp(\dot{A})=\cf(\gamma)<\kappa_\xi$".

We decompose as in Equation \ref{equation1}, $\mathbb{P}_\xi=\mathbb{S}_{\xi}\times\mathbb{R}$, where $\tau=s_\xi(f_\xi(\kappa_\xi))^{+3}$. 
Recall that $\mathbb{S}_\xi$ has size less than $\kappa_\xi$ and $\mathbb{R}$ is $\kappa_\xi^+$-closed. 
Therefore, by Easton's Lemma, $1_{\mathbb{S}_\xi}\Vdash_{\mathbb{S}_\xi}``\mathbb{S}\text{ is }<\kappa_\xi^+$-distributive".
Denote by $p=(p_0,p_1)\in \mathbb{S}_{\xi}\times\mathbb{R}$. Since $\cf(\gamma)<\kappa_{\xi}$, by its distributivity, $\mathbb{R}$ cannot have added $\dot{A}$, so we may assume that $\dot{A}$ is a $\mathbb{S}_\xi$-name.  

Next, 
we construct a condition $r^*\in\mathbb{R}$, such that $(p_0,r^*)\leq p$ and $(p_0,r^*)\Vdash_{\mathbb{P}_{\bar{E}}} ``\dot{A}\subset \dot{S}$". We do this as follows.

Suppose that $\l \dot{\beta}_i\mid i<cf(\gamma)\r$ is a sequence of $\mathbb{S}_\xi$-names for ordinals such that $p_0\Vdash_{\mathbb{S}_\xi}``
\dot{A}=\{\dot{\beta}_i\mid i<\cf(\gamma)\}"$. Working in $V$, for all $(s,r)\in\mathbb{P}_\xi=\mathbb{S}_\xi\times\mathbb{R}$ such that $(s,r)\leq p$ and every $i<\cf(\gamma)$, if for some $\beta<\gamma$, $s\Vdash_{\mathbb{S}_\xi} \beta=\dot{\beta}_i$, then, since $p\Vdash_{\mathbb{P}_\xi} ``\dot{A}\subseteq \dot{S}_\xi"$, $(s,r)\Vdash_{\mathbb{P}_\xi} ``\beta\in \dot{S}_\xi"$. By definition of $\dot{S}_\xi$, it follows that there is $(s',r')\leq (s,r)$ such that $\l \beta,(s',r')\r\in\dot{S}_\xi$ which in turn implies (again by the definition of $\dot{S}_\xi$) that $(s',r')\Vdash_{ \mathbb{P}_{\bar{E}}} ``\beta\in \dot{S}"$. Using that the closure of $\mathbb{R}$ is greater than $\max(\cf(\gamma), |\mathbb{S}_\xi|)=\kappa_\xi$,  we can do this inductively for all $i<\cf(\gamma)$, and all conditions in $s\in\mathbb{S}_\xi$ to produce a condition  $r^*\in\mathbb{R}$, such that $r^*\leq p_1$ and for every $i<cf(\gamma)$ and every $s\in\mathbb{S}_\xi$ such that $s\leq p_0$ there is $s'\leq s$ such that  $(s',r^*)\Vdash_{\mathbb{P}_{\bar{E}}}`` \dot{\beta}_i\in\dot{S}"$. It follows that $(p_0,r^*)\Vdash_{\mathbb{P}_{\bar{E}}}``\dot{A}\subseteq \dot{S}"$.

Finally, since $V^{\mathbb{P}_{\bar{E}}}$ has the same subsets of $\cf(\gamma)$, and $V^{\mathbb{S}_\xi}$,  and $\dot{A}$ is forced by $p_0$ to be stationary  in $V^{\mathbb{S}_\xi}$, $(p_0,r^*)\Vdash_{\mathbb{P}_{\bar{E}}}``\dot{A}$ is stationary". We conclude that $(p_0,r^*)\Vdash`` \dot{S}$ reflect at $\gamma$".
\end{proof}
\section{Defining the Iteration}\label{Section: iteratio}
Our goal is to define iterations $\mathbb{P}\subseteq\mathbb{P}^*$ which are Prikry-type and kill non-reflecting stationary sets that arise along the iteration.
Initially, we define  $\mathbb{Q}_1=\mathbb{Q}^*_1=\mathbb{P}_{\vec{E}}$. A general condition in $\mathbb{P}$ or $\mathbb{P}^*$ is a sequence $p=\l p_\alpha\mid \alpha<\kappa^{++}\r$ such each $p_0\in\mathbb{P}_{\vec{E}}$ and each $p_\alpha$ is going to be a $p\restriction \alpha$-strategy or a weak-$p\restriction\alpha$-strategy respectively (see Definition \ref{Def: startegy}). The support of each condition in both iterations is $\leq\kappa$, namely, $$\supp(p)=\{\alpha<\kappa\mid p_\alpha\neq \emptyset\}$$ has cardinality at most $\kappa$. We denote by $\mathbb{Q}_\beta=\{p\restriction\beta\mid p\in\mathbb{P}\}$, similarly $\mathbb{Q}^*_\beta=\{p\restriction\beta\mid p\in\mathbb{P}^*\}$, and define $(\mathbb{Q}_\beta,\leq_\beta),(\mathbb{Q}^*_\beta,\leq_\beta)$ recursively.

\subsection{Definitions at level $1$}
As we already declared, $\mathbb{Q}_1=\mathbb{P}_{\vec{E}}$, but further definitions are needed in order to define future steps of the iteration.
\begin{definition}
    For $p\in\mathbb{P}_{\bar{E}}$, we let $Ex(p)=\bigcup_{\vec{\xi}\in [\omega_1\setminus \supp(p)]^{<\omega}}Ex_{\vec{\xi}}(p)$.  We order $Ex(p)$ by $\subseteq$, namely $\vec{\nu}\leq \vec{\mu}$ if $\dom(\vec{\nu})\subseteq\dom(\vec{\mu})$ and for every $i\in\dom(\vec{\nu})$, $\vec{\nu}_i=\vec{\mu}_i$. 
\end{definition}
We would like to compare the extensions of $p$ and $q$ for $p\leq q$. For that we define a function $w^p_q:Ex(p)\rightarrow Ex(q)$ as follows: first, if $p\leq^*q$ then $\supp(p)=\supp(q)$ and for each $i\in \supp(p)$, $\pi_{\dom(f^p_i),\dom(f^q_i)}[A^p_i]\subseteq A^q_i$, in which case we define for $\l\mu_1,...,\mu_n\r\in A_{\xi_1}\times...\times A_{\xi_n}$, $$w^p_q(\l\mu_1,...,\mu_n\r)=\l\mu_1\restriction \dom(f^q_{\xi_1}),...,\mu_n\restriction \dom(f^q_{\xi_n})\r.$$
    
     Suppose that $p\leq q$, then there is $\vec{\nu}=\l\nu_1,...,\nu_k\r\in A^q_{\xi_1}\times...\times A^q_{\xi_k}$ such that $p\leq^*q^{\smallfrown}\vec{\nu}$.
     Let $\l\mu_1,...,\mu_k\r\in Ex(p)$, by Corollary \ref{Cor: properties of step extension}, there is a unique $\vec{\mu'}\in Ex(q)$ such that  $p^{\smallfrown}\vec{\mu}\leq^* q^\smallfrown\vec{\mu}'$. Define $w^p_q(\vec{\mu})=\vec{\mu}'$. Note that by uniqueness, we have that whenever $p\leq q\leq r$, $w^{p}_r(\vec{\mu})=w^{p}_q(w^{q}_r(\vec{\mu}))$ (as both sides of the equation satisfy the condition $p^\smallfrown\vec{\mu}\leq^* r^\smallfrown\vec{\mu}'$).
     
   Recall that given $r\in\mathbb{P}_{\vec{E}}$, $\xi(r)=\max(\supp(r))$ which is also the unique $\xi<\omega_1$ such that $r\in\mathbb{P}_\xi$.  
Also, recall that $$(\mathbb{P}_\xi,\leq)\simeq \mathbb{S}_{<\xi}\times Col(s_{\xi}(f_\xi(\kappa_\xi))^{+3},<\kappa_{\xi})\times \mathbb{S}_{>\xi}$$
and that: 
\begin{enumerate}
    \item $|\mathbb{S}_{<\xi}|<s_\xi(f_\xi(\kappa_\xi))^{+3}<\kappa_\xi$.
    \item  $\mathbb{S}_{>\xi}$ is $\kappa_\xi^+$-closed.
\end{enumerate}
For $p\in\mathbb{P}_{\vec{E}}$ with $\xi\in\supp(p)$, we denote by $p\restriction\mathbb{S}_\xi,p\restriction\mathbb{S}_{<\xi}$ and in case $\xi=\xi(p)$, $p\restriction\mathbb{S}_{>\xi}$, the restriction of $p$ to the components of each of the respective forcing.

  \begin{definition}
      Let $p\in\mathbb{P}_{\vec{E}}$, the domain of $p$-labeled blocks
      $$Ex^*(p)=\bigcup_{i<\omega_1}\{\l \vec{\mu},c\r\in Ex(p)\times \mathbb{S}_i\mid \xi(p^{\smallfrown}\vec{\mu})=i, \ c\leq(p^{\smallfrown}\vec{\mu})\restriction_{\mathbb{S}_i}\}$$

  \end{definition}

Whenever $p\leq q$, we define $w^{*p}_q:Ex^*(p)\rightarrow Ex^*(q)$ by $w^{*p}_q((\vec{\nu},c))=(w^p_q(\vec{\nu}),c)$. Note that $w^{*p}_q$ is well defined. Indeed, by definition, $\supp(p^\smallfrown\vec{\nu})=\supp(q^\smallfrown w^p_q(\vec{\nu}))$ and therefore $\xi(p^\smallfrown\vec{\nu})=\xi(q^\smallfrown w^p_q(\vec{\nu}))$, moreover, $c\leq(p^\smallfrown\vec{\nu})\restriction_{\mathbb{S}_i}\leq (q^\smallfrown w^p_q(\vec{\nu}))\restriction_{\mathbb{S}_i}$.

On $Ex^*(p)$ we have a natural ordering, $$(\vec{\mu},c)\leq (\vec{\nu},d)\text{ iff  }\vec{\nu}\subseteq \vec{\mu}\text{ and  }c\restriction_{\mathbb{S}_{ \xi(p^{\smallfrown}\vec{\nu})}}\leq d.$$
\begin{remark}
    
    $w^{*p}_q$ is (weakly) order preserving, namely, if $(\vec{\mu},c)\leq(\vec{\nu},d)$ then $w^{*p}_q(\vec{\mu},c)\leq w^{*p}_q(\vec{\nu},d)$.
\end{remark}
\subsection{The recursive steps of the iteration}
Suppose we have defined $(\mathbb{Q}_\delta,\leq_\delta)$ for every $\delta<\beta$, and we have maintained a recursive definition of $q^\smallfrown(\vec{\mu},c)\in\mathbb{Q}_\delta$ for all $q\in\mathbb{Q}_\delta$ and all $(\vec{\mu},c)\in Ex^*(q_0)$. Such that the operation cohere; that is:
$$(\dagger) \text{ For all }0<\delta_1<\delta_2<\beta\text{ and all }q\in\mathbb{Q}_{\delta_2}, \  (q^\smallfrown(\vec{\nu},c))\restriction\delta_1=(q\restriction\delta_1)^\smallfrown(\vec{\nu},c)$$
$$(\pitchfork)  \text{ For all }\delta<\beta\text{ and all }q\in\mathbb{Q}_{\delta}, \ \supp(q^\smallfrown(\vec{\nu},c))=\supp(q)$$
 
For a limit $\beta$, the underlining set of $\mathbb{Q}_\beta$ is determined by $\mathbb{Q}_\delta$ for $\delta<\beta$ as the inverse limit with $\leq\kappa$ support, or explicitly, 
$$\mathbb{Q}_\beta=\{p=\l p_\alpha\mid \alpha<\beta\r\mid  \forall \delta<\beta, \  p\restriction\delta\in\mathbb{Q}_\delta\text{ and } |\supp(p)|\leq\kappa\}$$
The order $\leq_\beta$ is simply the pointwise order, namely, $$p\leq_\beta q\text{ if and only if for every }\delta<\beta,\  p\restriction\delta\le_\delta q\restriction \delta.$$
Finally, for $q\in\mathbb{Q}_\beta$, and $(\vec{\mu},c)\in Ex^*(q_0)$, there is a unique object
$q^{\smallfrown}(\vec{\mu},c)$ which satisfy $(\dagger),(\pitchfork)$. Explicitly, define $$q^{\smallfrown}(\vec{\mu},c)=\bigcup_{\delta<\beta}(q\restriction\delta)^{\smallfrown}(\vec{\mu},c).$$ Note that the union is a well-defined sequence of length $\beta$ by $(\dagger)$ of the induction hypothesis, and clearly $(\dagger)$ is preserved. Also,  $(\pitchfork)$ holds by the induction hypothesis, which is also used to see that $q^\smallfrown(\vec{\mu},c)\in \mathbb{Q}_\beta$. This concludes the definition of $\mathbb{Q}_\beta$ for a limit $\beta$.

Next, suppose that $(\mathbb{Q}_\beta,\leq_\beta)$ has been defined and let us define $\mathbb{Q}_{\beta+1}$. As promised, conditions of $\mathbb{Q}_{\beta+1}$ will have the form $q^\smallfrown p_\beta$, where $q\in\mathbb{Q}_\beta$ and $p_\beta$ is a (weak) $q$-strategy. We will denote $q^\smallfrown p_\beta$ by $\l q,p_\beta\r$. For a condition $q=\l q_\alpha\mid \alpha<\beta\r$ we define $\xi(q)=\xi(q_0)$, and let $$(\mathbb{Q}_\beta)_{\xi}=\{q\in\mathbb{Q}_\beta\mid \xi(q)=\xi\}.$$
  \begin{definition}
      For a $\mathbb{Q}_{\beta}$-name $\dot{S}$, and $\xi<\omega_1$, let
      $$\dot{S}_\xi=\{\l\nu,p\r\in On\times(\mathbb{Q}_\beta)_{\xi} \mid p\Vdash_{\mathbb{Q}_\beta}\nu\in \dot{S}\}$$

  \end{definition}
  \begin{definition}
       For a condition $r^*\in\mathbb{Q}_{\beta}$ and a $\mathbb{Q}_{\beta}$-name 
 such that $r^*\Vdash_{\mathbb{Q}_{\beta}}``\dot{S}$ is a stationary subset of $\kappa^+$", we say that $\dot{S}$ is $r^*$-fragile, if for all sufficiently large $\xi<\omega_1$, for all $q\leq r^*$, with $\xi(q)\geq\xi$, $q\Vdash_{(\mathbb{Q}_\beta)_{\xi(q)}}`` \dot{S}_{\xi(q)}$ is nonstationary". We say that $\dot{S}$ is fragile if it is $1_{\mathbb{Q}_{\beta}}$-fragile.
  \end{definition}
As in Theorem \ref{Theore: nonreflecting implies fragile for Q0}, will eventually show that every name for a non-reflecting stationary set in the extension by $\mathbb{Q}_\beta$ is fragile, but we do not need that at this point. 
Suppose $\dot{T}$ is a $\mathbb{Q}_{\beta}$-name for a fragile stationary set. Let us define $\mathbb{Q}_{\beta+1}=\mathbb{A}(\mathbb{Q}_\beta,\dot{T})$ as follows: By definition of fragility, find $\bar{\xi}<\omega_1$ such that for every $\bar{\xi}<\xi<\omega_1$, $1_{(\mathbb{Q}_\beta)_\xi}\Vdash_{(\mathbb{Q}_\beta)_{\xi}}``\dot{T}_\xi$ is non-stationary", and let $\dot{C}_{\beta,\xi}$ be a name for a club such that $1_{(\mathbb{Q}_\beta)_{\xi}}\Vdash_{(\mathbb{Q}_\beta)_{\xi}} \dot{C}_{\beta,\xi}\cap\dot{T}_{\xi}=\emptyset$. 
Let $$R_\beta=\{\l\nu,q\r\in \kappa^+\times \mathbb{Q}_\beta\mid \forall r\leq q, 
\ r\Vdash_{(\mathbb{Q}_\beta)_{\xi(r)}}\nu\in\dot{C}_{\beta,\xi(r)}\}$$
Since in this section $\beta$ is fixed, we denote $R=R_\beta$ and $\dot{C}_{\beta,\xi}=\dot{C}_{\xi}$.
\begin{lemma}\label{Lemma: some properties of R}
    \begin{enumerate}
        \item If $\l\nu,q\r\in R$ then $q\Vdash_{\mathbb{Q}_\beta} \nu\notin \dot{T}$.
        \item If $\{\l\nu_i,q\r\mid i<\tau\}\subseteq R$, where $\tau<\kappa^+$ then $\l\sup_{i<\tau}\nu_i,q\r\in R$.
    \end{enumerate}
\end{lemma}
\begin{proof}
    For $(1)$, suppose otherwise, then there is some $r\leq q$ such that $r\Vdash_{\mathbb{Q}_\beta} \nu\in\dot{T}$, but then $r\Vdash_{(\mathbb{Q}_\beta)_{\xi(r)}}\nu\in\dot{T}_{\xi(r)}$ and therefore $r\Vdash_{(\mathbb{Q}_\beta)_{\xi(r)}}\nu\notin \dot{C}_{\xi(r)}$, contradicting the definition of $R$.
    $(2)$ holds since the $\dot{C}_{\xi}$'s are forced to be closed. 
\end{proof}
    \begin{definition}
 For $q\in\mathbb{Q}_{\beta}$, a $q$-\textit{labeled block} is a function $S:Ex^*(q_0)\rightarrow [\kappa^+]^{<\kappa^+}$ such that: \begin{enumerate}
    \item $S(\vec{\mu},c)$ is closed bounded in $\kappa^+$.
    \item If $(\vec{\mu},c)\leq(\vec{\nu},d)$, then $S(\vec{\mu},c)\supseteq S(\vec{\nu},d)$.
    \item $ q^\smallfrown (\vec{\mu},c)\Vdash_{\mathbb{Q}_\beta} S(\vec{\mu},c)\cap \dot{T}=\emptyset$.
    \end{enumerate}
    \end{definition}
    
\begin{definition} \label{Def: startegy}   
 A \textit{$q$-strategy} is a sequence $\mathcal{S}=\l S_i\mid i\leq\alpha\r$ such that:
\begin{enumerate}
    \item [($\alpha$)] $S_0((\vec{\mu},c))=\emptyset$ for every $(\vec{\mu},c)$.
    \item [($\beta$)] $0\leq l(\mathcal{S}):=\alpha<\kappa^+$.
    \item [($\gamma$)] $S_i$ is a $q$-labeled block. 
    \item [($\delta$)] for every $(\vec{\mu},c)\in Ex^*(q_0)$, $S_{i}((\vec{\mu},c))\sqsubseteq S_{i+1}((\vec{\mu},c))$.
    
    \item [($\epsilon$)] For each $(\vec{\mu},c)\leq(\vec{\nu},d)$, $S_{i+1}(\vec{\nu},d)\setminus S_{i}(\vec{\nu},d)\sqsubseteq S_{i+1}(\vec{\mu},c)\setminus S_{i}(\vec{\mu},c)$.
    \item [($\zeta$)] For a limit $i\leq\alpha$, $S_i(\vec{\mu},c)=\overline{\bigcup_{j<i}S_j(\vec{\mu},c)}$.
    \item [($\eta$)]   for every $(\vec{\mu},c)\in Ex^*(q)$, $\l\max(S^{l(\mathcal{S})}(\vec{\mu},c)),q\r\in R$ 
    \item [($\theta$)] For every $i\leq l(\mathcal{S})$, there is $\bar{\xi}=\bar{\xi}_{i,\mathcal{S}}<\omega_1$ s.t. for all $r\leq q$  $\xi(r)\geq\bar{\xi}$ and all $(\vec{\mu},c)\in Ex^*(q)$ , with $q^{\smallfrown}(\vec{\mu},c)\leq r\leq q$ then $\l \max(S^i(\vec{\mu},c)),r\r\in R$ 
\end{enumerate}
 A \textit{weak-$q$-strategy} is defined the same way excluding condition $(\eta)$. 
\end{definition}
Let us define two posts $\mathbb{A}(\mathbb{Q}_\beta,\dot{T})\subseteq\mathbb{A}^*(\mathbb{Q}_\beta,\dot{T})$, we will eventually show that $\mathbb{A}(\mathbb{Q}_\beta,\dot{T})$ is dense in $\mathbb{A}^*(\mathbb{Q}_\beta,\dot{T})$.
\begin{definition}
    Let $\mathbb{A}(\mathbb{Q}_\beta,\dot{T})$ consist of pairs $$a=\l q^a,\mathcal{S}^a\r=\l q^a,\l S^a_{i}\mid 0<i\leq l(\mathcal{S}^a)\r\r,$$ where $q\in\mathbb{Q}_{\beta}$, and $\mathcal{S}^a$ is a $q$-strategy. The order is defined by $a\geq b$ if
    \begin{enumerate}
        \item $q^a\geq q^b$ and $l(\mathcal{S}^a)\leq l(\mathcal{S}^b)$.
        \item For every $(\vec{\mu},c)\in Ex^*(b_0)$,  $S^b_i(\vec{\mu},c)=S^a_i(w^{*b_0}_{a_0}(\vec{\mu},c))$.
    \end{enumerate}
    The poset $\mathbb{A}^*(\mathbb{Q}_\beta,\dot{T})$ is defined similarly, replacing $q$-strategies with weak-$q$-strategies.
\end{definition}
\begin{notation}
    The \textit{delay} of a condition $a=(q,\mathcal{S})$ is defined as $\bar{\xi}_{l(\mathcal{S}),q}$.
    
\end{notation}
In particular conditions in $\mathbb{Q}_{\beta+1}$ are exactly those with delay $0$.
\begin{remark}
    Note that by the definition of $R$, condition $(\theta)$ is equivalent to requiring that for every $(\vec{\mu},c)\in Ex^*(r), \   \l\max(S^i(w^r_q(\vec{\mu},c)),r\r\in R$.
    Condition $(\theta)$ says that there is some $\bar{\xi}<\omega_1$ such that all extensions $(q',\mathcal{S}')\leq (q,\mathcal{S})$ with  $\xi(q')\geq \bar{\xi}$ have delay $0$, i.e. satisfy condition $(\eta)$. 
\end{remark}

    Fixing $(\vec{\mu},c)$, the closed bounded sets $S_i(\vec{\mu},c)$ are similar to the conditions of shooting a club through the complement of a non-reflecting stationary set. While this forcing is only $<\kappa^+$-strategically closed and not even $\sigma$-closed, Lemma \ref{Lemma: Closure of successor step of iteration} below and condition $(\eta')$ above ensures that the forcing is $\kappa^+$-closed once the first coordinate is fixed. The proof for the lemma is a straightforward verification.
    \begin{lemma}\label{Lemma: Closure of successor step of iteration}
        Let $\l (q,\mathcal{S}_i)\mid i<\tau\r$, $\tau\leq\kappa$ be a decreasing sequence of conditions in $\mathbb{A}(\mathbb{Q}_\beta,\dot{T})$ or in $\mathbb{A}(\mathbb{Q}_\beta,\dot{T})^*$  and there is some $\bar{\xi}<\omega_1$ bounding the delays of the conditions $(q,\mathcal{S}_i)$.    Then $\l q,\bigcup_{i<\tau}\mathcal{S}_i\cup\{\mathcal{S}_{l}\}\r$ is the greatest lower bound of the conditions in the respective poset where $$\mathcal{S}_l(\vec{\mu},c)=\overline{\bigcup_{i<\tau}\mathcal{S}_i^{l(\mathcal{S}_i)}(\vec{\mu},c)}.$$ 
    \end{lemma}
    \qed
    
The verification essentially uses conditions $(\eta)$, $(\theta)$ and Lemma \ref{Lemma: some properties of R}.


\begin{definition}
Let us define $a^{\smallfrown}(\vec{\mu},c)$ for $a\in \mathbb{A}^*(\mathbb{Q}_\beta,\dot{T})$ and $(\vec{\mu},c)\in Ex^*(a_0)$ as $(q^{a\smallfrown}(\vec{\mu},c),\mathcal{S}')$ where $l(\mathcal{S}')=l(\mathcal{S}^a)$ and for every $0<i\leq \mathcal{S}'$, and every $(\vec{\nu},c)$, $S'_i(\vec{\nu},c)=S^a_i(w^{*a^{\smallfrown}_0(\vec{\mu},c)}_{a_0}(\vec{\nu},c))$. 
\end{definition}
Since $a\restriction\beta=q^\beta$, we have that $(a^{\smallfrown}(\vec{\nu},c))\restriction \beta=a\restriction\beta^{\smallfrown}(\vec{\nu},c)$, hence $(\dagger)$ is satisfied.  Also $\supp(a^\smallfrown (\vec{\nu},c))=\supp(a)$ so $(\pitchfork)$ holds as well.
\begin{claim}\label{Claim: frown yield a condition}
   $a^\smallfrown(\vec{\mu},c)\in\mathbb{A}^*(\mathbb{Q}_\beta,\dot{T})$. If moreover $a\in \mathbb{A}(\mathbb{Q}_\beta,\dot{T})$, then $a^\smallfrown(\vec{\mu},c)\in\mathbb{A}(\mathbb{Q}_\beta,\dot{T})$. 
\end{claim}
\begin{proof}
    By the inductive constrction, $q^{a^\smallfrown(\vec{\mu},c)}=q^{a\smallfrown}(\vec{\mu},c)\in\mathbb{Q}_\beta$.
    So it remains to see that $\mathcal{S}'=\mathcal{S}^{a^{\smallfrown}(\vec{\mu},c)}$ is a $q^{a\smallfrown}(\vec{\mu},c)$-strategy.
    Fix $i\leq l(\mathcal{S}')$, to see that $\mathcal{S}'_i$ is a $q^{a\smallfrown}(\vec{\mu},c)$-block, $(1)-(2)$ 
    are inherited by $\mathcal{S}^a$ and the fact that $w^{*a_0^{\smallfrown}(\vec{\mu},c)}_{a_0}$ is order preserving. $(3)$ follows since $(q^{a\smallfrown}(\vec{\mu},c))^\smallfrown(\vec{\nu},d)= q^{a\smallfrown} w^{*a_0^{\smallfrown}(\vec{\mu},c)}_{a_0}(\vec{\nu},d)$ and  $S'_i(\vec{\nu},d)=S^{q^a}(w^{*a_0^{\smallfrown}(\vec{\mu},c)}_{a_0}(\vec{\nu},d))$.
    Let us verify $(\theta)$, let $\xi\subseteq \omega_1$ witness $(\theta)$ for $S^a_i$, and let $r\leq q$ with $\xi(r)\geq \xi$ and $(\vec{\nu},d)\in Ex(q_0^\smallfrown(\vec{\mu},c))$ be such that 
    $$q^\smallfrown(\vec{\mu},c)^\smallfrown(\vec{\nu},d)\leq r\leq q^\smallfrown(\vec{\mu},c),$$ we need to check that $\l \max(S'_i(\vec{\nu},d)),r)\r\in R$. We have that $\max(S'_i(\vec{\nu},d))=\max(S^a_i(w^{*a_0^\smallfrown(\vec{\mu},c)}_{a_0}(\vec{\nu},d)))$.
    Now $q^\smallfrown w^{*a_0^\smallfrown(\vec{\mu},c)}_{a_0}(\vec{\nu},d))\leq r\leq q$ and by applying $(\theta)$ to $q^a$ we get $\l \max(S^a_i(w^{*a_0^\smallfrown(\vec{\mu},c)}_{a_0}(\vec{\nu},d))),r\r\in R$.
    
    
    So $(\alpha),(\beta),(\gamma)$ hold for $S'$, also $(\delta)-(\zeta)$ are inherited from $\mathcal{S}^a$ and $(\eta)$ is a particular case of $(\theta)$.
\end{proof}
\begin{remark}\label{remark:minimal}
    We will often define strategies by extending a condition $q^a$ to $q^b$ and define $S^b_i(\vec{\mu},c)=S^a_i(w^{*q^b}_{q^a}(\vec{\mu},c))$. In most cases, the verification that this kind of definition is a legitimate $q^b$-strategy is essentially the same as the one of the previous claim, hence in the future we will just refer to the proof of that claim. Moreover, this definition yield a minimal extension in the sense that if $a'\leq a$ is such that $(a')_0\leq b_0$, then $a'\leq b$. The proof for that is essentially in Proposition \ref{Prop: minimal}.

    One delicate part, we will have to address in a few situation, and is the heart of the problem in the proof of the Prikry property, is that when we have a sequence of conditions, then condition $(\theta)$ (and $(\eta)$) do not automatically follow for amalgamations of blocks from those conditions. The reason is that different conditions have different witnessing $\xi$'s, and these might not converge. 
    
    Note however that if the sequence is decreasing, then the witnessing $\xi$'s can only decrease and this situation disappears.
\end{remark}
\begin{proposition}\label{Prop: minimal}
    If $p\leq q$ is such that $p_0\leq q_0^{\smallfrown}(\vec{\nu},c)$, then $p\leq q ^{\smallfrown}(\vec{\nu},c)$.
\end{proposition}
\begin{proof}
By induction on $\delta$, let us prove the proposition for $p,q\in\mathbb{Q}_\delta$. For $\delta=0$ this is trivial. Suppose this is true for all $\delta<\beta$, and let us prove it for $\beta$. Let $a,b\in \mathbb{Q}_\beta$ and assume that $a_0\leq b_0^{\smallfrown}(\vec{\nu},c)$. We split into cases:
\begin{enumerate}
    \item If $\beta$ is limit, the $a\leq b^\smallfrown(\vec{\nu},c)$ follows from $(\dagger)$ and the induction hypothesis.
    \item For $\beta+1$, since $a\leq b$, we have $q^a\leq q^b$ and since $q^a_0=a_0$, the induction hypothesis implies that $q^{a}\leq q^{b\smallfrown}(\vec{\nu},c)$. Also for every $(\vec{\mu},d)\in Ex^*(a_0)$, $w^{*a_0}_{b_0}(\vec{\mu},d)=w^{*b_0^{\smallfrown}(\vec{\nu},c)}_{b_0}(w^{*a_0}_{b_0^{\smallfrown}(\vec{\nu},c)}(\vec{\mu},d))$,
    $$S^a_i(\vec{\mu},c)=S^b_i(w^{*a_0}_{b_0}(\vec{\mu},d))=S^b_i(w^{*b_0^{\smallfrown}(\vec{\nu},c)}_{b_0}(w^{*a_0}_{b_0^{\smallfrown}(\vec{\nu},c)}(\vec{\mu},d)))$$
    By definition of the strategy of $b^\smallfrown(\vec{\nu},c)$, we conclude that $$S^b_i(w^{*b_0^{\smallfrown}(\vec{\nu},c)}_{b_0}(w^{*a_0}_{b_0^{\smallfrown}(\vec{\nu},c)}(\vec{\mu},d)))=S^{b^{\smallfrown}(\vec{\nu},c)}_i(w^{*a_0}_{b_0^{\smallfrown}(\vec{\nu},c)}(\vec{\mu},d))$$
    and therefore $a\leq b^\smallfrown(\vec{\nu},c)$.

\end{enumerate}
\end{proof}
\section{Projections and factorizations of $\mathbb{Q}_\alpha$.}\label{section: projections} Recall that $$\mathbb{Q}_{\alpha,\xi}=\{a\in\mathbb{Q}_\alpha\mid \xi(a)=\xi\}.$$ The order on $\mathbb{Q}_{\alpha,\xi}$ is defined to be the one induced from $\mathbb{Q}_\alpha$. Note that since inside $\mathbb{Q}_{\alpha,\xi}$ we cannot add elements to the support of $a_0$, in the part above $\xi$ the order is practically the direct extension order. The following claim is a straightforward verification:
\begin{claim}
    The projection to the first coordinate projects $\mathbb{Q}_{\alpha,\xi}$ onto $\mathbb{P}_{\xi}$.
\end{claim}
\begin{corollary}
    $\mathbb{Q}_{\alpha,\xi}$ collapses $\kappa^+$. 
\end{corollary}
The forcing $\mathbb{P}_\xi$ collapses $\kappa^+$ to $\kappa_{\xi+1}=(\kappa_\xi^{++})^{V^{P_\xi}}$. Let us prove next that $\mathbb{Q}_{\alpha,\xi}$ does not collapse this cardinal any further. To do this we need to analyze the quotient forcing. Generally, given a projection $\pi:\mathbb{Q}\rightarrow \mathbb{P}$ let us denote by $\mathbb{Q}/\mathbb{P}$ the quotient forcing\footnote{Altough the quotient forcing depends on the projection, we will usually only have one projection in hand between two given forcings. I case there are several projections, we will denote the quotient by $\mathbb{Q}/_{\pi}\mathbb{P}$.}; that is, a $\mathbb{P}$-name for the forcing $\pi^{-1}[G]$, where $G$ is $V$-generic for $\mathbb{P}$. An idea which will come up several times in this paper is that instead of analyzing the quotient forcing, we will define a variation of the term-space forcing $\mathbb{R}$, so that $\mathbb{P}\times\mathbb{R}$ projects onto $\mathbb{Q}$, and which is much easier to analyze.

For our specific purpose, suppose that $\mathbb{P}_{\bar{E}}$ can be factored into $\mathbb{P}\times\mathbb{R}$, and $\pi_{\alpha,\mathbb{P}}:\mathbb{Q}_{\alpha}\rightarrow \mathbb{P}$ is the restriction of the first coordinate to $\mathbb{P}$ i.e. $\pi_{\alpha,\mathbb{P}}(a)=a_0\restriction \mathbb{P}$. For example, recall that $$\mathbb{S}_\xi=\mathbb{P}_{<\xi}\times Col(\bar{\kappa}_\xi^+,<f_\xi(\kappa_\xi))\times Col(f_\xi(\kappa_\xi),s_\xi(f_\xi(\kappa_\xi))^+)\times Col(s_\xi(f_\xi(\kappa_\xi))^{+3},<\kappa_\xi).$$
The restriction of the first coordinate to $\mathbb{S}_\xi$ is denoted by $\pi_{\alpha,\mathbb{S}_\xi}:\mathbb{Q}_{\alpha}\rightarrow \mathbb{S}_\xi$ projects $\mathbb{Q}_{\alpha}$  onto $\mathbb{S}_\xi$.

If $\mathbb{P}$ happened to be a factor of $\mathbb{P}_\xi$, then $\pi_{\alpha,\mathbb{P}}$  also projects $\mathbb{Q}_{\alpha,\xi}$ onto $\mathbb{P}$.
\begin{definition}\label{definition: the term space forcing}
     Let $\mathbb{P}$ be a factor of $\mathbb{P}_{\bar{E}}$, we denote by $\mathbb{Q}^{\mathbb{P}}_{\alpha}=(\mathbb{Q}_\alpha,\leq_{\mathbb{P}})$ defined by  $q_1\leq_{\mathbb{P}} q_2$ iff $q_1\leq q_2$ and $\pi_{\mathbb{P}}(q_1)=\pi_{\mathbb{P}}(q_2)$. If $\mathbb{P}$ factors $\mathbb{P}_\xi$, we also denote $\mathbb{Q}^{\mathbb{P}}_{\alpha,\xi}=(\mathbb{Q}_{\alpha,\xi},\leq_{\mathbb{P}})$.
 \end{definition}
 Three important examples are: \begin{itemize}
     \item $\mathbb{P}=\mathbb{S}_\xi$.  This forcing factors $\mathbb{P}_\xi$, and we denote $\mathbb{Q}^{\mathbb{S}_\xi}_{\alpha,\xi}=\mathbb{Q}^{+}_{\alpha,\xi}$.
     \item $\mathbb{P}=\mathbb{S}_\xi\times Add(\kappa_{\xi}^+,\kappa^{++})\times Col(\kappa_\xi^+,<\kappa_{\xi+1})$, in this case we already mentioned that $\mathbb{P}_\xi\simeq \mathbb{P}\times \mathbb{T}$ here $\mathbb{T}$ is $\kappa_{\xi+1}^+$-directed.  
     \item $\mathbb{P}=\mathbb{P}_{\bar{E}}$. In this case $\mathbb{Q}^{\mathbb{P}}_{\alpha,\xi}$ is undefined for any $\xi<\kappa$, however, $\mathbb{Q}^{\mathbb{P}}_\alpha$ is defined and denoted by $\mathbb{Q}^+_{\alpha}$. Note that in the ordering of $\mathbb{Q}^+_\alpha$ we fix the first coordinate of conditions and only vary the strategies.
 \end{itemize}
\begin{remark} We say that $\mathbb{S}_{<\xi}$ factors $\mathbb{P}$ if $p\restriction\mathbb{P}=p'\restriction \mathbb{P}$ implies that $p\restriction\mathbb{S}_{<\xi}=p'\restriction\mathbb{S}_{<\xi}$. It is crucial to note that if $\mathbb{S}_{<\xi}$ factors $\mathbb{P}$,  then $\leq_{\mathbb{P}}$ on $\mathbb{Q}^{\mathbb{P}}_{0,\xi}$ implies $\leq^*$. 
\end{remark}   
 \begin{lemma}\label{lemma: Q^* degree of directed}
     Suppose that $\mathbb{P}$ factors $\mathbb{P}_{\xi}$, $\mathbb{S}_{<\xi}$ factors $\mathbb{P}$ and suppose that $\mathbb{R}$ is $\rho$-directed-closed such that $\mathbb{P}_{\xi}\simeq\mathbb{P}\times \mathbb{R}$.  Then $(\mathbb{Q}_{\alpha,\xi}^{\mathbb{P}},\leq_{\mathbb{P}})$ is $\rho$-directed closed.
     \end{lemma}
      \begin{proof}
   By induction on $\alpha$. For $\alpha=0$, $\mathbb{Q}^{\mathbb{P}}_{0,\xi}$ with $\leq_{\mathbb{P}}$ above any condition is isomorphic to $\mathbb{R}$ above some condition which is $\rho$-closed by our assumption.
   For successor $\alpha=\beta+1$, suppose that $\l a_i\mid i<\kappa_\xi\r$ is $\leq_{\mathbb{P}}$-decreasing. 
   Then by definition, $\l q^{a_i}\mid i<\kappa_\xi\r$ is $\leq_{\mathbb{P}}$-decreasing in $\mathbb{Q}^{\mathbb{P}}_{\beta,\xi}$. By the induction hypothesis, 
   there is $q^*$ a $\leq_{\mathbb{P}}$-lower bound. As for the strategies $\mathcal{S}^{a_i}$, we define $\mathcal{S}^*$ as follows:
$l(\mathcal{S}^*)=\sup_{i<\kappa_\xi}l(\mathcal{S}^{a_i})<\kappa^+$. 
Note that by the remark, $(a_i)_0$ is $\leq^*$-decreasing, and since $q^*_0$ is a $\leq^*$-lower bound for the $(a_i)_0$'s, for $i<j<\kappa_\xi$, 
$$Ex^*(q^*_0)\overset{w^{q^*_0}_{(a_j)_0}}{\to} Ex^*((a_j)_0)\overset{w^{(a_j)_0}_{(a_i)_0}}{\to} Ex^*((a_i)_0).$$ 
Moreover, for all $(\vec{\mu},c)\in Ex^*((a_j)_0)$, and all $\rho\leq l(\mathcal{S}^{a_i})$, $S^{a_j}_\rho(\vec{\mu},c)=S^{a_i}_{\rho}(w^{(a_j)_0}_{(a_i)_0}(\vec{\mu},c))$. 
Also by commutativity, for each $(\vec{m},c)\in Ex^*(q^*_0)$, $$w^{(a_j)_0}_{(a_i)_0}(w^{q^*_0}_{(a_j)_0}(\vec{\mu},c))=w^{q^*_0}_{(a_i)_0}(\vec{\mu},c).$$ 
Hence we can define $S^*_\rho(\vec{\mu},c)=S^{a_i}_\rho(\vec{\mu},c)$ for some $i$ such that $\rho\leq l(\mathcal{S}^{a_i})$, and this is well-defined. 
If $l(\mathcal{S}^*)$ is not yet defined, it is a limit ordinal and there is exactly one way we can define it (by the definition of strategies). The argument of Claim \ref{Claim: frown yield a condition} shows that $\mathcal{S}^*$ is a $q^*$-strategy and therefore
$a^*=(q^*,\mathcal{S}^*)$ is a condition which is not hard to see is a $\leq_{\mathbb{P}}$-lower bound for the sequence of $a_i$'s. As for condition $(\theta)$, 
since each $S^*_\rho$ is just a copy of $S^{a_i}_\rho$ (with shifted domain to $q_0^*$), any $\xi$ that witnesses $(\theta)$ 
for $S^{a_i}_\rho$ will also work for $S^*_\rho$. The reason that $S^*_{l(S^*)}$ satisfy condition $(\rho)$, is that $a_i$ all satisfy condition 
$(\rho)$ and Lemma \ref{Lemma: some properties of R}(2) applied to $\{\max(S^{a_i}_{l(S^{a_i})}(w^{q^*_0}_{a_i}(\vec{\mu},c))\mid i<\kappa_\xi\r$.
 For limit $\alpha$, suppose that $\l a_i\mid i<\kappa_\xi\r$ is $\leq_{\mathbb{P}}$-decreasing. Let us construct $a^*$ by induction of $\beta\leq\alpha$.
    We take $a^*_0$ any $\leq_{\mathbb{P}}$-bound for the sequence $(a_i)_0$ (which is in particular a $\leq^*$-bound). At successor steps, we already know that $a^*\restriction \beta$ is a lower bound for $a_i\restriction \beta$ such that $\supp(a^*\restriction\beta)=\bigcup_{i<\kappa_\xi}\supp(a_i\restriction \beta)$ and we define a lower bound as in the previous argument $a^*\restriction\beta+1$ for $a_i\restriction\beta+1$. Note that we will still preserve $\supp(a^*\restriction\beta+1)=\bigcup_{i<\kappa_\xi}\supp(a_i\restriction \beta+1)$ (which adds at most one more coordinate to the support). For limit $\beta$, we simply take $a^*\restriction\beta=\bigcup_{\delta<\beta}a^*\restriction \delta$. Note that $$\supp(a^*\restriction\beta)=\bigcup_{\delta<\beta}\supp(a^*\restriction \delta)=\bigcup_{\delta<\beta}(\bigcup_{i<\kappa_\xi}\supp(p_i\restriction \delta))=$$
    $$=\bigcup_{i<\kappa_\xi}(\bigcup_{\delta<\beta}\supp(p_i\restriction \delta))=\bigcup_{i<\kappa_\xi}\supp(p_i\restriction \beta)$$
    It is clear that the right-most-hand side is of cardinality at most $\kappa$. Hence $a^*\restriction\beta$ is a legitimate lower bound.
\end{proof}

 \begin{corollary}\label{Cor: specific directed degree of Q*} \begin{enumerate}
     \item $(\mathbb{Q}_{\alpha,\xi}^+,\leq_\xi)$ is $\kappa_\xi^+$-directed closed.
     \item For $\mathbb{P}=\mathbb{S}_\xi\times Add(\kappa_{\xi}^+,\kappa^{++})\times Col(\kappa_\xi^+,<\kappa_{\xi+1})$, $\mathbb{Q}^{\mathbb{P}}_{\alpha,\xi}$ is $\kappa_{\xi+1}^+$-directed closed.
 \end{enumerate}
 \end{corollary}
 A similar argument can be used to prove that 
 \begin{lemma}\label{lemma: closure of quotient}
Suppose that $\mathbb{P}$ factors $\mathbb{P}_{\bar{E}}$, and suppose that $\mathbb{R}$ is $\rho$-directed-closed such that $\mathbb{P}_{\bar{E}}\simeq\mathbb{P}\times \mathbb{R}$.  Then $(\mathbb{Q}_{\alpha}^{\mathbb{P}},\leq_{\mathbb{P}})$ is $\rho$-directed closed.
      In particular,
$(\mathbb{Q}^+_\alpha,\leq)$ is $\kappa^+$-directed closed.
 \end{lemma}
Next we would like to show that 
if $\mathbb{P}$ factors $\mathbb{P}_\xi$, then $\mathbb{P}\times \mathbb{Q}^{\mathbb{P}}_{\alpha,\xi}$ projects onto $\mathbb{Q}_{\alpha,\xi}$.
\begin{definition}\label{def: term space forcing projection}
   Suppose that $\mathbb{P}$ factors $\mathbb{P}_\xi$, and fix a condition $a_*\in \mathbb{Q}_{\alpha,\xi}$. For $(s,a)\in \mathbb{P}/\pi_{\mathbb{P}}(a_*)\times \mathbb{Q}^{\mathbb{P}}_{\alpha,\xi}/a_*$, define $a+s$ by induction of $\alpha$:
   \begin{enumerate}
       \item For $\alpha=0$, by assumption $\mathbb{P}_{\xi}\simeq \mathbb{P}\times \mathbb{R}$, and we define
       $$a+s=(a\restriction\mathbb{P},a\restriction \mathbb{R})+s=(s,a\restriction \mathbb{R}).$$
       \item At successor steps 
    $a=(q,\mathcal{S})$, define $a+s=(q+a,\mathcal{S}')$, where $\mathcal{S}'$ is defined as follows:  $l(\mathcal{S}')=l(\mathcal{S})$,  embed\footnote{We will have to prove inductively that $q+s\leq q$.} $Ex^*(q+s)$  into $Ex^*(q)$ using $w^{*(q+s)_0}_{q_0}$, and define $S'_i(\vec{\mu},c)=S_i(w^{*(q+s)_0}_{q_0}(\vec{\mu},c))$.
       \item At limit steps $\alpha$, let $a+s=\bigcup_{\beta<\alpha}(a\restriction \beta)+s$.
   \end{enumerate}
\end{definition}  
\begin{remark}
    the definition of $a+s$ is the same as if we took the minimal extension (in the sense of Remark \ref{remark:minimal}) of $a$ after extending $a_0$ to $a_0+s$.
\end{remark}

Let us provide some properties of this function: 
\begin{proposition} Suppose that $\mathbb{P}$ factors $\mathbb{P}_\xi$, and let $\alpha\leq \kappa^{++}$, and $a_*\in \mathbb{Q}_{\alpha,\xi}$. Then:
\begin{enumerate}
    \item For every $(a,s)\in\mathbb{P}/\pi_{\mathbb{P}}(a_*)\times \mathbb{Q}^{\mathbb{P}}_{\alpha,\xi}/a_*$, $a+s\in \mathbb{Q}_{\alpha,\xi}/a_*$.
    \item $\pi_{\mathbb{P}}(a+s)=s$. 
    \item For any $a\in\mathbb{Q}_{\alpha,\xi}$ and any $s\leq \pi_{\mathbb{P}}(a)$, we have that $a+s$ is the greatest element of the set $\{b\in\mathbb{Q}_{\alpha}\mid b\leq a\wedge \pi_{\mathbb{P}}(b)\leq s\}$.
    \item For every $b\leq a+s$, there is $a'\leq a$ such that $\pi_{\mathbb{P}}(a')=\pi_{\mathbb{P}}(a)$ (namely $a'\leq_\xi a$) and $b=a'+\pi_{\mathbb{P}}(b)$
\end{enumerate}
\end{proposition}
\begin{proof}
    (1) is by induction on $\alpha$. For $\alpha=0$, this is automatic from the factorization. For successor $\alpha$, by the argument of \ref{Claim: frown yield a condition}, $\mathcal{S}'$ is a $(q+s)$-strategy and therefore $a+s\in\mathbb{Q}_{\alpha,\xi}$. At limit steps $\alpha$, we note that for $\beta_1<\beta_2$, $(a\restriction\beta_1)+s$ cohere with $(a\restriction \beta_2)+s$ and therefore  $a+s=\bigcup_{\beta<\alpha}s+(a\restriction \beta)$ is well-defined. Note that for every $\beta$, $\supp((a\restriction \beta)+s)\subset \supp (a\restriction \beta)\cup\{0\}\subseteq \supp(a)\cup\{0\}$, and thus $a+s$ has $\leq\kappa$-support.
     (2) is trivial. For (3), let $b\in\mathbb{Q}_{\alpha}$ such that $b\leq a$ and  $\pi_{\mathbb{P}}(b)\leq s$. We prove that $b\leq a+s$ by induction on $\alpha$. For $\alpha=0$, since $b\leq a$, we have in particular that $b\restriction \mathbb{R}\leq a\leq \mathbb{R}$, where $\mathbb{P}_\xi\simeq\mathbb{P}\times \mathbb{R}$. Since $b\restriction \mathbb{P}=\pi_{\mathbb{P}}(b)\leq s$, we conclude that $b=(b\restriction \mathbb{P},b\restriction \mathbb{R})\leq (s, a\restriction \mathbb{R})=a+s$. At successor steps, we have by the induction hypothesis that $q^b\leq q^a+s=q^{a+s}$, moreover, for every $(\vec{\mu},c)\in Ex^*(b_0)$, we have that  $w^{*b_0}_{a_0}(\vec{\mu},c)=w^{*(a+s)_0}_{a_0}(w^{*b_0}_{(a+s)_0}(\vec{\mu},c))$. Hence, for every $\rho\leq l(\mathcal{S}^a)=l(\mathcal{S}^{a+s})$, $$S^b_\rho(\vec{\mu},c)=S^a(w^{*b_0}_{a_0}(\vec{\mu},c))=S^a(w^{*(a+s)_0}_{a_0}(w^{*b_0}_{(a+s)_0}(\vec{\mu},c)))=S^{a+s}_\rho(w^{*b_0}_{(a+s)_0}(\vec{\mu},c)).$$
     We conclude that $b\leq a+s$. At limit steps $\alpha$, since $b\leq a$ then for every 
     $\delta<\alpha$, $b\restriction \delta\leq a\restriction \delta$ and by the induction hypothesis $b\restriction \delta\leq a\restriction \delta+s=(a+s)\restriction \delta$.
     Hence $b\leq a+s$. $(4)$ is again by induction on $\alpha$. For $\alpha=0$, this is easy. At limit step $\alpha$, suppose that $b\leq a+s$, then for every $\delta<\alpha$, $b\restriction \delta\leq (a+s)\restriction \delta=a\restriction \delta+s$.
     By the induction hypothesis, there is $a'_\delta\leq a\restriction \delta$ such that $\pi_{\mathbb{P}}(a'_\delta)=\pi_{\mathbb{P}}(a)$ and $b\restriction\delta= a'_\delta+\pi_{\mathbb{P}}(b)$.
     Note that for every $\delta<\alpha$, $(a'_\delta)_0=(\pi_{\mathbb{P}}(a),b_0\restriction \mathbb{R})$, and therefore the domain of the block forming the strategies are identical. Since $b\restriction\delta=a'_\delta+\pi_{\mathbb{P}}(b)$, 
     the length of each strategy is determined by the strategies of $b$, below the length of the strategies of $a$, the blocks are determined uniquely by $a'_\delta\leq a\restriction \delta$, but for $\rho$ above it (but still below the length of the strategy of $b$), for $(\vec{\mu},c)$ in the image of $w^{*b_0}_{(a'_\delta)_0}$, 
     the  clubs in each block are determined by $b$, but there is some degree of freedom for other $(\vec{\mu},c)$. So when we define $a'$, we can just copy the strategy from some (any) $a'_\delta$, this will not effect the order. We will still have that 
     $a'\restriction \delta\leq a\restriction \delta$ (as this part is constant for all $\delta$), we will have that $a'\restriction\delta+\pi_{\mathbb{P}}(b)=a'_\delta+\pi_{\mathbb{P}}(b)=b\restriction\delta$,  again, since in this condition we remove all the part of $a'$ which is chosen arbitrary.
     Hence $a'+\pi_{\mathbb{P}}(b)=b$. In the successor step the same ideas repeat themselves in copying the information from $b$ and $a$ wherever we can and completing arbitrarily. The arbitrary part will anyway be removed when passing to $a'+\pi_{\mathbb{P}}(b)$. 
\end{proof}
It follows from the above proposition that $(s,a)\mapsto a+s$ is a projection of $\mathbb{P}/\pi_{\mathbb{P}}(a_*)\times \mathbb{Q}^{\mathbb{P}}_{\alpha,\xi}/a_*$ onto $\mathbb{Q}_{\alpha,\xi}$.

     \begin{corollary}
     $ V^{\mathbb{Q}_{\alpha,\xi}}\models |(\kappa^+)^V|=\kappa_\xi^{++}=\kappa_{\xi+1}$.
 \end{corollary}
 \begin{proof}
    Let $\mathbb{P}=\mathbb{S}_\xi\times Add(\kappa_{\xi}^+,\kappa^{++})\times Col(\kappa_\xi^+,<\kappa_{\xi+1})$. We have seen that $\mathbb{P}\times \mathbb{Q}_{\alpha,\xi}^{\mathbb{P}}$ projects onto $\mathbb{Q}_{\alpha,\xi}$ which in turn projects onto $(\mathbb{P}_{\vec{E}})_\xi$. Thus we have that $V^{(\mathbb{P}_{\vec{E}})_\xi}\subseteq V^{\mathbb{Q}_{\alpha,\xi}}\subseteq V^{\mathbb{P}\times \mathbb{Q}^{\mathbb{P}}_{\alpha,\xi}}$
    . Since $V^{(\mathbb{P}_{\vec{E}})_\xi}\models |(\kappa^+)^V|=\kappa_\xi^{++}=\kappa_{\xi+1}$, we just have to argue that $\mathbb{Q}_{\alpha,\xi}$ preserves $\kappa_\xi,\kappa_\xi^+$ and $\kappa_{\xi+1}$, but this we be true once we prove that these cardinals are preserved by $\mathbb{P}\times \mathbb{Q}^{\mathbb{P}}_{\alpha,\xi}$. Indeed, by corollary \ref{Cor: specific directed degree of Q*} $\mathbb{Q}^{\mathbb{P}}_{\alpha,\xi}$ is $\kappa^+_{\xi+1}$-directed. Finally, by properly factoring $\mathbb{P}$, and the Easton lemma we see that $\kappa_\xi,\kappa_\xi^+$ and $\kappa_{\xi+1}$ are preserved.
\end{proof}
\section{Killing non-reflecting stationary sets}\label{Section: Killing stationary sets}
In this subsection, we will prove that the successor steps of the iteration kills the intended fragile stationary set.
Let us check some properties of $\mathbb{Q}_{\beta+1}=\mathbb{A}(\mathbb{Q}_\beta,\dot{T})$.
Suppose that $G$ is generic for $\mathbb{Q}_{\beta+1}$, and let $G_0=\{(a)_0\mid a\in G\}$. Then clearly $G_0$ is generic for $\mathbb{P}_{\vec{E}}$. For any $a\in G$ and every $i\leq l(\mathcal{S}^a)$, let
$$\mathcal{C}_i(a)=\overline{\bigcup_{(\vec{\mu},c)\in Ex^*(a), (a)_0^\smallfrown(\vec{\mu},c)\in G_0}\mathcal{S}^a_i(\vec{\mu},c)}.$$
$$\mathcal{C}(a)=\mathcal{C}_{l(\mathcal{S}^a)}(a)$$
$$\mathcal{C}=\bigcup_{a\in G}\mathcal{C}(a)$$
   \begin{claim}
       Suppose that $i\leq j\leq l(\mathcal{S}^a)$ then $\mathcal{C}_i(a)\sqsubseteq \mathcal{C}_j(a)$. 
   \end{claim}
   \begin{proof}
        By induction on 
 $j$ we prove that for all $j\geq i$, $\mathcal{C}_i(a)\sqsubseteq \mathcal{C}_j(a)$. At successor step, suppose that $j=i+1$, by the induction hypothesis it suffices to prove that $\mathcal{C}_i(a)\sqsubseteq \mathcal{C}_j(a)$.
       Towards this, it suffices to prove that $$\bigcup_{(\vec{\mu},c)\in Ex^*(a), \ a_0^\smallfrown(\vec{\mu},c)\in G_0}\mathcal{S}^a_i(\vec{\mu},c)\sqsubseteq\bigcup_{(\vec{\mu},c)\in Ex^*(a), \ a_0^\smallfrown(\vec{\mu},c)\in G_0}\mathcal{S}^a_j(\vec{\mu},c).$$  Let\footnote{The inclusion $\mathcal{C}_i(a)\subseteq \mathcal{C}_j(a)$ follows from the definition of strategies.} $$\gamma\in \max(\mathcal{C}_i(a))\cap\Big( \bigcup_{(\vec{\mu},c)\in Ex^*(a), \ a_0^\smallfrown(\vec{\mu},c)\in G_0}\mathcal{S}^a_j(\vec{\mu},c)\Big).$$ In particular $\gamma\in \mathcal{S}^a_{i+1}(\vec{\mu},c)\cap\max(\mathcal{C}_i(a))$ for some $(\vec{\mu},c)$ such that $a_0^\smallfrown(\vec{\mu},c)\in G_0$. Let $(\vec{\mu},c)\geq (\vec{\nu},d)$ be such that $a_0^\smallfrown(\vec{\nu},d)\in G_0$ and $\gamma<\max(\mathcal{S}^a_i(\vec{\nu},d))$, then $\gamma\in S_{i+1}(\vec{\nu},d)$ (since we have inclusion as we increase the pair). But then $\gamma\in \mathcal{S}^a_{i+1}(\vec{\nu},d)\cap\max(\mathcal{S}^a_i(\vec{\nu},d))=\mathcal{S}^a_i(\vec{\nu},d)$. It follows that $$\gamma\in\bigcup_{(\vec{\mu},c)\in Ex^*(a), \ a_0^\smallfrown(\vec{\mu},c)\in G_0}\mathcal{S}^a_i(\vec{\mu},c)$$
       as desired. At limit steps $j$, we have $$\mathcal{C}_j(a)=\overline{\bigcup_{(\vec{\mu},c)\in Ex^*(a), \ a_0^\smallfrown(\vec{\mu},c)\in G_0} \mathcal{S}^a_j(\vec{\mu},c)}=\overline{\bigcup_{(\vec{\mu},c)\in Ex^*(a), \ a_0^\smallfrown(\vec{\mu},c)\in G_0}(\overline{\bigcup_{i<j}\mathcal{S}^a_i(\vec{\mu},c)})}$$
       $$=\overline{\bigcup_{i<j}\overline{\bigcup_{(\vec{\mu},c)\in Ex^*(a), \ a_0^\smallfrown(\vec{\mu},c)\in G_0}\mathcal{S}_i(\vec{\mu},c)}}=\overline{\bigcup_{i<j}\mathcal{C}_i(a)}$$ 
       By the induction hypothesis, the union on the righthand side is increasing with respect to end-extension. Hence for each $i<j$, $\mathcal{C}_i(a)\sqsubseteq \mathcal{C}_j(a)$. The ``in particular" part follows easily from the first part. 
   \end{proof}
   \begin{proposition}
       Suppose that $a\leq b$, $a\in G$. Then for each $i\leq l(\mathcal{S}^b)$, $\mathcal{C}_i(a)=\mathcal{C}_i(b)$. 
   \end{proposition}
   \begin{proof}
       Again it suffices to prove the equality of the unions before taking the closure. First we note that if $a_0^{\smallfrown}(c,\vec{\mu})\in G_0$, then since $$a_0^\smallfrown (\vec{\mu},c)\leq b_0^\smallfrown w^{*a_0}_{b_0}(\vec{\mu},c),$$ we also have $b_0^\smallfrown w^{*a_0}_{b_0}(\vec{\mu},c)\in G_0$. Hence  $w^{*a_0}_{b_0}(\vec{\mu},c)$ participates in the union on the $b$-side. By definition of the order, we have that $\mathcal{S}^a_i(\vec{\mu},c)=\mathcal{S}^b_i(w^{*a_0}_{b_0}(\vec{\mu},c))$, hence the union on the $a$-side is included in the union on the $b$-side. As for the other inclusion, each $(\vec{\mu},c)\in Ex^*(b)$ such that $b_0^\smallfrown(\vec{\mu},c)\in G_0$, is weaker than an element of the form $w^{*a_0}_{b_0}(\vec{\nu},d)$ for some $(\vec{\nu},d)\in Ex^*(a)$ (since both $a_0,b_0^\smallfrown(\vec{\mu},c)\in G_0$ then we can use Corollary \ref{Cor: properties of step extension}). By definition this implies that $$\mathcal{S}^b(\vec{\mu},c)\subseteq \mathcal{S}^b(w^{*a_0}_{b_0}(\vec{\nu},d))=\mathcal{S}^a(\vec{\nu},d).$$
       We conclude that the union on the $b$-side is included in the union on the $a$-side.
   \end{proof}
   It follows that:
   \begin{corollary}\label{Cor: C has union}
    Let $I=\{i<\kappa^+\mid \exists a\in G, \ l(\mathcal{S}^a)=i\}$ and for each $i\in I$ let $a_i\in G$ witness this. Then $\mathcal{C}=\bigcup_{i\in I}\mathcal{C}(a_i)$.  
   \end{corollary}
   To see that $\mathcal{C}$ is unbounded in $\kappa^+$, we will need the following lemma which is an analog of \cite[Claim 6.3.1]{SigmaPrikry3}. It is conditioned on the pre-strong Prikry Lemma (see definition \ref{dif:Pre-strong Prikry lemma}) which is proven in the next section:
   \begin{lemma}\label{Lemma: lemma involved in induction with pre-strong Prikry} Assume the pre-Prikry Lemma for $\mathbb{Q}_\beta$.
      For any $a\in\mathbb{Q}_{\beta}$ and any $\gamma<\kappa^+$ there is $a^*\leq^* a$ such that $a^*\restriction \mathbb{S}_{\xi(a)}=a\restriction \mathbb{S}_{\xi(a)}$ and $\gamma<\gamma'<\kappa^+$ such that for any $b\leq a^*$, $b\Vdash_{(\mathbb{Q}_{\beta})_{\xi(b)}}\dot{C}_{\xi(b)}\cap (\gamma,\gamma')\neq\emptyset$.
   \end{lemma}
   \begin{proof} 
Let
$$D=\{q\in\mathbb{Q}_{\beta}\mid \exists \gamma'<\kappa^+\  q\Vdash_{\mathbb{Q}_{\beta,\xi(q)}}\dot{C}_{\xi(q)}\cap (\gamma,\gamma')\neq\emptyset\}$$
Clearly, it is $\leq^*$-open and for all $\xi\geq\xi(a)$, $D\cap( \mathbb{Q}_{\beta})_{\xi}$ is dense in $\mathbb{Q}_{\beta,\xi}$.

Let $\xi=\xi(a)$. For every $l\leq_{\mathbb{S}_{\xi}}a\restriction \mathbb{S}_\xi$, and $s\in [\omega_1]^{<\omega}$ we apply the pre-strong-Prikry lemma to find a direct extension $a_{l,s}\leq^* l^\smallfrown a\restriction \mathbb{S}_{>\xi}$ (here $l^\smallfrown a\restriction \mathbb{S}_{>\xi}$ is just $a+l$ from the previous section)
such that:
\begin{enumerate}
    \item Either for every $\vec{\mu}$ with $s= \dom(\vec{\mu})$, $a_{s,l}^\smallfrown\vec{\mu}\in D$; or
    \item For every $b\leq a_{s,l}$ with $\supp(b)=\supp(a)\cup s$, $b\notin D$.
\end{enumerate}
If $s,l$ are as in case $(1)$, for each $\vec{\mu}$ pick $\gamma_{\vec{\mu},l}$ witnessing that $a_{s,l}^\smallfrown\vec{\mu}\in D$.
Now since there are only $\kappa_\xi$-many pair $(s,l)$, there is $(a^*)_{>\xi}\leq^* a_{s,l}\restriction \mathbb{S}_{>\xi}$ for every $s,l$ (here we use the fact that $GCH$ holds in the ground model and that $\mathbb{S}_{\xi>}$ is $\kappa_\xi^+$-closed with respect to $\leq^*$ appealing to Lemma \ref{lemma: closure of quotient}). Let $a^*=(a\restriction \mathbb{S}_\xi)^\smallfrown (a^*)_{>\xi}$, and take $\gamma^*=\sup_{\vec{\mu}}\gamma_{\vec{\mu}}<\kappa^+$.  Let us prove that $a^*$ and $\gamma^*$ are as desired. Otherwise, let $b\leq a^*$, if 
$b\not\Vdash_{\mathbb{Q}_{\beta,\xi(b)}}\dot{C}_{\xi(b)}\cap (\gamma,\gamma^*)\neq\emptyset$, there is $r\leq_{\xi(b)}b$ such that $r\Vdash_{\mathbb{Q}_{\beta,\xi(b)}}\dot{C}_{\xi(b)}\cap (\gamma,\gamma^*)=\emptyset$. There is some $r'\leq_{\xi(b)} r$ such that $r'\in D$.
Denote by $l=r'\restriction \mathbb{S}_\xi$. Since $r'\leq a^*$, we have that $r',a_{s,l}$ are compatible for some $s$. 
Thus $(1)$ must hold for $a_{s,l}$. By the definition of $\gamma_{l,\vec{\mu}}$, there will be an $r''\leq r'$ such that $$r''\Vdash_{(\mathbb{Q}_{\beta})_{\xi(b)}} \emptyset\neq\dot{C}_{\xi(b)}\cap(\gamma,\gamma_{l,\vec{\mu}})\subseteq  \dot{C}_{\xi(b)}\cap(\gamma,\gamma^*)=\emptyset.$$
Contradiction.
\end{proof}
\begin{corollary}\label{Cor: Entering F above any condition}
    Assume the pre-strong-Prikry lemma for $\mathbb{Q}_\beta$. Then for any $a\in\mathbb{Q}_\beta$ and $\gamma<\kappa^+$ there is $a^*\leq^* a$ with $a\restriction\mathbb{S}_{\xi(a)}=a^*\restriction \mathbb{S}_{\xi(a)}$ and $\gamma\leq \gamma^*<\kappa^+$ such that $\l a^*,\gamma^*\r\in R$. 
\end{corollary}
\begin{proof}
    We use the previous lemma to construct inductively an  increasing sequence of ordinals $\delta_n<\kappa^+$ and a $\leq^*$-decreasing sequence of conditions $a_n$ such that $\delta_0=\max\{\delta,\gamma\}$, $a_0=a$, and for all $n>0$, $a_n\restriction \xi(a)=a\restriction \xi(a)$ and 
   $$\forall b\leq a_n, \ b\Vdash_{\mathbb{Q}_{\beta,\xi(b)}}\dot{C}_{\xi(b)}\cap (\delta_{n-1},\delta_n)\neq\emptyset$$
   Take $a^*\leq^* a_n$ for every $n$ such that $a^*\restriction \xi(a)=a\restriction \xi(a)$. Finally let  $\delta^*=\sup_{n<\omega}\delta_n<\kappa^+$. To see that $\l a^*,\delta^*\r\in R$, let $b\leq a^*$. Then for every $n$, $b\Vdash_{\mathbb{Q}_{\beta,\xi(b)}}\dot{C}_{\xi(b)}\cap (\delta_{n-1},\delta_n)\neq\emptyset$
Hence $b$ forces that $\delta^*$ is a limit point of $\dot{C}_{\xi(b)}$ and therefore $b\Vdash_{\mathbb{Q}_{\beta,\xi(b)}}\delta^*\in\dot{C}_{\xi(b)}$. It follows that $\l a^*,\delta^*\r\in R$.
\end{proof}

   \begin{proposition}\label{Prop: Unbounded} Assume the pre-strong Prikry Lemma for $\mathbb{Q}_\beta$. For every $a=(q^a,\mathcal{S}^a)\in\mathbb{Q}^*_{\beta+1}$ and every $\delta<\kappa^+$ there is $a^*=(q^{a^*},\mathcal{S}^{a^*})\leq^* a$ such that $a^*\restriction \mathbb{S}_{\xi(a)}=a\restriction \mathbb{S}_{\xi(a)}$ and for every $(\vec{\mu},c)\in Ex^*((a^*)_0)$, $\max(S^{a^*}_{l(\mathcal{S}^{a^*})}(\vec{\mu},c))\geq \delta$ and $a^*\in \mathbb{Q}_{\beta+1}$. In particular,
       $\mathcal{C}$ is unbounded.
   \end{proposition}
   \begin{proof}
   Let us use the previous Corollary. Given $a\in\mathbb{Q}_{\beta+1}$, and $\delta<\kappa^+$ let $\gamma=\sup(\bigcup_{(\vec{\mu},c)\in Ex^*(a)}S^a_{l(\mathcal{S}^a)}(\vec{\mu},c))<\kappa^+$.  By the previous corollary, find $\delta^*>\gamma$ and $q^*\leq^* q$ with $q^*\restriction \mathbb{S}_{\xi(q^*)}=q\restriction \mathbb{S}_{\xi(q)}$ such that $\l q^*,\delta^*\r\in R$. Extend $a=(q^a,\mathcal{S}^a)$ to $a^*=(q^*,\mathcal{S}^*)$ where $l(\mathcal{S}^*)=l(\mathcal{S})+1$, for every $i\leq l(\mathcal{S})$, $\mathcal{S}^*_i$ is determined from $\mathcal{S}^{a^*}_i$ and the fact that $a^*_0\leq^* a_0$ and $\mathcal{S}^*_{l(\mathcal{S}^a)+1}(\vec{\mu},c)=\mathcal{S}^*_{l(\mathcal{S}^a)}(\vec{\mu},c)\cup\{\delta^*\}$. Then $(q^*,\mathcal{S}^*)$ is as desired. Indeed, $\delta^*\in \mathcal{S}^*_{l(\mathcal{S}^*)}(\emptyset)$ and by the choice of $\gamma$, it is way above all the elements of $\mathcal{S}^*_{l(S^a)}(\vec{\mu},c)$ (for all possible $(\vec{\mu},c)$). Hence $a^*\Vdash \delta^*\in \dot{C}(a^*)$.

\end{proof}
\begin{remark}\label{remark: density in successor step}
    The above proof shows that if $\mathbb{Q}_\beta$ satisfies the pre-strong-Prikry Lemma, and $(\mathbb{Q}_{\beta},\leq^*)$ is dense in $(\mathbb{Q}^*_{\beta},\leq^*)$, then $(\mathbb{Q}_{\beta+1},\leq^*)$ is dense in $(\mathbb{Q}^*_{\beta+1},\leq^*)$. This will be used later in our inductive argument to show that also $(\mathbb{Q}_{\beta+1},\leq^*)$ satisfies the Prikry-lemma.
\end{remark}

Finally, let us check that $\mathcal{C}$ is disjoint from the fragile stationary set $(\dot{T})_{G_\beta}$.
\begin{proposition}
    For every $\xi\leq l(S^a)$, $$\mathcal{C}_\xi(a)=\bigcup_{(\vec{\mu},c)\in Ex^*(a), \ a_0^\smallfrown(\vec{\mu},c)\in G_0}\mathcal{S}^a_\xi(\vec{\mu},c)\cup\{\sup_{(\vec{\mu},c)\in Ex^*(a), \ a_0^\smallfrown(\vec{\mu},c)\in G_0}(\max(\mathcal{S}^a_\eta(\vec{\mu},c))\mid \eta\leq\xi\}$$
\end{proposition}
\begin{proof}
    By induction of $\xi$, for $\xi=1$ we have that the $\mathcal{S}^a_1(\vec{\mu},c)$'s form a directed system with respect to $\sqsubseteq$. Therefore the closure of the union can only add the maximal element, as each of the sets $\mathcal{S}^a_1(\vec{\mu},c)$ is itself closed. At successor steps, we use the fact that $C_\xi(a)\sqsubseteq C_{\xi+1}(a)$. At limit $\xi$, this is clear.
\end{proof}
   \begin{lemma}
       Let $G$ be generic for $\mathbb{Q}_{\beta+1}$ and $G_\beta$ be the induced generic for $\mathbb{Q}_\beta$. Then, for every $a\in G$, $\mathcal{C}(a)\cap (\dot{T})_{ G_\beta}=\emptyset$.  
   \end{lemma}
   \begin{proof}
       Let $G_0$ be the induced generic for $\mathbb{P}_{\vec{E}}$. First we note that since for each $(\vec{\mu},c)\in Ex^*(a)$, if $a_0^\smallfrown(\vec{\mu},c) \in G_0$, then $q^{a\smallfrown}(\vec{\mu},c)\in G_\beta$. By condition $(3)$ of a $q^a$-labeled block.  $\mathcal{S}^a_{l(\mathcal{S}^a)}(\vec{\mu},c)$ is disjoint from $(\dot{T})_{G_\beta}$, as forced by $q^{a\smallfrown}(\vec{\mu},c)$. By the previous proposition, the only possible intersection of $\mathcal{C}_{l(\mathcal{S}^a)}(a)$ with $\dot{T}_{G_\beta}$ is a point of the form $$\gamma^*=\sup_{(\vec{\mu},c)\in Ex^*(a), \ a_0^\smallfrown(\vec{\mu},c)\in G_0}(\max(\mathcal{S}^a_\eta(\vec{\mu},c))$$ for some $\eta\leq l(\mathcal{S}^a)$.  
       Suppose towards a contradiction that $q\leq q^a$, $q\in G_\beta$ is a condition such that $q\Vdash_{\mathbb{Q}_\beta} \gamma^*\in \dot{T}$. Let $\xi<\omega_1$ be the delay of $\mathcal{S}_\eta^a$ and without loss of generality suppose that $q$ is past the delay $\xi$.  Hence $q\Vdash_{(\mathbb{Q}_{\beta})_{\xi(q)}}\gamma^*\in \dot{T}_{\xi(q)}$. 
       However, for each $(\vec{\mu},c)\in Ex^*(q)$, by the definition of delay,  $\l\max(\mathcal{S}^a_{\eta}(w^q_{q^a}(\vec{\mu},c)),q\r\in R$. Note that since $q\in G_\beta$, also $$\gamma^*=\sup_{(\vec{\mu},c)\in Ex^*(q), \ q^\smallfrown(\vec{\mu},c)\in G_0}(\max(\mathcal{S}^a_{\eta}(w^q_{q^a}(\vec{\mu},c)))).$$
       Therefore also $\l \gamma^*,q\r\in R$.
       In particular $q\Vdash_{(\mathbb{Q}_\beta)_{\xi(q)}}\gamma^*\in \dot{C}_{\xi(q)}$ 
        contradicting the choice of $\dot{C}_{\xi(q)}$.

      \end{proof}\begin{corollary}
          $\mathcal{C}\cap (\dot{T})_{G_\beta}=\emptyset$.
      \end{corollary}
      \begin{proof} Follows from the previous corollary and Corollary \ref{Cor: C has union}.
      \end{proof}
      \begin{corollary}\label{Cor: killing stationary sets}
    In the extension by $\mathbb{Q}_{\beta+1}=\mathbb{A}(\mathbb{Q}_\beta,\dot{T})$, $\dot{T}$ is forced to be a non-stationary set.
\end{corollary}
\section{The Prikry property and the $\kappa^{++}$-chain condition}\label{section: Prikry}
\begin{proposition}
    For every $\alpha\leq\kappa^{++}$ $\mathbb{Q}_\alpha$ is $\kappa^{++}$-cc. In fact, it is $\kappa^{++}$-Knaster
\end{proposition}
\begin{proof} 

We will actually show a stronger property: that for each $\alpha\leq\kappa^{++}$, there is function  $c:\mathbb{P}_\alpha\rightarrow H(\kappa^+)$ such that if $c(p)=c(q)$, then $p,q$ are $\leq^*$ compatible.

Let us define such a function for $\mathbb{P}_{\bar{E}}$, the first step of the iteration.
By Engelking-Kar\'lowicz theorem,  fix a sequence of functions $\l e_i\mid i < \kappa^{+}\r$ each $e_i$ is from  from $\kappa^{++}$ to $\kappa$ such that, for all $x\in [\kappa^{++}]^\kappa$ and
every function $e:x\to\kappa$, there exists $i<\kappa^+$ with $e \subseteq e_i$. For a condition $p\in\mathbb{P}_{\bar{E}}$ we define $c(p)$ to be a lift of the following pieces of information:
\begin{enumerate}
    \item $\supp(p)$ in $[\omega_1]^{<\omega}$.
    \item $\l h^{0,p}_\xi,h^{1,p}_\xi,h^{2,p}_\xi\mid \xi\in\supp(p)\r$ in $$\prod_{\xi\in\supp(p)}Col(\bar{\kappa}^+_\xi,<f^p_\xi(\kappa_\xi))\times Col(f^p_\xi(\kappa\xi),s_\xi(f^p_\xi(\kappa_\xi))^+)\times Col(s_\xi(f^p_\xi(\kappa_\xi))^{+3},<\kappa_\xi)$$
    \item $\l j_{E_\xi}(H^{0,p}_\xi)(mc_{\xi}(\dom(f^p_\xi))),j_{E_\xi}(H^{1,p}_\xi)(mc_{\xi}(\dom(f^p_\xi)))\mid \xi\notin\supp(p)\r$ in $\prod_{\xi\notin\supp(p)}Col(\bar{\kappa}_\xi^+,<\kappa_\xi)\times Col(\kappa_\xi,\kappa^+)$.
    \item $\l j_{E_\xi}(H^{3,p}_\xi)(\kappa_\xi)\mid \xi\notin\supp(p)\r$ in $$\prod_{\xi\notin\supp(p)}Col^{E_\xi(\kappa_\xi)}(j_{E_\xi(\kappa_\xi)}(s_\xi)(\kappa_\xi)^{+3},<j_{E_\xi(\kappa_\xi)}(\kappa_\xi)).$$
    \item $\l f^p_\xi\mid \xi<\max(\supp(p))\r$ in $Add(\kappa_{\max(\supp(p))}^{++},\kappa_{\max(\supp(p))}^+)^{\max(\supp(p))}$.
    \item $\l i(f^p_\xi)\mid \max(\supp(p))\leq\xi<\omega_1\r\r$ in $(\kappa^+)^{\omega_1}$
\end{enumerate} where $i(f^p_\xi)=\min\{i\mid f^p_\xi\subseteq e_i\}$. Note that since $j_{E_\xi(\kappa_\xi)}$ is the ultrapower by a normal ultrafilter on $\kappa_\xi$,  $j_{E_\xi(\kappa_\xi)}(\kappa_\xi)$ is much smaller than $\kappa$. Hence $c(p)$ can indeed be coded as an element of $H(\kappa^+)$. We claim that if $c(p)=c(q)$ then $p,q$ are $\leq^*$-compatible. To see this, by $(1),(2),(5)$, for every $\xi\in\supp(p)\cap\max(\supp(p))$, $p_\xi=q_\xi$, and for $\xi=\max(\supp(p))$, we use $(2)$ and $(6)$ to see that $\l e_{i(f^p_\xi)}\restriction(\dom(f^p_\xi)\cup\dom(f^q_\xi)), h^{0,p}_\xi, h^{1,p}_\xi,h^{2,p}_\xi\r$ is stronger than both $p_\xi$ and $q_\xi$. As for $\xi\notin \supp(p)$,   Once again, by $(1),(5),(6)$ we have that $e_{i(f^p_\xi)}\restriction(\dom(f^p_\xi)\cup\dom(f^q_\xi))$ is stronger than $f^p_\xi$ and $f^q_\xi$. By $(3),(4)$, we can find a measure one set $A^*_\xi\in E_{\xi}(d)$, where $d=\dom(f^p_\xi)\cup\dom(f^q_\xi)$ such that for every $\vec{\mu}\in A^*_\xi$, and $H^{i,*}_p$ for $i=1,2,3$ such that $H^{i,p}_{\xi}(\vec{\mu}\restriction \dom(f^p_\xi))=H^{i,*}_\xi(\vec{\mu})=H^{i,q}_{\xi}(\vec{\mu}\restriction \dom(f^q_\xi))$. Hence $\l e_{i(f^p_\xi)}\restriction(\dom(f^p_\xi)\cup\dom(f^q_\xi)), H^{1,*}_\xi,H^{2,*}_\xi,H^{3,*}_\xi, A^*_\xi\cap \pi^{-1}_{d,\dom(f^p_\xi)}\cap \pi^{-1}_{d,\dom(f^q_\xi)}\r$ is stronger than both $p_\xi$, $q_\xi$, as wanted.

So, we have the existence of the desired function for the first step of the iteration. For the rest, it follows by induction. For each successor $\alpha$, we define the function $c$ for $\mathbb{Q}_\alpha$ as in definition 4.13 in \cite{SigmaPrikry2}. For limit $\alpha$'s, we define the $c$-function as in subsection 3.1 \cite{SigmaPrikry2}.  Then the same arguments as in \cite{SigmaPrikry2}, Lemma 3.14, more precisely Claim 3.14.1, show that this function is as desired.

\end{proof}
\begin{corollary}
    Cardinals $\theta\geq\kappa^{++}$ are preserved.
\end{corollary}

Next, let us prove the Prikry property for $\mathbb{Q}_\alpha$. First we will show it holds for $\mathbb{Q}^*_\alpha$ and then prove that $\mathbb{Q}_\alpha$ is $\leq^*$-dense in $\mathbb{Q}^*_\alpha$ which would then imply the Prikry property for $\mathbb{Q}_\alpha$.

\begin{definition}
Suppose that $\mathbb{Q}$ is one of the steps of the iteration without delays $\mathbb{P}^*$.
    We say that $\mathbb{Q}$ has \textit{the diagonalization property} if Good has a winning strategy in the following game of length $\kappa_\xi$: At even and limit stages, Good plays conditions $p_{2i}\leq^*p$. At odd stages, Bad plays  a $\xi$-step extension $q_{2i+1}\leq^* p_{2i}^\smallfrown \nu_i$ such that $\nu_i\neq \nu_j$ for all $j<i$. Good wins if the game continues for all $i<\kappa_\xi$ and produces a run $\langle p_{2i}, q_{2i+1}\mid i<\kappa_\xi\rangle$, and $p^*\leq^* p$ for all $i$, such that for every $\mu\in Ex_\xi(p^*)$, if  there is $i<\kappa_\xi$ such that $\mu\restriction \dom(f^{p_{2i}}_\xi)=\nu_i$, then $p^{*\smallfrown}\mu\leq^* q_{2i+1}$. 
    
\end{definition}
\begin{remark}\label{remark: vacously}
    Having played $\l p_{2i},q_{2i+1}\mid i<\kappa_\xi\r$ according to the rules of the game, we can form $D=\bigcup_{i<\kappa_\xi}\dom(f^{p_{2i}}_\xi)$, and let $A^*_\xi=\pi^{-1}_{D,\dom(f^p_\xi)}[A^p_\xi]\in E_\xi(D)$. 
    If $B^*_\xi=\{\nu\in A^*_\xi\mid \neg \exists i<\kappa_\xi, \ \nu\restriction D=\nu_i\}$, then shrink $A^*_\xi$ to $B^*_\xi$, and directly extend $p$ to $p^*$ by enlarging $\dom(f^p_\xi)$ to $D$ and restricting the measure-one set $A^p_\xi$ to $B^*_\xi$. It is clear that $p^*$ vacuously satisfies the diagonalization requirement. So we may assume that $\{\nu\in A^*\mid \exists i<\kappa_\xi, \ \nu\restriction D=\nu_i\}\in E_\xi(D)$.  
\end{remark}
Let us prove that the diagonalization property suffices to prove the pre-Strong Prikry Lemma:

\begin{definition}\label{dif:Pre-strong Prikry lemma}
We say the the pre-strong Prikry lemma holds for $\mathbb{Q}_\alpha$ if for any open with respect to $\leq^*$ subset $D$ of $\mathbb{Q}_\alpha$ and any increasing  $\vec{\xi}\in[ \omega_1]^{<\omega}$ disjoint from the support of $p$, there is a direct extension of $a^*\leq^* a$, such that either all of its $\vec{\xi}$-step extensions are in $D$, or none are. 
\end{definition}

\begin{lemma}\label{Lemma: pre-Prikry Lemma}
Let $\alpha\leq \kappa^{++}$ be any ordinal, and $a\in \mathbb{Q}_\alpha$ be a condition. If $\mathbb{Q}_\alpha$ has the diagonalization property, then $\mathbb{Q}_\alpha$ satisfy the pre-strong Prikry Lemma. 
\end{lemma}
\begin{proof}
By induction on $|\vec{\xi}|$. For $\vec{\xi}=\emptyset$, this is trivial (and does not require the diagonalization property). Suppose this is true for $n$ and let $\vec{\xi}=\l \xi_1,...,\xi_{n+1}\r$. Let  $\phi:\kappa_\xi\rightarrow\kappa_\xi\times \kappa_\xi$ be a bijection such that $\pi_1(\phi(i))\leq i$, enumerate $Ex_{\xi_1}((a)_0)$ by $\{\nu_i\mid i<\kappa_\xi\}$. We perform a run of the diagonalization game $\l b_{2i+1},a_{2i}\mid i<\kappa_\xi\r$ as follows. At stage $0$,  we set $a_0=a$ and $\mu_{\phi^{-1}(0,i)}=\nu_i$, and let $b_1\leq^* a^\smallfrown \mu_0$ be such that one of the following holds:  \begin{enumerate}
    \item [$(\star_{\mu_0})$]  $\forall\vec{\nu}\in Ex_{\vec{\xi}\setminus \{\xi_1\}}(b_1),\  b_1^\smallfrown \vec{\nu}\in D$.
    \item [$(\dagger_{\mu_0})$]  $\forall\vec{\nu}\in Ex_{\vec{\xi}\setminus \{\xi_1\}}(b_1)\text{ and all }b_1^\smallfrown\vec{\nu}\geq^* b, \ b\notin D.$ 
\end{enumerate} The existence of $b_1$ is justified by the induction hypothesis to the condition $a_0^\smallfrown\mu_0$. The winning strategy for the game at $\mathbb{Q}_\alpha$ gives $a_2\leq^* a_0$. Suppose we have defined $\l a_{2i},b_{2i+1}\mid i<j\r$. The winning strategy defines $a_{2j}$. Enumerate $Ex_\xi(a_{2j})\setminus \bigcup_{i<j}Ex_\xi(a_{2i})$ by $\{\nu_r\mid r<\kappa_\xi\}$ and set $\mu_{\pi^{-1}(j,r)}=\nu_r$. By the assumption about $\phi$, $\phi(j)$ is defined, if $\mu_j\notin Ex_{\xi}(a_{2j})$ such that $\nu\restriction \dom(f^{a_{2\cdot\pi_1(\phi(r))}}_{\xi})=\mu_r$, we do something trivial. Otherwise, we fix such a $\nu$, and find $b_{2j+1}\leq^* a_{2j}^\smallfrown \nu$ such that 
one of the following holds:  \begin{enumerate}
    \item [$(\star_{\mu_r})$]  $\forall\vec{\nu}\in Ex_{\vec{\xi}\setminus \{\xi_1\}}(b_{2j+1}),\  b_{2j+1}^\smallfrown \vec{\nu}\in D$.
    \item [$(\dagger_{\mu_r})$]  $\forall\vec{\nu}\in Ex_{\vec{\xi}\setminus \{\xi_1\}}(b_{2j+1})\text{ and all }b\leq^*  b_{2j+1}^\smallfrown\vec{\nu}, \ b\notin D.$ 
\end{enumerate}
Once again, the existence of $b_{2j+1}$ is justified by the induction hypothesis.
By the winning strategy, fix a condition $a^*\leq^* a$, which diagonalizes the $b_{2i+1}$'s. Denote by $S=\bigcup_{i<\kappa_\xi}\dom(f^{a_{2i}}_\xi)$ and let $\nu\in Ex_\xi(a^*)$ be any object and $\mu=\nu\restriction S$. Since $|\mu|<\kappa_\xi$, there is $i<\kappa_\xi$ such that $\dom(\mu)\subseteq \dom(f^{a_{2i}}_\xi)$. Since $a^*\leq^* a_{2i}$, $\nu=\mu\restriction \dom(f^{a_{2i}}_\xi)\in Ex_\xi(a_{2i})$. Hence $\nu$ has been enumerated as some $\mu_r$ for some $r\geq i$. It also follows that $\mu\restriction \dom(f^{a_{2r}}_\xi)=\nu\in Ex(a_{2r})$. By definition of $a^*$, it follows that $a^{*\smallfrown}\mu\leq^*b_{2r+1}$.  Note that for all $\mu^\smallfrown \vec{\mu}\in Ex_{\vec{\xi}}(a^*)$, $a^{*\smallfrown}\mu\leq^* b_i$ for some $i$ and thus $$a^{*\smallfrown}\mu^\smallfrown \vec{\mu}\leq^* b_i^\smallfrown w^{a^{*\smallfrown}\mu}_{b_i}(\vec{\mu}).$$
Hence $(\star_{\mu})$ and $(\dagger_\mu)$ are preserved when moving from $b_i$ to $a^{*\smallfrown}\nu$.
Finally, we shrink the measure-one set  $A^{a^*}_\xi\in Ex_{\xi}(a^*)$ to get a condition $a^{**}\leq^* a^*$ such that either for any $\nu_i \in Ex_\xi(a^{**})$, $(\star_i)$ holds, or for any $\nu_i\in Ex_\xi(a^{**})$, $(\dagger_i)$ holds. The condition $a^{**}$ is as wanted. 
\end{proof}
\begin{lemma}
    $\mathbb{Q}_1=\mathbb{Q}_1^*=\mathbb{P}_{\bar{E}}$ has the strategic-diagonalization porperty.
\end{lemma}
\begin{proof}
Fix a coordinate $\xi>\xi(p)$. For each $i$ we will construct conditions $p_{2i}, q_{2i+1}$, such that $p_{2i}\leq^* p$ is a $\leq^*$-decreasing sequence in $\mathbb{P}_{\bar{E}}$, and  $q_{2i+1}\leq^*p_{2i}^{\smallfrown}\nu_i$, such that $p_{2i}\restriction \xi=p\restriction\xi$. 
At successor steps $2i+2$, assume we have defined $p_{2i}$ and Bad has played $q_{2i+1}\leq^* p_{2i}^\smallfrown \nu_i$, let us define $p_{2i+2}$. First we set $p_{2i+2}\restriction \xi=p_{2i}\restriction \xi$ and $p_{2i+2}\restriction (\xi+1,\omega_1)=q_{2i+1}\restriction(\xi+1,\omega_1)$. Define $p_{2i+2,\xi}=\l f^{2i+2}_\xi, A^{2i+2}_\xi,H^{2i+2}_{0,\xi},H^{2i+2}_{1,\xi},H^{2i+2}_{2,\xi}\r$ as follows.  Note that $f^{q_{2i+1}}_\xi\supseteq f^{p_{2i}}_\xi+\nu_i$; let $$f^{p_{2i+2}}_\xi=\big(f^{q_{2i+1}}_\xi\restriction [\dom(f^{q_{2i+1}}_\xi)\setminus \dom(\nu_i)]\big)\cup f^{p_{2i}}_\xi\restriction \dom(\nu_i)$$
so $f^{p_{2i+2}}_\xi\supseteq f^{p_{2i}}_\xi$. Let $A^{2i+2}=\{\mu\mid \mu\restriction \dom(f^{p_{2i}}_\xi)\in A^{2i}\}$ be the canonical preimage of $A^{2i}$ of the projection of $E_\xi(\dom(f^{p_{2i+2}}_\xi))$ to $E_\xi(\dom(f^{p_{2i}}))$. $H^{2i+2}_{i,\xi}$ for $i=0,1,2$ is the canonical shift of $H^{2i}_{i,\xi}$ to the new domain except for $i=2$ which we add one more change at coordinate $\nu_i(\kappa_\xi)$ on which we define $H^{2i+2}_{2,\xi}(\nu_i(\kappa_\xi))=h^{q_{2i+1}}_{2,\xi}$.

For limit $i$, define $p^{2i}$ as follows, below $\xi$ we keep $p\restriction\xi$.  $p_{2i}\restriction (\xi+1,\omega_1)$ is a $\leq^*$-lower bound for the $p_{2j}\restriction (\xi+1,\omega_1)$ for $j<i$ and let us define $p_{2i,\xi}$. $f^{p_{2i}}_\xi=\bigcup_{j<i}f^{p_{2j}}_\xi$ and $A^{2i}$, $H^{2i}_{0,\xi},H^{2i}_{1,\xi}$ are the shifts to the new domain. Finally, set $$H^{2i}_{2,\xi}(\bar{\nu})=\bigcup_{j<i, \nu_j(\kappa_\xi)=\bar{\nu}}h^{q_{2j+1}}_\xi$$
Note that by Proposition \ref{Prop: Bound the number of same projection to normal}, there are at most $s_\xi(\bar{\nu})^{++}$-many $j<i$ such that $\nu_j(\kappa_\xi)=\bar{\nu}$, and therefore we have sufficient closure to ensure that $H^2_{2,\xi}(\bar{\nu})\in Col(s_\xi(\bar{\nu})^{+3},<\kappa_\xi)$. 

We are ready to define the condition $p'\leq^*p$ which will diagonalize the $q_{2i+1}$'s. First, let $f^{p'}_\xi=\bigcup_{i<\kappa_\xi}f^{p_{2i}}_\xi$.
each of the $q_{2i+1}$'s can be factored into a conditions in the product $$\mathbb{S}_{<\xi}\times Col(\underset{\tau}{\underbrace{s_\xi(f_\xi(\kappa_\xi))^{+3}}}, <\kappa_\xi)\times \mathbb{S}_{>\xi},$$ and recall that
 $\mathbb{S}_{<\xi}$ has size less $s_\xi(f_\xi(\kappa_\xi))^{+}$ and  $\mathbb{S}_{\xi}$ is $\kappa_\xi^+$-closed with respect to $\leq^*$. Note that inductively we maintained that the $\mathbb{S}_{>\xi}$ -part of the $q_{2i+1}$'s is decreasing. By passing to a measure one set of $E_\xi(\dom(f_\xi^{p'}))$  we can find a set $B$ such that either no $\nu\in B$ satisfy  $\nu\restriction D=\nu_i$ for $i<\kappa_\xi$ or for every $\nu\in B$ there is $i<\kappa_\xi$ such that $\nu\restriction D=\nu_i$ (and this $i$ is unique since $\nu_i\neq \nu_j$ for $i\neq j$). In the case where no $\nu\in B$ restricts on $D$ to some $\nu_i$, we set $A'_\xi=B$ and copy the information for the strategies from $a$ (in this case, the resulting condition vacuously satisfies the requirement of $a^*$). Otherwise, we can further shrink $B$ to $A'_\xi\in \dom(f^{p'}_\xi)$ so that for each $\nu=\nu_i\in A'_\xi$, $q_{2i+1}\restriction \mathbb{S}_{<\xi}$ is constantly $l$, and $A'_\xi:=\{\nu_i\mid i\in I\}\in E_\xi(\dom(f_\xi^{p'}))$. Shrink further so that for each $\mu\in A'_\xi$, and each $i<\kappa_\xi$, $\mu\restriction\dom(f^{p_{2i}}_\xi)\in A^{p_{2i}}_\xi$. This is possible as we maintained that measure one sets were just shifts to the new domains.

Note that $l$ is not a legitimate lower part for a condition $p'$ with $\xi$ not in his support, since the domain of the functions below $\xi$ in $l$ have been squished by some $\mu$ to have domain $s_\xi(\mu(\kappa_\xi))^{++}$. Letting $\xi^*=\max(\supp(p)\cap\xi\cup\{0\})$, set  $p'\restriction\xi^*=l\restriction \xi^*$, also $p'\restriction(\xi+1,\omega_1)$ be below each $q_{2i+1}\restriction (\xi+1,\omega_1)$. for each $r\in[\xi^*,\xi)$, we have that $l_r=\l f_r,A_r,H_{r,0},H_{r,1},H_{r,2}\r$, where $\dom(f_r)=\nu_i[\dom(f^p_r)]$ for every $i\in I$. By shrinking $A'_\xi$ even more, we may assume that for all $r\in [\xi^*,\xi)$, and for all $\nu_i\in A'_\xi$, $\nu_i\restriction\dom(f^p_r)$ is constant, and therefore there is a unique way to lift $l_r$ to $p'_r$ so that $\dom(f^{p'}_r)=\dom(f^p_r)$ using some (any) of the $\nu_i$ where $i\in I$.

It remains to define the $\xi$-coordinate coordinate of $p'$. We already defined $f^{p'}_\xi$ and $A^{p'}_\xi$, define $H^{p'}_{r,\xi}(\nu_i)=h^{q_{2i+1}}_{r,\xi}$ for $r=0,1$, and $H^{p'}_{2,\xi}(\bar{\nu})=\bigcup_{i<\kappa_\xi, \ \nu_i(\kappa_\xi)=\bar{\nu}}h^{q_{2i+1}}_{2,\xi}$. One can verify that for each $\nu\in Ex_{\xi}(p')$, if $\nu\restriction D=\nu_i$, then $p'{}^{\smallfrown}\nu_i\leq^* q_{2i+1}$.
\end{proof}
Now let us turn to the proof of the diagonalization property. The proof is by induction on $\beta\leq\kappa^+$, and also involves \ref{Lemma: lemma involved in induction with pre-strong Prikry}-\ref{remark: density in successor step} from the previous section.

\begin{theorem}\label{diagonalization property + density}
    For every $\alpha\leq\kappa^{++}$,  $\mathbb{Q}^*_\alpha,\mathbb{Q}_\alpha$ have the diagonalization property and $\mathbb{Q}_\alpha$ is $\leq^*$-dense in $\mathbb{Q}^*_\alpha$. In particular, $\mathbb{Q}_\alpha$ satisfies the pre-strong Prikry Lemma.  
\end{theorem}
\begin{proof}
By induction on $\alpha$. For the successor stage, suppose that $\mathbb{Q}^*_\alpha$ has the diagonalization property and $\mathbb{Q}_\alpha$ is $\leq^*$-dense in $\mathbb{Q}^*_\alpha$; let us show it for $\mathbb{Q}^*_{\alpha+1}$. By the induction hypothesis $\mathbb{Q}_\alpha$ satisfies the pre-strong Prikry lemma and by remark \ref{remark: density in successor step}, it follows that $\mathbb{Q}_{\alpha+1}$ is $\leq^*$ dense in $\mathbb{Q}_{\alpha+1}$. Hence for the successor step it remains to prove that $\mathbb{Q}^*_{\alpha+1}$ satisfies the diagonalization property. Let $a\in\mathbb{Q}^*_{\alpha+1}$ be any condition and fix a coordinate $\xi\notin \supp((a)_0)$. Suppose we are in stage $i$ of the game and suppose that $a_{2j}$ and $b_{2j+1}$ for $j<i$ have been defined. Inductively we assume that $q^{a_{2j}}$ and $q^{b_{2j+1}}$ are played according to some winning strategy for Good which is guaranteed by the diagonalization property for $\mathbb{Q}^*_\alpha$. Define $a_{2i}=\l q,\mathcal{S}\r$ as follows, $q$ is given by the winning strategy of the corresponding run of the game in $\mathbb{Q}^*_\alpha$ and for the definition of $\mathcal{S}$ we split into cases, we take $l(\mathcal{S})=\overline{\bigcup_{j<i}l(\mathcal{S}^{b_{2j+1}})}$. We define $\mathcal{S}\restriction l(\mathcal{S}^a)+1$ to be the shift of $\mathcal{S}^a$ to $Ex^*((q)_0)$, and for $l(\mathcal{S}^a)< i\leq l(\mathcal{S})$, $\mathcal{S}_i=\mathcal{S}^a_{l(\mathcal{S}^a)}$. This way we maintain $a_{2i}\leq^* a_{2j}$ for all $j<i$ and also whenever the bad player plays $b_{2i+1}=(q^{2i+1},\mathcal{R})\leq^* a_{2i}^\smallfrown\nu_i$, then the strategy $\mathcal{R}$ up to $l(\mathcal{S})$ is determined completely by $a$. 


Let $p_*\leq^*q^{a_{2i}}$ for all $i<\kappa_\xi$ be the condition diagonalizing the $q^{b_{2i+1}}$'s given by the winning strategy. By remark \ref{remark: vacously}, if $B^*=\{\nu\in A^{p_*}_\xi\mid \neg\exists \nu\restriction \dom(f^{p_{2i}}_\xi)\neq \nu_i\}\in E_\xi(\dom(f^{p^*}_\xi))$, then we are done. Otherwise, for  each $\nu\in B^*$, we can define $F(\nu)<\omega_1$ to be the delay of $\mathcal{S}^{b^{2i+1}}_{l(\mathcal{S}^{b^{2i+1}})}$ where $\nu=\nu_i$. The by proposition \ref{Prop: Rowbottom for object-measures}, there is $B_*\subseteq B^*$, $B_*\in E_\xi(\dom(f^{p^*}_\xi))$ such that for each $\nu\in B_*$, $F(\nu)$ is constantly $\xi^*$.
Define $a'=\l p',\mathcal{S}'\r$ where $$l(\mathcal{S}')=\overline{\bigcup_{i<\kappa_\xi}l(\mathcal{S}^{q^{b_{2i+1}}})}$$
  Fix $\eta\leq l (\mathcal{S}')$, and $(\vec{\mu},c)\in Ex^*((p')_0)$. For the easy case, when $\xi\notin \dom(\vec{\mu})$,   we may simply copy the strategies from $a$, namely, $$\mathcal{S}'_\eta(\mu,c)=\mathcal{S}^a_{\min\{\eta,l(\mathcal{S}^a)\}}(w^{p'}_{a_0}(\vec{\mu},c))$$  If $\xi\in\dom(\vec{\mu})$, then since $\mu_\xi\in B^*$, for some unique $i<\kappa_\xi$,  $\mu_\xi\restriction d=\nu_i$ where $d=\dom(f^{a_{2i}}_\xi)$.  By Corollary \ref{Cor: properties of step extension} item (3),  there is a unique $(\mu',c')\in Ex^*(q^{b_{2i+1}})$ such that  $p'{}^\smallfrown(\vec{\mu},c)\leq^* (q^{b_{2i+1}})^\smallfrown(\mu',c')$. Now define $$\mathcal{S}'_\eta(\vec{\mu},c)=\mathcal{S}^{q^{b_{2i+1}}}_{\min\{\eta,l(\mathcal{S}^{q^{b_{2i+1}}})\}}(\vec{\mu}',c')$$
 Let us check that at each coordinate $\eta\leq l(\mathcal{S}')$ there is a well defined delay. If $\eta<l(\mathcal{S}')$ from the definition of the $a_{2i}$'s,  then there is a unique $i$ such that $l(\mathcal{S}^{a_{2i}})<\eta\leq l(\mathcal{S}^{b_{2i+1}})$, and so the delay of $\mathcal{S}'_\eta(\vec{\mu},c)$ is at most the delay of $\mathcal{S}^a_{\min\{\eta,l(\mathcal{S}^a)\}}$,   the delay of $\mathcal{S}^{b_{2i+1}}_\eta$, and $\xi^*$. If $\eta=l(\mathcal{S}')$, then the delays is at most the delay of $\mathcal{S}^a_{l(\mathcal{S}^a)}$ and $\xi^*$.

It follows that $a'\leq^* a$ and for every $\nu\in Ex_\xi(a')$, if $\nu\restriction d=\nu_i$, then $a'{}^\smallfrown\nu\leq^* b_{2i+1}$. 

Now suppose that $\alpha$ is limit and for all $\beta<\alpha$ the diagonalization property holds for $\mathbb{Q}^*_\beta$ and $\mathbb{Q}_\beta$ is $\leq^*$-dense in $\mathbb{Q}^*_\beta$.  First let us address the diagonalization property for $\mathbb{Q}^*_\alpha$. The argument will be similar to the one in the successor step, but we will have to insure somehow that the delyas at each element in the support of a condition are stabilizes. The next proposition shows that once we were able to do so, then we can diagonalize:
\begin{lemma}\label{babydiag}
    Suppose that $a\in\mathbb{Q}_\alpha^*$, $\xi>\xi(a)$ is such that for some fixed $\xi'<\omega_1$, we have the following density: for every $b\leq a$, with $\xi(b)=\xi$, there is $b'\leq^* b$ with delay $\xi'$. Then the diagonalization property holds at $a$ and $\xi$.
\end{lemma}
\begin{proof}
    The argument is similar to the one in the successor step except that no stabilization will be required in the end. We play according to a winning strategy of $\mathbb{Q}_0$ as follows: Suppose that $\l a_{2j},b_{2j+1},b'_{2j+1}\mid j<i\r$ have been played, where   $b'_{2j+1}\leq^* b_{2j+1}$ (such $b'_{2j+1}$ can be found using the density assumption). Assume that $\l (a_{2j})_0, (b'_{2j+1})_0\mid j<i\r$ are played according to a winning strategy. Define $a_{2i}$ as follows, $\supp(a_{2i})=\bigcup_{j<i}\supp(b'_{2j+1})$, and for each $\beta\in \supp(a_{2i})$, let $$l(\mathcal{S}^{a_{2i,\beta}})=\sup_{j, \ \beta\in \supp(b'_{2j+1})}l(\mathcal{S}^{b'_{2j+1}})$$
    Let $p^*$ diagonalize $\l (b'_{2i+1})_0\mid i<\kappa_\xi\r$, 
Let $a^*\in\mathbb{Q}_{\alpha}$ the following condition: $(a^*)_0=p^*$, $\supp(a^*)=\bigcup_{i<\kappa_\xi}\supp(b'_{2i+1})$. For every $\beta\in \supp(a^*)$, let us define $\mathcal{S}^{(a^*)_\beta}$. $$l(\mathcal{S}^{(a^*)_\beta})=\sup_{i<\kappa_\xi, \ \beta\in \supp(b'_{2i+1})}l(\mathcal{S}^{(b'_{2i+1})_\beta}).$$ 
Fix any $\eta\leq l(\mathcal{S}^{(a^*)_\beta})$, and $(\vec{\mu},c)\in Ex^*(p^*)$. Once again, let us split into cases. If $\xi\in\dom(\vec{\mu})$, let $i<\kappa_\xi$ be the unique such that $p^*{}^\smallfrown(\vec{\mu},c)\leq p^\smallfrown \nu_i$, then there is a unique $(\mu',c')\in Ex^*(p^\smallfrown\nu_i)$ such that  $p^*{}^\smallfrown(\vec{\mu},c)\leq^* (p^\smallfrown\nu_i)^\smallfrown(\mu',c')$. Now define $$\mathcal{S}^{(a^*)_\beta}_\eta(\vec{\mu},c)=
\mathcal{S}^{(b'_{2i+1})_\beta}_{\min(\eta,l(\mathcal{S}^{(b'_{2i+1})_\beta}))}(\vec{\mu}',c')$$
 If $\xi\notin\dom(\vec{\mu})$ we simply copy the strategies from $a$, namely, $$\mathcal{S}^{(a^*)_\beta}_\eta(\mu,c)=\mathcal{S}^{(a)_\beta}_{\min\{\eta,l(\mathcal{S})\}}(\vec{\mu},c)$$
If $\beta\notin \supp(a)$, then define $\mathcal{S}^{(a^*)_\beta}_\eta(\mu,c)=\emptyset$. Note that similar to the successor step, since we have fixed the delay to be $\xi'$ condition $(\theta)$ will still hold for every strategy of $a^*$.

\end{proof}

 First let us deal with the harder case where $\cf(\alpha)<\kappa$.
    The following proposition can be viewed as a limit version of Proposition \ref{Prop: Unbounded}.
\begin{proposition}
     For every $\beta<\alpha$, $b\in\mathbb{Q}^*_\beta$ and every $\delta<\kappa^+$ there is $b'\leq^* b$ such that $b'\in\mathbb{Q}_\beta$, and the maximums of all of the strategies in $b'$ is at least $\delta$. More precisely, for every $\gamma\in\supp(b')$, and $(\vec{\mu},c)\in Ex^*((b')_0)$, $\max(S^{b'}_{l(\mathcal{S}^{b^*, \gamma})}(\vec{\mu},c))\geq \delta$. 

     Moreover, we can arrange that $(b^*)_0\upharpoonright \xi(b) = (b)_0\upharpoonright \xi(b)$
\end{proposition}
\begin{proof}
We assume that $\beta$ is limit since otherwise the conclusion follows from \ref{Prop: Unbounded} and a suitable induction hypothesis regarding $\beta$. First note that by the inductive hypothesis, we have that $\mathbb{Q}_\beta$ is dense in $\mathbb{Q}^*_\beta$ and therefore $\mathbb{Q}_\beta$ satisfies the pre-strong Prikry property.

Fix such $b,\delta$. For each $\eta<\omega_1$, let $\dot{C}_\eta$ be a $\mathbb{Q}_{\beta,\eta}$-name for the intersection $\bigcap_{\gamma\in\supp(b)}\dot{C}_{\gamma,\eta}$. Since $\supp(b)$ has size at most $\kappa$, $\dot{C}_\eta$ is a name for a club. By an identical argument to the one in Lemma \ref{Lemma: lemma involved in induction with pre-strong Prikry}, we can find $\delta<\delta'<\kappa^+$ and $b^0\leq^*b$, such that for all $q\leq b^0$, $q\Vdash_{\mathbb{Q}_{\beta, \xi(q)}}\dot{C}_{\xi(q)}\cap(\delta, \delta')\neq\emptyset$. Moreover, we can find such $b^0$, so that $b^0\upharpoonright\xi(b)=b\upharpoonright \xi(b)$ and $b^0\in \mathbb{Q}_\beta$.

Since $\supp(b^0)$ might have points not in $\supp(b)$, we need do this process
$\omega$-many times. 
Namely, construct a $\leq^*$-decreasing sequence of conditions $\l b^n\mid n<\omega\r$ with $b^n\restriction \xi(b)=b\restriction \xi(b)$, and increasing $\l\delta_n\mid n<\omega\r$, such that for all $q\leq b^{n+1}$, $$q\Vdash_{\mathbb{Q}_{\beta, \xi(q)}}\bigcap_{\gamma\in\supp(b^n)}\dot{C}_{\gamma,\xi(q)}\cap(\delta^n, \delta^{n+1})\neq\emptyset.$$ Since $(\mathbb{Q}_\beta,\leq^*)$ is countably closed, we can find $b'\leq^*b^n$ for all $n$, and $\delta^* = \sup_n\delta_n$. Let $b''\leq^* b'$ be obtained by simply adding $\delta^*$ at the end of each strategy unless the maximum is already greater than $\delta^*$.

Then we have a conditions $b''\leq^*b$, with $b''\upharpoonright \xi(b) = b\upharpoonright \xi(b)$, such that for all $\gamma\in\supp b''$, the $\gamma$-strategies of $b''$ have maximums at least $\delta^*\geq \delta$. 
\end{proof}
\begin{lemma}\label{lemma: lemma1}
If $b\in\mathbb{Q}^*_\alpha$ and $\xi(b)$ is such that $\cf(\alpha)<\kappa_{\xi(b)}$, and every $\delta<\kappa^+$, then there is $b'\leq^*b$ such that $b'\in\mathbb{Q}_\alpha$ and all of its strategies have maximum above $\delta$.
\end{lemma}
\begin{proof}
    Let $\l\alpha_i\mid i<\cf(\alpha)\r$ be an increasing sequence with limit $\alpha$,   Using the previous proposition, build a coherent increasing sequence $\l b_i\mid i<\cf(\alpha)\r$, such that:
\begin{enumerate}
\item each $b_i\in\mathbb{Q}_{\alpha_i}$, $b_i\leq^*b\upharpoonright\alpha_i$, $(b_i)_0\upharpoonright\xi(b)=(b)_0\upharpoonright\xi(b)$.
\item 
for all $i<j$, $b_j\upharpoonright\alpha_i\leq^*b_i$,
\item 
all strategies of $b_i$ have maximums above $\delta$.
\end{enumerate}
Then we can construct $b^*\leq^* b$, such that for each $i$, ${b^*}\upharpoonright\alpha_i\leq^* b_i$. Here we use that below $\xi(b)$ the conditions do not change, and since $cf(\alpha)<\kappa_{\xi(b)}$ by \ref{lemma: closure of quotient} we can take $\leq^*$-lower bounds. Note that each since each $b_i\in\mathbb{Q}_{\alpha_i}$ it can also be viewed as a condition in $\mathbb{Q}_\alpha$.

\end{proof}
Note that the above lemma also shows the density of $\mathbb{Q}_\alpha$ in $\mathbb{Q}^*_\alpha$ for conditions $b$ such with $\kappa_{\xi(b)}>\cf(\alpha)$. The next corollary generalizes this.
\begin{corollary}\label{cor: stabilized delay}
    Let $a\in \mathbb{Q}^*_\alpha$, $\delta<\kappa^+$, and  let $\xi>\xi(a)$ be such that $\cf(\alpha)<\kappa_\xi$. Then there is $a'\leq^* a$, $a'\in \mathbb{Q}^*_\alpha$ with delay $\xi$, i.e. for each $\beta\in \supp(a')$, $\max(\mathcal{S}^{a'_\beta}_{l(\mathcal{S}^{a'_\beta})}(\vec{\mu},c))>\delta$ and the delay of $\mathcal{S}^{a'_\beta}_{l(\mathcal{S}^{a'_\beta})}$ is at most $\xi$.
\end{corollary}
\begin{proof}
  Construct $b_\nu\leq^*a^\smallfrown\nu$, for all $\xi$-extensions of a, such that
each $b_\nu\in \mathbb{Q}_\alpha$ (using Lemma \ref{lemma: lemma1}) and has max strategies all above $\delta$. Since all the $b_\nu$'s have delay $0$, we can diagonalize them using Lemma \ref{babydiag} to get a condition $a'\leq^* a$ in $\mathbb{Q}^*_\alpha$.    Since the delay for each $b_\nu$ is zero, we have that $a'$ has delay $\xi$ (or possibly better).
\end{proof}
\begin{corollary}\label{adultdiag}
    $\mathbb{Q}^*_\alpha$ satisfies the diagonalization property (and also the pre-strong Prikry lemma)
\end{corollary}
\begin{proof}
    Immediate from Lemma \ref{babydiag} and Corollary \ref{cor: stabilized delay}.
\end{proof}
It remains to complete the proof for the density of $\mathbb{Q}_\alpha$.
\begin{lemma}\label{Lemma: lemma2}
Let $a\in\mathbb{Q}^*_\alpha$ and $\bar{\delta}<\kappa^+$. Then there is $a'\leq^* a$, $a'\in\mathbb{Q}^*_\alpha$ and $\delta>\bar{\delta}$, such that for each $\beta\in \supp(a')$,  and all $\eta\geq \xi(a)$, for all $\eta$-extensions $r$ of ${a'}\upharpoonright\beta$, $r\Vdash_{\mathbb{Q}_{\beta,\eta}} \delta\in\dot{C}_{\beta,\eta}$.
\end{lemma}
\begin{proof}
Fix $a$ and let $\dot{C}_\eta$ be a $\mathbb{Q}^*_\alpha$-name for  $\bigcap_{\beta\in\supp(a)}\dot{C}_{\beta,\eta}$. Note that this is forced to be a club. Using the pre-strong Prikry property (once) for $\mathbb{Q}^*_\alpha$, we can find $\delta_0>\bar{\delta}$ and $a^0\leq^* a$, such that for all $r\leq a^0$, 
$r\Vdash (\bar{\delta},\delta_0)\cap\dot{C}_{\xi(r)}\neq\emptyset$. Using that $(\mathbb{Q}^*_\alpha,\leq^*)$ is $\sigma$-closed, we repeat this $\omega$-many times to find the desired $\delta$ and $a'$.

\end{proof}

Let $a'\in\mathbb{Q}^*_\alpha$, be want to find $a^*\leq a'$ such that $a^*\in\mathbb{Q}_\alpha$. Let $\delta_0$ be above all the max of the strategies of $a$ and $\xi>\xi(a)$. By Corollary \ref{cor: stabilized delay}, we can find $a'\leq^* a$ with delay $\xi$. By lemma \ref{Lemma: lemma2}, we can find $a^0\leq^* a'$  such that for some $\delta_0'>\delta_0$, for all $\beta\in \supp(a^0)$, all $\eta\geq \xi(a)$, and all $r$ $\eta$-extensions of $a^0\upharpoonright\beta$, $r\Vdash_{\mathbb{Q}_{\beta,\eta}} \delta'_0\in\dot{C}_{\beta,\eta}$.

Let $\delta_1>\delta'_0$ be above all the max of the strategies of  $a^0$.

Repeat this $\omega$-many times to get $\delta_n<\delta'_n<\delta_{n+1}$ and $\leq^*$-decreasing conditions $a^n, n<0$, such that:
\begin{enumerate}
\item each max of the strategies of $a^n$ is less than $\delta_{n+1}$
\item $a^n$ has delay $\xi$.
\item
for each $\beta\in \supp(a^n)$,  and all $\eta\geq \xi(a)$, for all $\eta$ extensions $r$ of ${a^n}\upharpoonright\beta$, $r\Vdash_{\mathbb{P}_{\beta,\eta}} \delta'_n\in \dot{C}_{\beta,\eta}$.
\end{enumerate}

Now, let $\delta_\omega = \sup_n\delta_n=\sup_n\delta'_n$ and let $c\leq^* a^n$ for each $n$ be the minimal lower bound. Then for each $\beta\in\supp(c)$, if $(\mu,d)\in Ex^*((c)_0)$, and $\xi\in\dom(\vec{\mu})$ then $\max(\mathcal{S}^{c_\beta}(\vec{\mu},d))=\delta_\omega$. This follows by item $(3)$. If $\xi\notin\dom(\vec{\mu})$, then by item $(1)$, $\max(\mathcal{S}^{c_\beta}(\vec{\mu},d))\leq\delta_{\omega}$.
  
Moreover, by construction, for any $\xi$-step extension $c^\smallfrown\nu$ of $c$, we have that $\langle\delta_\omega, {c^\smallfrown\nu}\upharpoonright\beta\rangle\in R$. This follows since the $a^n$'s have delay $\xi$.

\begin{claim} $\langle\delta_\omega, c\upharpoonright\beta\rangle\in R$.\end{claim}
\begin{proof} Since $\delta_\omega=\sup_n\delta'_n$, by (3) above we have that each $\eta$ and each $r$ which is an $\eta$-step exteion of $c\restriction\beta$,
$r\Vdash_{\mathbb{Q}_{\beta,\eta}}\delta_\omega\in\dot{C}_{\beta,\eta}$. 
\end{proof}
Finally, we construct the condition $c'\leq^* c$ by increasing the length of each strategy by one and adding $\delta_\omega$ at the end.

This ensures that $\max(\mathcal{S}^{c\restriction\beta}_{l(\mathcal{S}^{c\restriction\beta})}(\vec{\mu},d))$ is  $\delta_\omega$ for all $(\vec{\mu},d)\in Ex^*(c'_0)$, and in turn that $c'\in\mathbb{Q}_\alpha$.


 If $\cf(\alpha)=\kappa^+$, we use that the support is bounded in some $\beta$ and the inductive assumption for $\beta$. More precisely, density is immediate since $\mathbb{Q}_\alpha^*=\bigcup_{\beta<\alpha} \mathbb{Q}^*_{\beta}$ and so for the diagonalization property we can now use Lemma \ref{babydiag}. $\qedhere_{\text{ Theorem \ref{diagonalization property + density}}}$
\end{proof}

\begin{theorem}
    (The Strong Prikry Lemma) Let $\alpha\leq\kappa^{++}$, $a\in\mathbb{Q}_\alpha$, and suppose that $D\subset \mathbb{Q}_\alpha$ is a dense open set. Then there is $b\leq^* a$ and $\vec{\xi}\in[\omega_1]^{<\omega}$, such that all $\vec{\xi}$-step extensions of $b$ are in $D$.
\end{theorem}
\begin{proof}

Fix $a,D$ as in the statement. Let $p=(a)_0$.

Let $b\leq^*a$, $b\in \mathbb{Q}_\alpha$ be given by pre-strong Prikry Lemma, i.e. $b$ such that for every
increasing  $\vec{\xi}\in[ \omega_1]^{<\omega}$ disjoint from the support of $b$, either all of its $\vec{\xi}$-step extensions are in $D$, or none are. We apply the lemma $\omega_1$-many times (i.e. for every $\vec{\xi}\in[\omega_1]^{<\omega}$ to find a single $b$ that works for every $\vec{\xi}$.

Since $D$ is dense, let $c\leq b$ be a condition in $D$. Let  $\vec{\nu}$ be such that $c\leq^* b^{\smallfrown}\vec{\nu}$ and let $\vec{\xi}$ be the finite subset of $\omega_1$ corresponding to the coordinates of $\vec{\nu}$. It follows by the choice of $b$ that all of its $\vec{\xi}$-step extensions are in $D$. That concludes the proof of the Prikry lemma.
\end{proof}
\begin{corollary}[The Prikry Lemma]
    Suppose that $\phi$ is a sentence and $a$ is a condition, Then there is a direct extension of $a$ deciding $\phi$.
\end{corollary}
\begin{proof}
    To derive the usual Prikry property, let $\phi$ be any statement in the forcing language and $D$ be the dense open set of conditions deciding $\phi$. By further shrinking the measure one sets we may assume that if the $\vec{\xi}$-step extensions decide $\phi$, the decision in uniform across all $\vec{\xi}$-step extensions (this uses Proposition \ref{Prop: Rowbottom for object-measures}). Now take $\vec{\xi}$ as above of minimal size, and we claim that $\vec{\xi}=\emptyset$. 
    
    Otherwise, let $\vec{\xi}=\vec{\eta}^{\smallfrown}\{\max(\xi)\}$ and without loss of generality we may assume that the uniform decision for the $\vec{\xi}$-step extension is $\phi$. Let $\vec{\mu}$ be any $\vec{\eta}$-step condition. Consider the condition $b^*=b^{\smallfrown}\vec{\mu}$. We claim that $b^*\Vdash \phi$ which then implies that the $\vec{\eta}$-extensions will all decide the $\phi$, contradicting the minimality of $\vec{\xi}$. Otherwise, there is an extension of $b^*$ deciding $\neg\phi$. But this extension must be compatible with some extension of the form $b^{*\smallfrown}\nu$ for some $\max(\xi)$-step extension $\nu$, this is a contradiction, since to compatible elements cannot force contradictory information.

\end{proof}
\begin{corollary}
  Let $\phi$ be a sentence, $a\in \mathbb{Q}_\alpha$ is a condition and $\xi$ is in the support of $a$. Then there is a direct extension  $b\leq^*a$, such that $b\restriction\mathbb{S}_\xi=a\restriction\mathbb{S}_\xi$ that is $b\leq_{\mathbb{S}_\xi}a$ (see definition \ref{definition: the term space forcing}), and for all $c\leq b$ which decides $\phi$, we have that\footnote{Here $l^\smallfrown b\upharpoonright[\xi, \omega_1)$ means $b+l$ in the sense of definition \ref{def: term space forcing projection}.} ${c\upharpoonright\mathbb{S}_\xi}^{\smallfrown} b\upharpoonright[\xi, \omega_1)$ also decides $\phi$ (in the same way).
\end{corollary}
\begin{proof}
Fix $a\in\mathbb{Q}_\alpha$ and $\xi\in \supp(a)$. Recall the poset $\mathbb{S}_\xi$ from Proposition \ref{Properties}(3). For any lower part $l\in\mathbb{S}_\xi$ below $a\upharpoonright\mathbb{S}_\xi$, let $a_l\leq^* l^{\smallfrown} a\upharpoonright[\xi, \omega_1)$ be given by the Prikry property. I.e. $a_l$ decides $\phi$.  The number of possible lower parts $l\in\mathbb{S}_\xi$ is $\kappa_\xi$ (note that the last part of $l$ is in $Col(s_\xi(f_\xi(\kappa_\xi))^{+3}, <\kappa_\xi)$). 
By \ref{lemma: closure of quotient}, fixing the lower part, we gain $\kappa_\xi^+$-closure of the upper part.  So by inductively maintaining that $a_l+(a\restriction\mathbb{S}_\xi)$ are decreasing in $\leq_{\mathbb{S}_\xi}$ we can find a lower bound. Finally let $b\leq^* a$ be such that $b\upharpoonright{\mathbb{S}_\xi}=a\upharpoonright{\mathbb{S}_\xi}$ and $b+l\leq_{\mathbb{S}_\xi} a_l$ for each $l$. Then $b$ is as desired.

\end{proof}
It follows from this corollary and again from \ref{lemma: closure of quotient}, that every $\mathbb{Q}_\alpha$-name of a subset of $\kappa_\xi$ can be reduced to a $\mathbb{S}_\xi$-name. In other words, the quotient forcing $\mathbb{Q}^{\mathbb{S}_\xi}_\alpha$ does not add bounded subsets to $\kappa_\xi^+$. This is used to show the preservation of all the other relevant cardinals.
\begin{corollary}\label{corollary: cardinal structure in Qalpha below kappa}
    For every $\xi<\omega_1$, $V^{\mathbb{Q}_\alpha}$ models that
$$ ``[\bar{\kappa}_\xi,\kappa_\xi]\cap card.=\{\bar{\kappa}_\xi,\bar{\kappa}^+_\xi,f^{p_0}_\xi(\kappa_\xi),s_\xi(f^{p_0}_\xi(\kappa_\xi))^{++},s_\xi(f^{p_0}_\xi(\kappa_\xi))^{+3},\kappa_\xi\}"$$
In particular, $\kappa=\kappa_{\omega_1}$.
\end{corollary}
\begin{proof}
    Since $\mathbb{Q}_\alpha$ projects to $\mathbb{P}_{\bar{E}}$ which satisfy the above theorem, we need to prove that no further damage has been done.  As we mentioned in the paragraph before the Corollary, the extension from $V^{\mathbb{S}_\xi}$ to $V^{\mathbb{Q}_\alpha}$ does not add new  bounded subsets to $\kappa_\xi^+$. In particular, the cardinal structure of the interval $[\bar{\kappa}_\xi,\kappa_\xi]$ is not changed between these models.
\end{proof}

\begin{corollary}
    $\kappa$ is strong limit in $V^{\mathbb{Q}_\alpha}$
\end{corollary}
\begin{proof}
For each $\xi<\omega_1$, since  $V^{\mathbb{Q}_\alpha}$ does not add bounded subsets in $\kappa_\xi^+$ to $V^{\mathbb{S}_\xi}$, we have $(P(\bar{\kappa}_\xi))^{V^{\mathbb{Q}_\alpha}}=(P(\bar{\kappa}_\xi))^{V_{\mathbb{S}_\xi}}$ and $(2^{\bar{\kappa}_\xi})^{V_{\mathbb{S}_\xi}}\leq \kappa_\xi$.
\end{proof}
\begin{corollary}
    $\kappa^+$ is preserved by $\mathbb{Q}_{\alpha}$.
\end{corollary}
\begin{proof}
    Otherwise, $\theta=cf^{V^{\mathbb{Q}_\alpha}}((\kappa^+)^V)<\kappa$ and there is an unbounded function $f:\theta\rightarrow (\kappa^+)^V$ unbounded witnessing this and $a\in\mathbb{Q}_\alpha$ that forces that $\dot{f}$ is unbounded in $(\kappa^+)^V$. Pick any $\xi<\omega_1$ such that $\kappa_\xi>\theta$. Now as we did before, we can construct a $\leq^*$-decreasing sequence $\l a_i\mid i<\theta\r$ that fixes $a\restriction\mathbb{S}_\xi$ and for each $i<\theta$ and each $l\leq a\restriction\mathbb{S}_\xi$, there is $\vec{\xi}_{i,l}\in[\omega_1]^{<\omega}$ such that every $\vec{\mu}\in Ex_{\vec{\xi}}(a_i)$, $a_i^\smallfrown\vec{\mu}+l$ decides the value of $\dot{f}(i)$. Taking $a^*$ to be a lower bound for the $a_i$'s
 (which is possible by \ref{lemma: closure of quotient}), we can now define a function $g$ in $V$ from a set of size $\kappa$, namely, the set of lower parts times all possible $\vec{\mu}\in Ex_{\xi_i}(a^*)$ times $\theta$ to $\kappa^+$, defined by $g(l,\vec{\mu},i)=\gamma$ if $l\leq a^*\restriction\vec{\xi}$ and $a^{*\smallfrown}\vec{\mu}+\vec{\mu}\Vdash \dot{f}(i)=\gamma$. Note that $g$ is unbounded in $\kappa^+$ as $a^*$ forces that $\dot{f}$ is unbounded, and each possible value $\dot{f}(i)$ is recorded in one of the values of $g$. This produces a contradiction. \end{proof}
\begin{corollary}\label{Cor: Failure of SCH}
    $V^{\mathbb{Q}_\alpha}\models 2^\kappa>\kappa^+$.
\end{corollary}
\begin{proof}
    Thus far we have shown that the cardinals of $V^{\mathbb{Q}_\alpha}$ are identical to those of $V^{\mathbb{P}_{\bar{E}}}$. Hence, since $\mathbb{P}_{\bar{E}}$ is a projection of $\mathbb{Q}_\alpha$, $\kappa^+<(2^\kappa)^{\mathbb{P}_{\bar{E}}}\leq (2^{\kappa})^{\mathbb{Q}_\alpha}$.
\end{proof}
\section{Non-reflecting is fragile}\label{Section: reflection} 
\begin{lemma}\label{Lemma: Reflectionholds in the product}
     Suppose that the sequence of $\kappa_\xi$'s are indestructibly supercompact cardinals, then $V^{\mathbb{S}_\xi\times\mathbb{Q}^+_{\alpha,\xi}}\models \text{Refl}(E^{\kappa^+}_{\leq\bar{\kappa}_\xi},E^{\kappa^+}_{<\kappa_\xi})$.
\end{lemma}
\begin{proof}
The argument is the same as in Lemma \ref{Lemma: generic lift}, the point is that $\kappa_\xi$ remains supercompact in $V^{\mathbb{Q}^+_{\alpha,\xi}}$ (by indestructibility), then generically lifting a supercompact embedding from $V^{\mathbb{Q}^+_{\alpha,\xi}}$ to $V^{\mathbb{S}_\xi\times\mathbb{Q}^+_{\alpha,\xi}}$ is identical to \ref{Lemma: generic lift} and so is the rest of the argument.
\end{proof}
\begin{theorem}\label{Thm: stationary sets preservation passing to product}
 Every stationary set in $V^{\mathbb{Q}_{\alpha,\xi}}$ is stationary in $V^{\mathbb{S}_\alpha\times\mathbb{Q}^+_{\alpha,\xi}}$.
\end{theorem}
\begin{proof}

First we recall that $\mathbb{S}_\xi=\mathbb{S}_{<\xi}\times Col(\tau,<\kappa_\xi)$. Denote\footnote{ Recall that $\mathbb{Q}^{\mathbb{S}_{<\xi}}_{
\alpha,\xi}$ is the forcing with underlining set $\mathbb{Q}_{\alpha,\xi}$ ordered by $a\leq'_\xi b$ iff $a_0\restriction \mathbb{S}_{<\xi}=b_0\restriction \mathbb{S}_{<\xi}$, $a\leq b$.} by $\mathbb{Q}'_{\alpha,\xi}=\mathbb{Q}^{\mathbb{S}_{<\xi}}_{
\alpha,\xi}$. By Lemma \ref{lemma: Q^* degree of directed}, $\mathbb{Q}'_{\alpha,\xi}$ is $\tau^+$-directed closed, and as we have seen $\mathbb{S}_{<\xi}\times \mathbb{Q}'_{\alpha,\xi}$ projects onto $\mathbb{Q}_{\alpha,\xi}$ via $(s,a)\mapsto a+s$. We define a projection from $\pi_0:\mathbb{S}_\xi\times \mathbb{Q}^+_{\alpha,\xi}\rightarrow \mathbb{S}_{<\xi}\times \mathbb{Q}'_{\alpha,\xi}$ as follows $$\pi_0(s,a)=(s\restriction\mathbb{S}_{<\xi},a+s\restriction Col(\tau,<\kappa_\xi))$$
where $a+s\restriction Col(\tau,<\kappa_\xi)$ is obtained by extending the $Col(\tau,<\kappa_\xi)$ part of $a_0$ to $s\restriction Col(\tau,<\kappa_\xi)$, then taking the minimal condition in the sense of Remark \ref{remark:minimal}. The following claim is more of the same arguments we have seen so far.
 \begin{claim}
     \begin{enumerate}
         \item $\pi_0$ is a projection.
         \item $\pi_0(s,a)=(s',a')$ then $a+s=a'+s'$.
     \end{enumerate}
 \end{claim} 
 \qed
 \[\begin{tikzcd}
{\mathbb{S}_{\xi}\times \mathbb{Q}^+_{\alpha,\xi}} \arrow[d, "\pi_0"] \arrow[dd, "{(s,a)\mapsto a+s}"', bend right=60] \\
{\mathbb{S}_{<\xi}\times \mathbb{Q}'_{\alpha,\xi}} \arrow[d, "{(s',a')\mapsto a'+s'}"]                                      \\
{\mathbb{Q}_{\alpha,\xi}}                                                                                              
\end{tikzcd}\]
Note that $\pi_0$ is not an isomorphism, since the strategies in $(s,a)$ depends on the fixed collapse in $p$, while the strategy $\pi_0((s,a))$ depends on the varying collapse $s\restriction Col(\tau,<\kappa_\xi)$.
 The plan is to  show that stationary sets in $V^{\mathbb{Q}_{\alpha,\xi}}$ remain stationary in $V^{\mathbb{S}_{<\xi}\times \mathbb{Q}'_{\alpha,\xi}}$ (Lemma \ref{Lemma: Second stationary pres}) and that stationary sets in $V^{\mathbb{S}_{<\xi}}\times \mathbb{Q}'_{\alpha,\xi}$ remain stationary in $V^{\mathbb{S}_{\xi}}\times \mathbb{Q}^+_{\alpha,\xi}$ (Lemma \ref{Lemma: First stationary pres}).
\begin{lemma}\label{Lemma: First stationary pres}
    If $S\subseteq E^{\kappa^+}_{\leq\bar{\kappa}_\xi}$is stationary in $V^{\mathbb{S}_{<\xi}\times \mathbb{Q}'_{\alpha,\xi}}$, then it is stationary in $V^{\mathbb{S}_{\xi}\times \mathbb{Q}^+_{\alpha,\xi}}$
\end{lemma}
\begin{proof}[Proof of Lemma.] First note that since $\mathbb{S}_{<\xi}$ has small cardinality (less than $\tau<\kappa_\xi$ in this case), there is $S'\subseteq S$ stationary such that $S'\in V^{\mathbb{Q}'_{\alpha,\xi}}$. Since the cardinality in $V^{\mathbb{Q}'_{\alpha,\xi}}$ of $(\kappa^+)^V$ is $\kappa_\xi^{++}$, and since $S'$ concentrates on cofinality less than $\bar{\kappa}_\xi$, by Shelah, $S'$ is approachable.  Note that $\mathbb{Q}'_{\alpha,\xi}$ is a projection of $Col(\tau,<\kappa_\xi)\times \mathbb{Q}^+_{\alpha,\xi}$ via the projection $\pi_2(s,a)= a'$, where $a'$ is determined by extending $a'_0=a_0^\smallfrown s$ and taking the minimal strategies determined by the fact that $a_0^\smallfrown s\leq a_0$ as in remark \ref{remark:minimal}. We would like to use Shelah again, to say that the quotient of that projection is $\tau$-closed and therefore stationarity of $S'$ is preserved.  Indeed, let $G$ be $V$-generic for $\mathbb{Q}'_{\alpha,\xi}$ and let $\l (s_i,a_i)\mid i<\sigma\r\in V[G]$ for $\sigma<\tau$, be
a decreasing sequence of conditions such that $\pi_2(s_i,a_i)=a_i'\in G$. Since the forcing is $\tau$-closed, the sequence of conditions is in the ground model, and therefore a simple density argument  (which uses $\tau$-closure again) shows that the set of all $q\in\mathbb{Q}'_{\alpha,\xi}$ such that either there is $i<\sigma$ such that $q\bot \pi_2(s_i,a_i)$ or for every $i<\sigma$, $\pi_2(s_i,a_i)\geq q$ is dense. So we can find $q\in G$ such that $q\leq \pi_2(s_i,a_i)$ for every $i<\sigma$. We can assume that $(q)_0$ is the least upper bound of the $(a_i')_0$'s and that the length of the strategies of each $(q)_i$ are exactly $\sup_{j<\sigma}l(\mathcal{S}^{(a'_j)_i})=\sup_{j<\sigma}l(\mathcal{S}^{(a_j)_i})$. Let us define a condition $(s,a)$ as follows:
\begin{itemize}
    \item $s=(q)_0\restriction Col(\tau,<\kappa_{\xi})$.
    \item $(a)_0=(a_0)_0\restriction \mathbb{S}_\xi$ and $(a)_0\restriction \mathbb{S}_{>\xi}=(q)_0\restriction\mathbb{S}_{>\xi}$.
    \item $(a)_i$ is defined recursively, at successor steps, we define $(a)_{i+1}=((a)_i,\mathcal{S})$ so that $l(\mathcal{S})=l(\mathcal{S}^{(q)_{i+1}})$ and for every $\rho\leq l(\mathcal{S})$, if there is $j<\sigma$ such that $\rho\leq l(S^{a_j})$ we define $S_\rho=S^{a_j}_\rho\restriction Ex^*((a)_0)$ (this is well defined as $(a)_0\leq^* (a_j)$ and so $E^*((a)_0)\subseteq E^*(a_j)$.) if $\sup_{j<\sigma}l(\mathcal{S}^{(a_j)_{i+1}})=\rho= l(\mathcal{S}^{(q)_{i+1}})$ is a limit ordinal and we define $$\mathcal{S}_\rho(\vec{\mu},c)=\overline{\bigcup_{j<\sigma}\mathcal{S}^{(a_j)_{i+1}}_{l(\mathcal{S}^{(a_j)_{i+1}})}(\vec{\mu},c)}$$
    Note again that this is well a well-defined  $(a)_i$ strategy.
    \item At limit steps, the support is the same as the support of $q$ so form a legitimate condition.
\end{itemize}
It remains to note that $\pi(s,a)=q\in G$ which is not hard to check and therefore $(s,a)$ is a lower bound of the $(s_i,a_i)$ in the quotient forcing.


It follows that $S'$ remains stationary in $V^{Col(\tau,<\kappa_\xi)}\times \mathbb{Q}^+_{\alpha,\xi}$. Finally, the quotient from $\mathbb{S}_\xi\times \mathbb{Q}^+_{\alpha,\xi}$ to $Col(\tau,<\kappa_\xi)\times \mathbb{Q}^+_{\alpha,\xi}$ is just $\mathbb{S}_{<\xi}$. The forcing $\mathbb{S}_{<\xi}$ is a small forcing (and in particular has a good chain condition), and thus preserves the stationarity of $S'$. 
\end{proof}
Now let us move to the second preservation lemma:
\begin{lemma}\label{Lemma: Second stationary pres}
    If $S\subseteq E^{\kappa^+}_{\leq\bar{\kappa}_\xi}$is stationary in $V^{ \mathbb{Q}_{\alpha,\xi}}$, then it is stationary in $V^{\mathbb{S}_{<\xi}\times \mathbb{Q}'_{\alpha,\xi}}$
\end{lemma}
\begin{proof}[Proof of Lemma.]
Suppose otherwise, and let $S=(\dot{S})_G\in V^{\mathbb{Q}_{\alpha,\xi}}=V[G]$ be a  stationary set which is no longer stationary in $V^{\mathbb{S}_{<\xi}\times\mathbb{Q}'_{\alpha,\xi}}=V[G^*]$, $\pi_{\mathbb{S}_{<\xi}*}(G^*)=G$.  Hence there is a club $C\in V[G^*]$ such that $S\cap C=\emptyset$. WLOG supposet that $1\Vdash_{\mathbb{S}_{<\xi}\times \mathbb{Q}_{\alpha,\xi}}\dot{S}\cap \dot{C}=\emptyset$. Since $\mathbb{S}_{<\xi}$ is a small forcing, $C$ contains a club $C^*\in V^{\mathbb{Q}'_{\alpha,\xi}}=V[G_*]$,
where $G_*=\pi_3(G^*)$, so that $G^*=H\times G_*$ where $\pi_3((s,a))=a$ and $H\subseteq \mathbb{S}_{<\xi}$ is $V[G_*]$-generic. Consider the projection $\pi_4:\mathbb{Q}'_{\alpha,\xi}\rightarrow Col(\tau,<\kappa_\xi)\times\mathbb{S}_{>\xi}$, and recall that $\mathbb{S}_{>\xi}$ is the upper $\kappa_\xi^+$-directed part of the forcing $\mathbb{P}_{\bar{E}}$. Let $F$  be the generic for $Col(\tau,<\kappa_\xi)\times \mathbb{R}$ induced from $G_*$ or equivalently from $G$ (using the projection $\pi$ from $\mathbb{Q}_{\alpha,\xi}$ to $Col(\tau,<\kappa_\xi)\times \mathbb{R}$ through $\mathbb{P}_\xi$- see the diagram below).
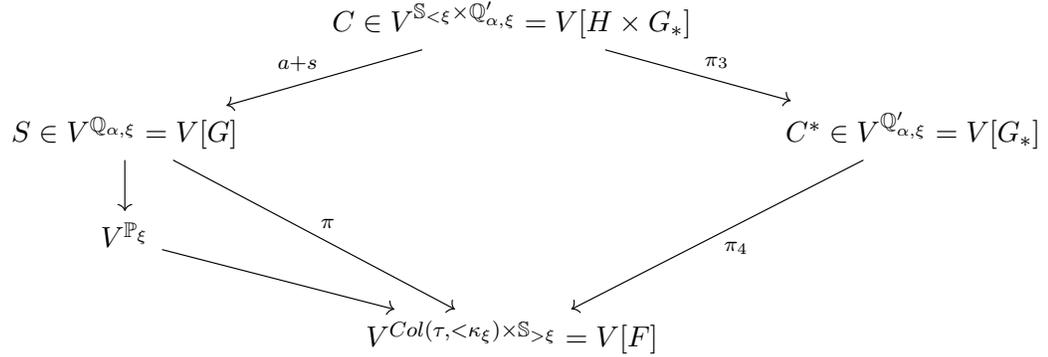
\[\begin{tikzcd}
	& {C\in V^{\mathbb{S}_{<\xi}\times\mathbb{Q}'_{\alpha,\xi}}=V[H\times G_*]} \\
	{S\in V^{\mathbb{Q}_{\alpha,\xi}}=V[G]} && {C^*\in V^{\mathbb{Q}'_{\alpha,\xi}}=V[G_*]} \\
	{V^{\mathbb{P}_\xi}} \\
	& {V^{Col(\tau,<\kappa_\xi)\times\mathbb{S}_{>\xi}}=V[F]}
	\arrow["{a+s}"', from=1-2, to=2-1]
	\arrow["\pi_3", from=1-2, to=2-3]
	\arrow[from=2-1, to=3-1]
	\arrow["\pi_4", from=2-3, to=4-2]
	\arrow[from=3-1, to=4-2]
    \arrow["\pi", from=2-1, to=4-2]
\end{tikzcd}\]

In $V[G_*]$, we consider the finite support product of $\omega$-many copies of $\mathbb{S}_{<\xi}$ and let $H^\omega$ be $V[G_*]$-generic for this forcing. Then $H^\omega$ induces $\l H_i\mid i<\omega\r$ such that each $H_i$ is $V[G_*]$-generic for $\mathbb{S}_{<\xi
}$ and the $H_i$'s are mutually generic.
\begin{fact}\label{fact: union of generics is the poset}
    $\mathbb{S}_{<\xi}=\bigcup_{i<\omega}H_i$
\end{fact}
\begin{proof}
    Follows from a simple density argument.
\end{proof}
\begin{claim}
    Let $G_i=\pi_{\mathbb{S}_{<\xi},*}(H_i\times G_*)$, then $\prod_{i<\omega}^{\text{fin}}G_i$ is $\prod^{\text{fin}}_{i<\omega}(\mathbb{Q}_{\alpha,\xi}/F)$ generic.
\end{claim}
\begin{proof}
    Let $D\subseteq \prod^{\text{fin}}_{i<\omega}(\mathbb{Q}_{\alpha,\xi}/F)$ be a dense open set in $V[F]$,. In $V[G_*]$,  let $$D^*=\{\vec{s}\in \prod^{\text{fin}}_{i<\omega}\mathbb{S}_{<\xi}\mid \exists a\in G_*, \ \vec{s}+a\in D\}$$
    where $\vec{s}+a$ is defined by adding $a$ to each $s$ in the support of $\vec{s}$. By the commutativity of the diagram, $\vec{s}+a\in \prod^{\text{fin}}_{i<\omega}(\mathbb{Q}_{\alpha,\xi}/F)$ whenever $a\in G^*$ and $\vec{s}\in \prod^{\text{fin}}_{i<\omega}\mathbb{S}_{<\xi}$. To see that $D^*$ is dense in $\prod^{\text{fin}}_{i<\omega}\mathbb{S}_{<\xi}$, consider any $\vec{s}$, and let us proceed by a density argument in $\mathbb{Q}'_{\alpha,\xi}/F$. Let $a$ be any element such that $\pi_4(a)\in F$. Then $\vec{s}+a\in\prod^{\text{fin}}_{i<\omega}\mathbb{Q}_{\alpha,\xi}/F$. Since $D$ is dense, there is $\vec{s}\geq\vec{t}=\l t_1,...,t_n\r\in D$. We can assume without loss of generality that for each $i,j$, $\pi_4(a)\geq\pi(t_i)=\pi(t_j)=f\in F$. It follows that for some $a'\leq a$, $\pi_4(a')=f$ (again without loss of generality). Then $a'$ has the property that for  $\vec{s}^*=\l t_1\restriction \mathbb{S}_{<\xi},...,t_n\restriction \mathbb{S}_{<\xi}\r\leq \vec{s}$, $a'+\vec{s}^*\in D$. Note that this dense set depends on $\vec{s}$ and $F$ and is therefore definable in $V[F]$. We conclude that there is some $a_*\in G_*$ for which there is $\vec{s}^*\leq \vec{s}$ such that $a_*+\vec{s}^*\in D$. In particular, $\vec{s}^*$ is an extension of $\vec{s}$ in $D^*$. We conclude that $D^*$ dense. By genericity there is some $\vec{s}\in \prod_{i<\omega}^{\text{fin}}H_i$ such that for some $a^*\in G_*$, $\vec{s}+a^*\in D$. But also $\vec{s}+a^*\in \prod^{\text{fin}}_{i<\omega}G_i$, as wanted.  
\end{proof}
Let $S_i$ be $(\dot{S})_{G_i}$ and for each $q\in\mathbb{S}_{<\xi}$ in $V[G_*]$ we let $$S^*_q=\{\gamma\mid \exists a\in G_*, \ a+q\Vdash \gamma\in \dot{S}\}.$$
\begin{claim}
    $\bigcup_{i<\omega}S_i=\bigcup_qS^*_{q\in\mathbb{S}_{<\xi}}$
\end{claim}
\begin{proof}
Let $\gamma\in \bigcup_{i<\omega}S_i$, then there is $i<\omega$ and $a\in G_i$ such that $a\Vdash_{
\mathbb{Q}_{\alpha,\xi}} \gamma\in \dot{S}$, then $a=\pi_*((s,b))$ for some $s\in H_i$ and $b\in G_*$, which implies that $\gamma\in S^*_s$. In the other direction, let $\gamma\in \bigcup_qS^*_{q\in\mathbb{S}_{<\xi}}$, then there is $q\in \mathbb{S}_{<\xi}$ such that for some $a\in G_*$, $a+q\Vdash \gamma\in\dot{S}$. By Fact \ref{fact: union of generics is the poset}, there is $i<\omega$ such that $q\in H_i$ and therefore $(a,q)\in H_i\times G_*$. It follows that $a+q\in G_i$ and thus $\gamma\in (\dot{S})_{G_i}=S_i$.
    
\end{proof}
Let $T=\bigcup_{i<\omega}S_i=\bigcup_{q\in\mathbb{S}_{<\omega}}S^*_q$, then $T\in V[\prod^{\text{fin}}_{i<\omega}G_i]\cap V[G_*]$. 
\begin{claim}
    $T\cap C^*=\emptyset$
\end{claim}
\begin{proof}
    Suppose not, and let $\gamma\in T\cap C^*$. Then there is $a_*\in G_*$ such that $a_*\Vdash \gamma\in \dot{C}^*$. We may assume that there is $q\in\mathbb{S}_{<\xi}$ such that $a_*+q\Vdash \gamma\in \dot{S}$. But then $(q,a_*)\in H_i\times G_*$ for some $i$ and $(q,a_*)\Vdash \gamma\in \dot{S}\cap \dot{C}$, contradiction.
\end{proof}
So $T\in V[G]$ is a stationary set (as any of the $S_i$'s is) which is not stationary in $V[G_*]$. On the other hand, we will show now that $T\in V[F]$ is stationary (stationarity is downwards absolute) and the extension from $V[F]$ to $V[G_*]$ preserves stationery sets. This will lead us to the desired contradiction. 
\begin{claim}
    $T\in V[F]$.
\end{claim}
\begin{proof}
    In $V[F]$ we have $\sigma$ and $\tau$, a $\prod_{i<\omega}^{\text{fin}}\mathbb{Q}_{\alpha,\xi}/F$-name and a $\mathbb{Q}'_{\alpha,\xi}/F$-name respectively, such that $(\sigma)_{\prod^{\text{fin}}_{i<\omega}G_i}=(\tau)_{G_*}=T$. Let $(\vec{s},a)\in \prod_{i<\omega}^{\text{fin}}H_i\times G_*$ be a condition which forces $\sigma=\tau$ (via the natural interpretation of $\sigma,\tau$ as $(\prod^{\text{fin}}_{i<\omega}\mathbb{S}_{<\omega})\times \mathbb{Q}'_{\xi,\alpha}$-names). We claim that (for example) $a$ already decides the statements ``$\check{\gamma}\in \tau$" for every ordinal $\gamma$. Otherwise, pick any extension $\vec{s}'\leq\vec{s}$ which decides the statement ``$\check{\gamma}\in \sigma$", and suppose without loss of generality that $\vec{s}'\Vdash \check{\gamma}\in \sigma$. Since $a$ does not decide the statement ``$\check{\gamma}\in \tau$", there is $a'\leq a$ such that $a'\Vdash \check{\gamma}\notin \tau$. But then $(\vec{s}',a')\leq (\vec{s},a)$ forces that $\tau\neq\gamma$, contradiction. 
\end{proof}
\begin{claim}
    The extension from $V[F]$ to $V[G_*]$ preserves the stationarity of $T$.
\end{claim}
\begin{proof}
    Let us prove that the quotient $\mathbb{Q}'_{\alpha,\xi}/F$ is $\tau$-closed, and since the stationary set $T$ is an approachable stationary set which consists of point of cofinality $\bar{\kappa}_\xi$, and $\bar{\kappa}_\xi<\tau$, by Shelah, the stationarity of $T$ will be preserved when passing from $V[F]$ to $V[G_*]$. Let $\l a_i\mid i<\sigma\r$ be a decreasing sequence of conditions in $\mathbb{Q}'_{\alpha,\xi}/F$, where $\sigma<\tau$. Let us find a lower bound for that sequence. Similar to previous argument, we can find $b\in F$ such that $b\leq \pi_4(a_i)$ for every $i<\sigma$. Define $(a^*)_0$ be defined by $(a^*)_0\restriction \mathbb{S}_{<\xi}=(a_0)_0$ and $(a^*)_0\restriction Col(\tau,<\kappa_\xi)\times\mathbb{S}_{>\xi}=b$. Note that by definition of the order on $\mathbb{Q}'_{\alpha,\xi}$ and by definition of the projection $\pi_4$, $(a^*)_0\leq^* (a_i)_0$ for all $i<\sigma$ (and in fact the part below $\xi$ is fixed). Now we define the strategies as in Lemma \ref{Lemma: First stationary pres} to obtain the condition $a^*$ such that $a^*\leq_{\mathbb{Q}_{\alpha,\xi}} a_i$ for all $i<\sigma$ and $\pi_4(a^*)=b\in F$.
\end{proof}
This concludes the proof that stationary sets are preserved from $V^{\mathbb{Q}_{\alpha,\xi}}$ to $V^{\mathbb{S}_{<\xi}\times\mathbb{Q}_{\alpha,\xi}'}$, and therefore conclude the Theorem.
    
\end{proof}
\end{proof}
\begin{corollary}
    $V^{\mathbb{Q}_{\alpha,\xi}}\models \text{Refl}(E^{\kappa^+}_{\leq\bar{\kappa}_\xi},E^{\kappa^+}_{<\kappa_\xi})$
\end{corollary}
\begin{proof}
    This is a straightforward application of Lemma \ref{Lemma: Reflectionholds in the product} and Theorem \ref{Thm: stationary sets preservation passing to product}.
\end{proof}

Finally, let us show that in $V^{\mathbb{Q}_{\alpha}}$, non-reflecting stationary sets are fragile.
\begin{theorem}\label{non reflecting is fragile}
    Suppose that $p\Vdash_{\mathbb{Q}_\alpha}\dot{S}$ is a non-reflecting stationary set then $\dot{S}$ is fragile.
\end{theorem}
\begin{proof}
Suppose that $\dot{S}$ is a name for a non-fragile stationary subset of $\kappa^+$. By shrinking if necessary, we may assume that it concentrates on some fixed cofinality below $\kappa$. By the non-fragility, there is some $\xi<\omega_1$ with $\max(\supp(p_0))=\xi$ and $p\in\mathbb{Q}_{\alpha,\xi}$ such that $p\Vdash_{\mathbb{Q}_{\alpha,\xi}}\dot{S}_{\xi}$ is stationary. By Theorem \ref{Thm: stationary sets preservation passing to product}, for every $(s,a)\in \mathbb{S}_{\xi}\times \mathbb{Q}^+_{\alpha,\xi} $, such that $a+s\leq p$, $(a+s)\Vdash_{\mathbb{S}_{\xi}\times \mathbb{Q}^+_{\alpha,\xi}} \dot{S}_{\xi}$ is stationary.
Increasing $\xi$ if necessary, we may assume that the points in $S$ have cofinality below $\bar{\kappa}_\xi$.
By the above lemma, $p\Vdash_{\mathbb{Q}_{\alpha,\xi}}\dot{S}_{\xi}$ reflects at a point $\gamma$ of cofinality less than $\bar{\kappa}_{\xi}$.  

As above, $\mathbb{Q}_{\alpha,\xi}$ is a projection of $\mathbb{S}_{<\xi}\times Col(\tau,<\kappa_\xi)\times\mathbb{Q}^+_{\alpha,\xi}$, where $\tau=s_\alpha(f_\xi(\kappa_\xi))^{+3}$. Recall that $\mathbb{S}_{<\xi}$ has size less than $\tau$ and $\mathbb{Q}^+_{\alpha,\xi}$ is $\kappa_\xi^+$-closed. 
Let $G_\xi$ be  $\mathbb{Q}_{\alpha,\xi}$-generic and let $H$ be the induced $\mathbb{S}_{<\xi}$-generic. 
Note that $\mathbb{Q}_{\alpha,\xi}^u:=Col(\tau,<\kappa_\alpha)\times\mathbb{Q}^+_{\alpha,\xi}$ is $\tau$-closed in $V$ and by Easton's Lemma, $<\tau$-distributive over $V[H]$.

In $V[G_\xi]$, let $S^*=(\dot{S}_\xi)_{G_\xi}$ and let $A\subset S^*\cap\gamma$ be stationary in $\gamma$ of order type $\cf(\gamma)$. Since $\cf(\gamma)<\bar{\kappa}_{\xi}<\tau$, by its distributivity, $\mathbb{Q}_{\alpha,\xi}^u$ cannot have added $A$, so $A\in V[H]$.  

Next,
we construct a condition $r^*\in\mathbb{Q}_{\alpha,\xi}$, such that $r^*\Vdash_{\mathbb{Q}_{\alpha,\xi}/H} \dot{A}\subset \dot{S}_\xi$. We do this as follows.

Enumerate $A=\{\beta_i\mid i<\cf(\gamma)\}$, and suppose that this is forced by $1_{\mathbb{S}_{<\xi}}$ and that $1_{\mathbb{Q}_{\alpha,\xi}}\Vdash \dot{A}\subseteq \dot{S}_\xi$. Working in $V$, for all $(s,r)\in\mathbb{S}_{<\xi}\times\mathbb{Q}_{\alpha,\xi}^u$ and $i<\cf(\gamma)$, if for some $\beta<\gamma$, $s\Vdash \beta=\dot{\beta}_i$, then $(s,r)\Vdash \beta\in \dot{S}_\xi$ and therefore there is $(s',r')\leq (s,r)$ such that $\l \beta,(s',r')\r\in\dot{S}_\xi$ which in turn implies (by the definition of $\dot{S}_\xi$) that $(s',r')\Vdash_{ \mathbb{Q}_\alpha} \beta\in \dot{S}$. Using that the closure of $\mathbb{Q}_{\alpha,\xi}^u$ is greater than $\max(\cf(\gamma), |\mathbb{S}_{<\xi}|)$,  we can do this inductively for all $i<\cf(\gamma)$, and all conditions in $\mathbb{S}_{<\xi}$ to construct a condition  $r^*\in\mathbb{Q}_{\alpha,\xi}$, such that $r^*\Vdash_{\mathbb{Q}_{\alpha}/H} \dot{A}\subset \dot{S}$. 

Finally, since $\mathbb{Q}_\alpha/H$ does not add new subsets of $\cf(\gamma)$,  it preserves the stationarity of $A$. This proves that $r^*\Vdash_{\mathbb{Q}_\alpha} \dot{S}$ reflects.

\end{proof}

\section{The end game}\label{Section: end game}
We are ready to prove the main theorem of this paper:
\begin{theorem}
    Assume GCH and suppose that there are uncountable many supercompact cardinals. Then it is consistent that 
    $2^{\aleph_{\omega_1}}>\aleph_{\omega_1+1}$ and reflection holds at $\aleph_{\omega_1+1}$.
\end{theorem}
\begin{proof}
    We start by making the uncountable many supercompacts indestructible. Fix a surjection $\phi:\kappa^{++}\rightarrow H(\kappa^{++})$ such that for every $h\in H(\kappa^{++})$ $\phi^{-1}(h)$ is unbounded in $\kappa^{++}$. We define the iteration $\mathbb{Q}_\alpha$ as before, such that at stage $\alpha$, if $\phi(\alpha)$ happens to be a $\mathbb{Q}_\alpha$-name for a fragile stationery set, we use $\mathbb{A}(\mathbb{Q}_\alpha,h(\alpha))$, or otherwise, we just take the trivial forcing at successor steps. In the extension by $\mathbb{Q}_{\kappa^{++}}$, $2^\kappa=2^{\aleph_{\omega_1}}>\aleph_{\omega_1+1}=\kappa^+$ by Corolarry \ref{Cor: Failure of SCH}. To see that stationary reflection holds at $\aleph_{\omega_1+1}$, suppose otherwise and let $S\subseteq \kappa^+$ be a name for a non-reflecting stationary set. Then, there is a name $\dot{S}\in H(\kappa^{++})$ for $S$ such that $p\Vdash\dot{S}$ is not reflecting. By the small support and the $\kappa^{++}$-c.c, there is $\alpha<\kappa^{++}$ such that $\dot{S}$ is a $\mathbb{Q}_\alpha$-name, and $p\in \mathbb{Q}_\alpha$. By Theorem \ref{non reflecting is fragile}, $\dot{S}$ is fragile. Let $\beta\in \phi^{-1}(h)$ be above $\alpha$, then at stage $\beta$ of the iteration, we are facing a $\mathbb{Q}_\beta$-name (clearly every $\mathbb{Q}_\alpha$ can be identified with a $\mathbb{Q}_\beta$-name) for a fragile stationary set, so by Corollary \ref{Cor: killing stationary sets}, in $V^{\mathbb{Q}_{\beta+1}}$, $\dot{S}$ is non-stationary. By absoluteness, $\dot{S}$ is non-stationary in $V^{\mathbb{Q}_{\kappa^{++}}}$, contradiction. 
\end{proof}
 \bibliographystyle{amsplain}
\bibliography{ref}
\end{document}